\documentclass[12pt]{amsart}
\usepackage{amssymb}
\numberwithin{equation}{section}
\theoremstyle{plain}
\newtheorem{theorem}{Theorem}[section]
\newtheorem{lemma}[theorem]{Lemma}
\newtheorem{proposition}[theorem]{Proposition}
\newtheorem{corollary}[theorem]{Corollary}
\newtheorem{example}[theorem]{Example}
\theoremstyle{definition}
\newtheorem{definition}[theorem]{Definition}

\theoremstyle{remark}
\newtheorem{remark}[theorem]{Remark}
\allowdisplaybreaks
\newcommand{\dt}{\left.\frac{d}{dt}\right|_{t=0}}
\newcommand{\mult}{m}
\newcommand{\con}{k}
\newcommand{\GiE}{K}
\newcommand{\giE}{\mathfrak{k}}
\newcommand{\iso}{F}
\newcommand{\one}{1}

\newcommand{\unit}{e}
\newcommand{\pone}{q}
\newcommand{\ptwo}{r}

\title[Holomorphic multiplier representations]{holomorphic multiplier representations \\for bounded homogeneous domains}
\author{koichi arashi}
\address{K. Arashi: Graduate School of Mathematics, Nagoya University, Chikusa-ku, Nagoya, 464-8602 Japan}
\email{m15005y@math.nagoya-u.ac.jp}
\begin{document}
\maketitle
\begin{abstract}
In this paper, we study the unitarizations in the spaces of holomorphic sections of equivariant holomorphic line bundles over a bounded homogeneous domain under the action of a connected algebraic group acting transitively on the domain. We give a complete classification of unitary representations arising from such unitarizations. As an application, we classify all such unitary representations for a specific five-dimensional non-symmetric bounded homogeneous domain.
\end{abstract}

\section{Introduction}
Unitary representations realized in the spaces of the holomorphic sections of equivariant holomorphic line bundles appear in various areas of the representation theory of Lie groups. For instance, we can recall the Borel-Weil theory for compact Lie groups, the holomorphic discrete series and its analytic continuation for Hermite Lie groups, the Bargmann-Fock representation for the Heisenberg group, and the Auslander-Kostant theory for solvable Lie groups. We shall formulate such unitary representations as follows. Let $\mathcal{M}$ be a connected complex manifold, let $\mathrm{Aut}_{hol}(\mathcal{M})$ be the holomorphic automorphism group of $\mathcal{M}$, let $G_0\subset\mathrm{Aut}_{hol}(\mathcal{M})$ be a connected subgroup which acts on $\mathcal{M}$ transitively, and let $L$ be a $G_0$-equivariant holomorphic line bundle over $\mathcal{M}$. We denote by $\Gamma^{hol}(L)$ the space of holomorphic sections of $L$. Let $l$ be the representation of $G_0$ given by 
\begin{equation*}
 l(g)s(z)=gs(g^{-1}z)\quad(g\in G_0,s\in\Gamma^{hol}(L),z\in\mathcal{M}).
\end{equation*}
Let us consider all $G_0$-equivariant holomorphic line bundles $L$ over $\mathcal{M}$ and the following fundamental questions:
\begin{enumerate}
\item[(Q1)] What is the condition that the representation $l$ of $G_0$ is unitarizable?
\item[(Q2)] Which unitarizations are equivalent as unitary representations of $G_0$?
\end{enumerate}
Here we make precise the class of representations we study.
\begin{definition}
We say that the representation $l$ is {\it unitarizable} if there exists a nonzero Hilbert space $\mathcal{H}\subset\Gamma^{hol}(L)$ satisfying the following conditions:
\begin{enumerate}
\item[(i)]the inclusion map $\iota:\mathcal{H}\hookrightarrow \Gamma^{hol}(L)$ is continuous with respect to the open compact topology of $\Gamma^{hol}(L)$,
\item[(ii)] $l(g)\mathcal{H}\subset \mathcal{H}\,(g\in G_0)$ and $\|l(g)s\|=\|s\|\,(g\in G_0, s\in\mathcal{H})$, where $\|\cdot\|$ denotes the norm of $\mathcal{H}$. 
\end{enumerate}
This notion is closely related to the holomorphic induction introduced by Auslander and Kostant. We will mention the relation later. For a unitarizable representation $l$, we call the subrepresentation $(l,\mathcal{H})$ a {\it unitarization} of the representation $(l,\Gamma^{hol}(L))$ of $G_0$. 
\end{definition}
A Hilbert space $\mathcal{H}$ satisfying the condition (i) is a reproducing kernel Hilbert space. The following theorem is known.
\begin{theorem}[{\cite[Theorem 6]{ishi 2011}}, {\cite{kobayashi}}, {\cite{kunze}}]\label{uniqueness of unitarification1}
A Hilbert space giving a unitarization of $l$ is unique if it exists. In particular, the unitarization is irreducible.
\end{theorem}
In this paper, we shall give a complete answer to the questions (Q1) and (Q2) in the case that $\mathcal{M}$ is a bounded homogeneous domain $\mathcal{D}$ and $G_0\subset\mathrm{Aut}_{hol}(\mathcal{D})$ is the identity component $G$ of a real algebraic group. Here it is known \cite[Theorem 3.2]{kaneyuki} that $\mathrm{Aut}_{hol}(\mathcal{D})$ admits a structure of a Lie group and its identity component is isomorphic to the identity component of a linear algebraic group. The identity component of $\mathrm{Aut}_{hol}(\mathcal{D})$, which is denoted by $\mathrm{Aut}_{hol}(\mathcal{D})^o$ is an example of $G$. When $\mathcal{D}$ is symmetric, any parabolic subgroup of $\mathrm{Aut}_{hol}(\mathcal{D})^o$ is also an example of $G$. Now we introduce a notion of Iwasawa subgroup of a Lie group.
\begin{definition}
For a Lie group $G_0$, we call a subgroup $B_0\subset G_0$ an {\it Iwasawa} subgroup of $G_0$ if $B_0$ is a maximal connected real split solvable Lie subgroup of $G_0$.
\end{definition}
It is known that the isotropy subgroup of $\mathrm{Aut}_{hol}(\mathcal{D})^o$ at a point $p\in\mathcal{D}$ is a maximal compact subgroup of $\mathrm{Aut}_{hol}(\mathcal{D})^o$, and in our setting, it follows that an Iwasawa subgroup $B$ of $G$ acts on $\mathcal{D}$ simply transitively (see \cite[Chapter 4, Theorem 4.7]{encyclopedia}). 
\begin{definition}
An analytic function $\mult:G\times \mathcal{D}\rightarrow \mathbb{C}^\times$ is called a {\it multiplier} if the following cocycle condition is satisfied:
\begin{equation*}
\mult(gg',z)=\mult(g,g'z)\mult(g',z)\quad(g,g'\in G, z\in \mathcal{D}).
\end{equation*}
Moreover, a multiplier $\mult$ is called a holomorphic multiplier if $\mult(g,z)$ is holomorphic in $z\in \mathcal{D}$. 
\end{definition}
Let $m:G\times\mathcal{D}\rightarrow\mathbb{C}^\times$ be a holomorphic multiplier. Let $E_m$ be the $G$-equivariant trivial line bundle $\mathcal{D}\times\mathbb{C}$, where the $G$-action on $E_m$ is defined by 
\begin{equation*}
g(z,\zeta)=(gz,m(g,z)\zeta).
\end{equation*}
Since a bounded homogeneous domain is a contractible Stein manifold, every holomorphic line bundle over $\mathcal{D}$ is trivial. Thus there exists a holomorphic multiplier $m:G\times\mathcal{D}\rightarrow\mathbb{C}^\times$ such that $L$ and $E_m$ are isomorphic as $G$-equivariant holomorphic line bundles. 
Let $\mathcal{O}(\mathcal{D})$ denote the space of holomorphic functions on $\mathcal{D}$. We identify $\Gamma^{hol}(E_m)$ with $\mathcal{O}(\mathcal{D})$, and let us denote $T_m$ the representation $l$ for $E_m$. The representation $T_m$ of $G$ is described as
\begin{equation*}
 T_m(g)f(z)=m(g^{-1},z)^{-1}f(g^{-1}z)\quad(f\in\mathcal{O}(\mathcal{D})).
\end{equation*}

The scalar-valued holomorphic discrete series and its analytic continuation is a special case of our object. In this case $\mathcal{D}$ is a bounded symmetric domain and $G$ is a semisimple Lie group which is locally isomorphic to the group $\mathrm{Aut}_{hol}(\mathcal{D})$. Let $\gamma$ be a complex number, let $\mathcal{D}$ be an irreducible bounded symmetric domain, and let $J:\mathrm{Aut}_{hol}(\mathcal{D})^o\times\mathcal{D}\rightarrow\mathbb{C}^\times$ denote the complex Jacobian. Consider the following representation of $\mathrm{Aut}_{hol}(\mathcal{D})^o$ on the space $\mathcal{O}(\mathcal{D})$:
\begin{equation*}
T_{J^{-\gamma}}(g)f(z)=J(g^{-1},z)^\gamma f(g^{-1}z)\quad(g\in\mathrm{Aut}_{hol}(\mathcal{D})^o,f\in\mathcal{O}(\mathcal{D})).
\end{equation*}
To be precise, we should consider $J(g^{-1},z)^\gamma$ as a function defined on $\widetilde{\mathrm{Aut}_{hol}(\mathcal{D})^o}\times\mathcal{D}$, where $\widetilde{\mathrm{Aut}_{hol}(\mathcal{D})^o}$ denotes the universal covering group of $\mathrm{Aut}_{hol}(\mathcal{D})^o$. The unitarizations of the above representations $ T_{J^{-\gamma}}$ are highest weight unitary representations, and the equivalence classes of these unitary representations are determined by their highest weights $\gamma$. On the other hand, not all $ T_{J^{-\gamma}}$ are unitarizable. First we consider the condition that $ T_{J^{-\gamma}}$ has a nontrivial $\widetilde{\mathrm{Aut}_{hol}(\mathcal{D})^o}$-invariant subspace which is given as a weighted Bergman space. The condition is a special case of the Harish-Chandra condition \cite{harishV, harishVI}. More generally, the set of $\gamma$ for which $ T_{J^{-\gamma}}$ is unitarizable is called the Wallach set of $\mathcal{D}$, and is determined by Vergne and Rossi \cite{vergne} and Wallach \cite{wallach}. 

We can consider the same kind of representations for bounded homogeneous domains. Let $\mathcal{D}$ be a (not necessarily symmetric) bounded homogeneous domain. Ishi shows the following theorem.
\begin{theorem}[Ishi, {\cite[Proposition 14]{ishi 2011}}]
Let $\mathcal{H}\subset \mathcal{O}(\mathcal{D})$ be a reproducing kernel Hilbert space. Suppose that $ T_{J^{-\gamma}}(b)\mathcal{H}\subset \mathcal{H}$ for all $b\in B$, and $\| T_{J^{-\gamma}}(b)f\|=\|f\|$ for all $b\in B$ and $f\in\mathcal{O}(\mathcal{D})$. Then we have $ T_{J^{-\gamma}}(g)\mathcal{H}\subset \mathcal{H}$ for all $g\in \widetilde{\mathrm{Aut}_{hol}(\mathcal{D})}$, and $\| T_{J^{-\gamma}}(g)f\|=\|f\|$ for all $g\in \widetilde{\mathrm{Aut}_{hol}(\mathcal{D})}$ and $f\in\mathcal{O}(\mathcal{D})$.
\end{theorem}
\begin{theorem}[Ishi, {\cite{ishi 2013}}]
Unitarizations of $ T_{J^{-\gamma}}$ and $ T_{J^{-\gamma'}}$ are equivalent as unitary representations of $\widetilde{\mathrm{Aut}_{hol}(\mathcal{D})^o}$ if and only if $\gamma=\gamma'$.
\end{theorem}
When $\mathcal{D}$ is an irreducible bounded symmetric domain, every $\mathrm{Aut}_{hol}(\mathcal{D})^o$-equivariant holomorphic line bundle over $\mathcal{D}$ is isomorphic to $E_{J^{-\gamma}}$ for some $\gamma\in\mathbb{C}$. On the other hand, for $G\subsetneq \mathrm{Aut}_{hol}(\mathcal{D})^o$, it can happen that there exists a $G$-equivariant holomorphic line bundle $L$ such that $L$ is not isomorphic to $L_{J^{-\gamma}}$ for any $\gamma\in\mathbb{C}$ as a $G$-equivariant holomorphic line bundle. Moreover, when $\mathcal{D}$ is not symmetric, the same can happen even for $G=\mathrm{Aut}_{hol}(\mathcal{D})^o$ (see Section \ref{Applicatio}). 

As we will see in Section \ref{EXISTENCEO}, we can reduce the question (Q1) for $G$ to the question for $B$. Let $\mathcal{D}$ be a bounded homogeneous domain, and let $L$ be a $G$-equivariant holomorphic line bundle over $\mathcal{D}$.
\begin{theorem}[see Theorem \ref{LetmGtimes}]\label{Letmathcal}
Let $\mathcal{H}\subset \Gamma^{hol}(L)$ be a reproducing kernel Hilbert space. Suppose that $l(b)\mathcal{H}\subset \mathcal{H}$ for all $b\in B$ and $\|l(b)s\|=\|s\|$ for all $b\in B$ and $s\in\Gamma^{hol}(L)$. Then we have $l(g)\mathcal{H}\subset \mathcal{H}$ for all $g\in G$ and $\|l(g)s\|=\|s\|$ for all $g\in G$ and $s\in\Gamma^{hol}(L)$. Namely, the unitarizability as the representation of $B$ implies the one as the representation of $G$.
\end{theorem}
We fix a reference point $p\in\mathcal{D}$. Let $L_p$ be the fiber over the point $p$, and let $K$ be the isotropy subgroup of $G$ at $p$. Note that $K$ is a maximal compact connected subgroup of $G$. Concerning the question (Q2), we obtain
\begin{theorem}[see Theorem \ref{main}]\label{LetLkappak1}
Let $L$ and $L'$ be $G$-equivariant holomorphic line bundles over $\mathcal{D}$. Suppose that $\mathcal{H}\subset \Gamma^{hol}(L)$ and $\mathcal{H}'\subset\Gamma^{hol}(L')$ give unitarizations of representations $l$ and $l'$, respectively. Then $(l,\mathcal{H})$ and $(l',\mathcal{H}')$ are equivalent as unitary representations of $G$ if and only if $(l|_B,\mathcal{H})$ and $(l'|_B,\mathcal{H}')$ are equivalent as unitary representations of $B$ and the actions of $K$ on the fibers $L_p$ and $L'_p$ coincide.
\end{theorem}
Now we give a concrete parametrization of the $G$-equivariant holomorphic line bundles $L$ for which the representations $l$ are unitarizable, and we shall give the partition of the parameter set $\Theta(G)$ which corresponds to the equivalence classes of the unitarizations. In other words, the partition gives an answer to the question (Q2), and describes the classification of the unitary representations of $G$ obtained by unitarizations. 

Let $\mathfrak{g}=\mathrm{Lie}(G)$, and let $\mathfrak{g}_-\subset \mathfrak{g}_\mathbb{C}$ be the complex subalgebra defined by
\begin{equation*}
\mathfrak{g}_-=\left\{Z=X+iY\in \mathfrak{g}_\mathbb{C}; \left.\frac{d}{dt}\right|_{t=0}e^{tX}p+i\left.\frac{d}{dt}\right|_{t=0}e^{tY}p\in T_{p}^{0,1}\mathcal{D}\right\}.
\end{equation*}
Let $\mathfrak{k}=\mathrm{Lie}(K)$. Clearly $\mathfrak{k}\subset \mathfrak{g}_-$. By Tirao and Wolf {\cite[Theorem 3.6]{tirao}}, the set of equivalence classes of $G$-equivariant holomorphic line bundles over $\mathcal{D}$ can be identified with the set 
\begin{equation*}
\mathcal{L}(G)=\left\{\theta\in\mathfrak{g}_-^*;\begin{array}{c}\theta \text{ is a complex one-dimensional representation of }\mathfrak{g}_-\\\text{ such that } \theta|_\mathfrak{k} \text{ lifts to a representation of }K\end{array} \right\}. 
\end{equation*}
Let $L$ be a $G$-equivariant holomorphic line bundle corresponding to $\theta\in\mathcal{L}(G)$. There exists a $G$-invariant Hermitian metric on $L$ and we consider the space $\Gamma^2(L)$ of square integrable holomorphic sections of $L$. If $\Gamma^2(L)\neq\{0\}$, then $\Gamma^2(L)$ gives the unitarization, and $(l,\Gamma^2(L))$ is nothing else but the holomorphically induced representation in \cite{Auslander} from a ``polarization" $\mathfrak{g}$ at $\xi\in\mathfrak{g}^*$, where $i\xi|_{\mathfrak{g}_-}=\theta$ (see p. 11). We note that even though $\Gamma^2(L)=\{0\}$, there may exist $\mathcal{H}\neq\{0\}$ giving a unitarization of $l$. Let $\mathfrak{b}=\mathrm{Lie}(B)$. We identify $T_p\mathcal{D}$ with $\mathfrak{b}$. Then a $B$-invariant K\"{a}hler metric on $\mathcal{D}$ defines a normal $j$-algebra $(\mathfrak{b},j,\omega)$ (see \cite[Part III, Lemma 1]{CIME}). Let $\mathfrak{a}$ denote the orthogonal complement of $[\mathfrak{b},\mathfrak{b}]$ in $\mathfrak{b}$ with respect to the inner product $\langle\cdot,\cdot\rangle=\omega([j\cdot,\cdot])$ on $\mathfrak{b}$. Put $\mathfrak{a}_-=\mathfrak{g}_-\cap\mathfrak{a}_\mathbb{C}$. The set of equivalence classes of $B$-equivariant holomorphic line bundles over $\mathcal{D}$ is parametrized by $\mathfrak{a}_-^*$. Let $r=\dim\mathfrak{a}$. For $\varepsilon=(\varepsilon_1,\cdots,\varepsilon_r)\in\{0,1\}^r$, we put $Z(\varepsilon)=\{\underline{\zeta}=(\zeta_1,\cdots, \zeta_r)\in\mathbb{R}^r; \zeta_k=0 \text{ for all }k \text{ such that }\varepsilon_k=1 \}$. Ishi \cite{ishi 1999} gives the subset $\Theta$ of $\mathfrak{a}_-^*$ and the partition 
\begin{equation*}
\Theta=\bigsqcup_{\varepsilon\in\{0,1\}^r}\bigsqcup_{\underline{\zeta}\in Z(\varepsilon)}\Theta(\varepsilon,\underline{\zeta}) 
\end{equation*}
such that a representation $l$ of $B$ is unitarizable if and only if the corresponding parameter belongs to $\Theta$ and unitarizations of $l$ and $l'$ are equivalent if and only if the corresponding parameters belong to the same $\Theta(\varepsilon,\underline{\zeta})$. Combining Theorem \ref{Letmathcal} and Theorem \ref{LetLkappak1} with the results of \cite{ishi 1999, ishi 2011}, we obtain a method of giving a concrete parametrization in question. Let 
\begin{equation*}
\Lambda=\{\lambda\in \mathfrak{z}(\mathfrak{k})^*;i\lambda=d\chi|_{\mathfrak{z}(\mathfrak{k})}\text{ for some one-dimensional representation } \chi \text{ of } K\}. 
\end{equation*} 
For $\varepsilon\in\{0,1\}^r$, $\underline{\zeta}\in Z(\varepsilon)$, and $\lambda\in \Lambda$, we put
\begin{equation*}
\Theta(G,\varepsilon,\underline{\zeta},\lambda)=\{\theta\in\mathcal{L}(G);\theta|_{\mathfrak{a}_-}\in\Theta(\varepsilon,\underline{\zeta}), \theta|_{\mathfrak{z}(\mathfrak{k})}=i\lambda\}.
\end{equation*}
Set
\begin{equation*}
P=\{(\varepsilon,\underline{\zeta},\lambda)\in\{0,1\}^r\times \mathbb{R}^r\times \Lambda; \underline{\zeta}\in Z(\varepsilon), \Theta(G,\varepsilon,\underline{\zeta},\lambda)\neq \emptyset\}.
\end{equation*}
Then the set 
\begin{equation*}
\Theta(G)=\{\theta\in\mathcal{L}(G);\theta|_{\mathfrak{a}_-}\in\Theta\}
\end{equation*}
and the partition 
\begin{equation*}
\Theta(G)=\bigsqcup_{(\varepsilon,\underline{\zeta},\lambda)\in P} \Theta(G,\varepsilon,\underline{\zeta},\lambda)
\end{equation*}
describe the set of equivalence classes $[L]$ of $G$-equivariant holomorphic line bundles such that the representations $l$ of $G$ are unitarizable and the partition of the set corresponding to the unitary equivalence classes of representations $l$ of $G$. In Section \ref{Applicatio}, we see an example of the set $\Theta(G)$ and the partition 
\begin{equation*}
\Theta(G)=\bigsqcup_{(\varepsilon,\underline{\zeta},\lambda)\in P} \Theta(G,\varepsilon,\underline{\zeta},\lambda) 
\end{equation*}
for a five-dimensional non-symmetric bounded homogeneous domain which is biholomorphic to the Siegel domain
\begin{equation*}
\mathcal{D}(\Omega_1)=\left\{U=\left[\begin{array}{ccc}z_\one&0&z_4\\0&z_2&z_5\\z_4&z_5&z_3\end{array}\right]\in \mathrm{Sym}(3,\mathbb{C});\Im U\gg 0\right\},
\end{equation*}
where $G$ is the identity component of the holomorphic automorphism group of the domain.

As a byproduct of the proof of Theorem \ref{LetLkappak1}, we obtain the following theorem.
\begin{theorem}[see Corollary \ref{LetEandEbe}]\label{LetLkappak}
Let $L$ and $L'$ be $G$-equivariant holomorphic line bundles over $\mathcal{D}$. Suppose that the actions of $K$ on the fibers $L_p$ and $L'_p$ coincide. Then $L$ and $L'$ are isomorphic as $K$-equivariant holomorphic line bundles.
\end{theorem}
Let us explain the organization of this paper. In Section \ref{EXISTENCEO}, we prove Theorem \ref{Letmathcal}. In Section \ref{Normaljalg}, we review the theory of normal $j$-algebras. In Section \ref{Algebraicp}, first we prove Lemma \ref{Letpartial} about a property of the gradation of the Lie algebra $\mathfrak{aut}_{hol}(\mathcal{D})$ of $\mathrm{Aut}_{hol}(\mathcal{D})$ and its bracket relations. After that, Section \ref{Algebraicp} is devoted to the proof of Proposition \ref{Forasubalg2}, that is a generalization of Lemma \ref{Letpartial} in which $\mathfrak{aut}_{hol}(\mathcal{D})$ gets replaced by $\mathfrak{g}$. Proposition \ref{Forasubalg2} plays an important role in the proof of Theorem \ref{Supposethat}, which implies Theorem \ref{LetLkappak} immediately. In Section \ref{Unitaryequ}, we show Theorem \ref{LetLkappak1} using Theorem \ref{Supposethat}. In Section \ref{Applicatio}, we see an example of the set $\Theta(G)$ and the partition \begin{equation*}
\Theta(G)=\bigsqcup_{(\varepsilon,\underline{\zeta},\lambda)\in P} \Theta(G,\varepsilon,\underline{\zeta},\lambda) 
\end{equation*}
for the five-dimensional non-symmetric bounded homogeneous domain mentioned above.

\section{Existence of unitarizations}\label{EXISTENCEO}
Throughout this paper, for a Lie group $G_0$, we denote its Lie algebra by the corresponding Fraktur small letter $\mathfrak{g}_0$.
\subsection{General theory of holomorphic multiplier representations}\label{sec:1.1}
We review the theory of homogeneous holomorphic vector bundles and the theory of holomorphic multiplier representations in \cite{ishi 2011,kobayashi,tirao}.

Let $\mathcal{D}_0$ be a domain in $\mathbb{C}^N$, and let $G_0$ be a Lie group which acts holomorphically on $\mathcal{D}_0$. We assume that the action of $G_0$ on $\mathcal{D}_0$ is analytic, i.e. the map $G_0\times\mathcal{D}_0\ni(g,z)\mapsto gz\in\mathcal{D}_0$ is analytic. Let $\mathcal{V}$ be a finite-dimensional complex vector space.

\begin{definition}
An analytic function $\mult:G_0\times \mathcal{D}_0\rightarrow GL(\mathcal{V})$ is called a {\it multiplier} if the following cocycle condition is satisfied:
\begin{equation*}
\mult(gg',z)=\mult(g,g'z)\mult(g',z)\quad(g,g'\in G_0, z\in \mathcal{D}_0).
\end{equation*}
Moreover, a multiplier $\mult$ is called a holomorphic multiplier if $\mult(g,z)$ is holomorphic in $z\in \mathcal{D}_0$. 
\end{definition}
\begin{remark}
When $\mathcal{V}=\mathbb{C}$, let 
\begin{equation*}
\mathcal{G}=\{\mult:G_0\times\mathcal{D}_0\rightarrow \mathbb{C}^\times;\mult\text{ is a holomorphic multiplier}\}.
\end{equation*}
Pointwise multiplication of holomorphic multipliers gives $\mathcal{G}$ the natural structure of a group. We write the product of two elements $m,m'$ of $\mathcal{G}$ as $mm'$.
\end{remark}

Let $\mult:G_0\times\mathcal{D}_0\rightarrow GL(\mathcal{V})$ be a holomorphic multiplier. Let $T_{\mult}$ be the representation of $G_0$ defined by
\begin{equation*}
T_{\mult}(g)f(z)=\mult(g^{-1},z)^{-1}f(g^{-1}z) \quad(g\in G_0, f\in \mathcal{O}(\mathcal{D}_0,\mathcal{V}),z\in\mathcal{D}_0),
\end{equation*}
where $\mathcal{O}(\mathcal{D}_0,\mathcal{V})$ denotes the space of vector-valued holomorphic functions on $\mathcal{D}_0$. When $\mathcal{V}=\mathbb{C}$, a power of the complex Jacobian $J(g,z)^{-\gamma}\,(g\in G_0, z\in \mathcal{D}_0,\gamma\in\mathbb{Z})$ is an example of a holomorphic multiplier. 
We fix a reference point $p_0\in\mathcal{D}_0$. Let $(\mathfrak{g}_0)_-\subset (\mathfrak{g}_0)_\mathbb{C}$ be the complex subalgebra defined by
\begin{equation}\label{mathfrakg0l}
(\mathfrak{g}_0)_-=\left\{Z=X+iY\in (\mathfrak{g}_0)_\mathbb{C}; \left.\frac{d}{dt}\right|_{t=0}e^{tX}p_0+i\left.\frac{d}{dt}\right|_{t=0}e^{tY}p_0\in T_{p_0}^{0,1}\mathcal{D}_0\right\},
\end{equation}
and let $\theta_m:(\mathfrak{g}_0)_-\rightarrow\mathfrak{gl}(\mathcal{V})$ be the complex linear map given by
\begin{equation*}
\theta_m(Z)=\left.\frac{d}{dt}\right|_{t=0}\mult(e^{tX},p_0)+i\left.\frac{d}{dt}\right|_{t=0}\mult(e^{tY},p_0)\quad(Z=X+iY\in(\mathfrak{g}_0)_-).
\end{equation*}
The smooth map 
\begin{equation*}
F:G_0 \ni g\mapsto m(g,p_0)\in GL(\mathcal{V}) \end{equation*} 
satisfies
\begin{equation*}
(F_*)_g\left(\dt ge^{tX}\right)=\dt m(g,e^{tX}p_0)m(e^{tX},p_0)\quad(g\in G_0,X\in\mathfrak{g}_0).
\end{equation*}
For $X\in\mathfrak{g}_0$, let us use the same symbol $X$ to denote the corresponding left invariant vector field on $G_0$. We extend $(F_*)_g$ to a $\mathbb{C}$-linear map for all $g\in G_0$. At the identity element $e$ of $G_0$, this is a complex-linear map $(F_*)_e:(\mathfrak{g}_0)_\mathbb{C}\rightarrow \mathfrak{gl}(\mathcal{V})$. Then for $Z\in(\mathfrak{g}_0)_-$, we have 
\begin{equation*}
(F_* Z)_{F(g)}=\theta_m(Z)_{F(g)}\quad(g\in G_0).
\end{equation*}
Thus for $Z,Z'\in(\mathfrak{g}_0)_-$, we have
\begin{equation*}\begin{split}
\theta_m([Z,Z'])&=\theta_m([Z,Z'])_e=(F_*)_e[Z,Z']_e
=[\theta_m(Z),\theta_m(Z')]_e
\\&=[\theta_m(Z),\theta_m(Z')].
\end{split}\end{equation*}
We see from the above equation that $\theta_m:(\mathfrak{g}_0)_-\rightarrow\mathfrak{gl}(\mathcal{V})$ is a complex representation of $(\mathfrak{g}_0)_-$. Consider the action of $G_0$ on the trivial bundle $\mathcal{D}_0\times\mathcal{V}$ given by 
\begin{equation}\label{gzzetagzmg}
g(z,v)=(gz,\mult(g,z)v)\quad(g\in G_0,z\in\mathcal{D}_0,v\in\mathcal{V}).
\end{equation}
We denote by $E_{\mult}$ the $G_0$-equivariant holomorphic vector bundle $\mathcal{D}_0\times\mathcal{V}$. 

\begin{lemma}[{\cite[Lemma 1]{ishi 2011}}]\label{equivalentline}
Let $\mult, \mult': G_0\times \mathcal{D}_0\rightarrow GL(\mathcal{V})$ be holomorphic multipliers. Then $E_{\mult}$ and $E_{\mult'}$ are isomorphic as $G_0$-equivariant holomorphic vector bundles if and only if there exists a matrix-valued holomorphic function $f:\mathcal{D}_0\rightarrow GL(\mathcal{V})$ such that
\begin{equation}\label{mult2gzfgz}
\mult'(g,z)=f(gz)\mult(g,z)f(z)^{-1}\quad(g\in G_0, z\in\mathcal{D}_0).
\end{equation}
\end{lemma}

\begin{definition}
We say that two holomorphic multipliers $\mult,\mult': G_0\times \mathcal{D}_0\rightarrow GL(\mathcal{V})$ are $G_0$-{\it equivalent} if they satisfy \eqref{mult2gzfgz} with some matrix-valued function $f$.
\end{definition}
The next theorem is fundamental for our paper. Let $K_0$ be the isotropy subgroup of $G_0$ at $p_0$.
\begin{theorem}[{\cite[Theorem 3.6]{tirao}}]\label{fundamentalone}
Suppose that the group $G_0$ acts on $\mathcal{D}_0$ transitively. 
Let $\mult,\mult': G_0\times \mathcal{D}_0\rightarrow GL(\mathcal{V})$ be holomorphic multipliers. Then holomorphic vector bundles $E_{\mult}$ and $E_{\mult'}$ are isomorphic as $G_0$-equivariant holomorphic vector bundles if and only if $\theta_{\mult}(Z)=\theta_{\mult'}(Z)$ for all $Z\in(\mathfrak{g}_0)_-$.

\end{theorem}

From now on, we discuss the representation $T_m$ and its unitarizations.
\begin{definition}
We say that the representation $T_{\mult}$ is {\it unitarizable} if there exists a nonzero Hilbert space $\mathcal{H}\subset\mathcal{O}(\mathcal{D}_0,\mathcal{V})$ satisfying the following conditions:
\begin{enumerate}
\item[(i)]the inclusion map $\iota:\mathcal{H}\hookrightarrow \mathcal{O}(\mathcal{D}_0,\mathcal{V})$ is continuous with respect to the open compact topology of $\mathcal{O}(\mathcal{D}_0,\mathcal{V})$,
\item[(ii)] $T_{\mult}(g)\mathcal{H}\subset \mathcal{H}\,(g\in G_0)$ and $\|T_{\mult}(g)f\|=\|f\|\,(g\in G_0, f\in\mathcal{H})$, where $\|\cdot\|$ denotes the norm of $\mathcal{H}$. 
\end{enumerate}
For a unitarizable representation $T_{\mult}$, we call the subrepresentation $(T_{\mult},\mathcal{H})$ a {\it unitarization} of the representation $T_{\mult}$ of $G_0$.
\end{definition}
A Hilbert space $\mathcal{H}$ satisfying the condition (i) is a reproducing kernel Hilbert space. 
\begin{remark}
Let $\mult,\mult':G_0\times\mathcal{D}_0\rightarrow GL(\mathcal{V})$ be holomorphic multipliers. If $\mult$ and $\mult'$ are $G_0$-equivalent and $T_{\mult}$ is unitarizable, then $T_{\mult'}$ is also unitarizable, and the unitarizations are equivalent as unitary representations of $G_0$. 
On the other hand, even though $\mult$ and $\mult'$ are not $G_0$-equivalent, the unitarizations of $T_{\mult}$ and $T_{\mult'}$ can be equivalent as unitary representations of $G_0$ (see Section \ref{Applicatio}). 
\end{remark}

We fix a Hermitian inner product on $\mathcal{V}$. Suppose that a holomorphic multiplier representation $T_{\mult}$ has a unitarization $(T_{\mult},\mathcal{H})$. For $v\in\mathcal{V}$ and $w\in\mathcal{D}_0$, let $\mathcal{K}_{v,w}\in \mathcal{O}(\mathcal{D}_0,\mathcal{V})$ be the function defined by
\begin{equation*}
(f,\mathcal{K}_{w,v})_\mathcal{H}=(f(w),v)_\mathcal{V}\quad(f\in\mathcal{O}(\mathcal{D}_0,\mathcal{V})).
\end{equation*}
Let $\mathcal{K}:\mathcal{D}_0\times\mathcal{D}_0\rightarrow\mathrm{End}(\mathcal{V})$ be the reproducing kernel of $\mathcal{H}$ defined by 
\begin{equation*}
\quad \mathcal{K}(z,w)v=\mathcal{K}_{w,v}(z)\quad(z,w\in\mathcal{D}_0,v\in\mathcal{V}).
\end{equation*}
Then $\mathcal{K}$ satisfies
\begin{equation}\label{mathcalKgz}
\mathcal{K}(gz,gw)=\mult(g,z)\mathcal{K}(z,w)\mult(g,w)^*\quad(z,w\in\mathcal{D}_0, g\in G_0).
\end{equation}
The next lemma shows that the converse also holds.
\begin{lemma}[{\cite[Lemma 5]{ishi 2011}}]\label{lem:1}
Let $\mathcal{H}\subset\mathcal{O}(\mathcal{D}_0,\mathcal{V})$ be a Hilbert space with reproducing kernel $\mathcal{K}$ and let $\mult: G_0\times\mathcal{D}_0\rightarrow GL(\mathcal{V})$ be a holomorphic multiplier. Then 
$(T_{\mult}, \mathcal{H})$ is a unitarization of $T_{\mult}$ if and only if \eqref{mathcalKgz} holds.
\end{lemma}

The next theorem is also fundamental for our paper.
\begin{theorem}[{\cite[Theorem 6]{ishi 2011}}, {\cite{kobayashi}}, {\cite{kunze}}]\label{uniqueness of unitarification}
If $G_0$ acts on $\mathcal{D}_0$ transitively and the map $K_0\ni k\mapsto m(k,p_0)\in GL(\mathcal{V})$ defines an irreducible representation of $K_0$, then a Hilbert space giving a unitarization of $T_{\mult}$ is unique if it exists. In particular, the unitarization is irreducible.
\end{theorem}

For any $g \in G_0$, $v\in \mathcal{V}$ and $f\in\mathcal{H}$, we have
\begin{equation*}\begin{split}
(f,T_{\mult}(g)\mathcal{K}(\cdot,p_0)v)&=(T_{\mult}(g^{-1})f,\mathcal{K}(\cdot,p_0)v)=(T_{\mult}(g^{-1})f(p_0),v)
\\&=(\mult(g,p_0)^{-1}f(gp_0),v). 
\end{split}\end{equation*}
The right hand side of the above equation is a $C^\omega$-function of $g\in G_0$. Hence $\mathcal{K}(\cdot,p_0)v$ is a $C^\omega$-vector of the representation $(T_{\mult},\mathcal{H})$. 

Let $\mathcal{V}=\mathbb{C}$. For $X\in\mathfrak{g}_0$ and $z\in\mathcal{D}_0$, we have
\begin{equation*}\begin{split}
dT_m(X)\mathcal{K}_{p_0}(z)&
=\dt T_m(e^{tX})\mathcal{K}_{p_0}(z)
\\&=\dt m(e^{-tX},z)^{-1}\mathcal{K}_{p_0}(e^{-tX}z)
\\&=\dt \overline{m(e^{tX},p_0)^{-1}}\mathcal{K}(z,e^{tX}p_0)
\\&=-\mathcal{K}(z,p_0)\dt \overline{m(e^{tX},p_0)}+\dt \mathcal{K}(z,e^{tX}p_0).
\end{split}\end{equation*}
Thus for $Z=X+iY\in(\mathfrak{g}_0)_-$ and $z\in\mathcal{D}_0$, we have
\begin{equation*}\begin{split}
dT_m(X-iY)\mathcal{K}_{p_0}(z)&=-\mathcal{K}(z,p_0)\dt \overline{m(e^{tX},p_0)+im(e^{tY},p_0)}\\&\quad+\dt \mathcal{K}(z,e^{tX}p_0)-i\dt\mathcal{K}(z,e^{tY}p_0).
\end{split}\end{equation*}
Since $\dt e^{tX}p_0-i\dt e^{tY}p_0\in T_{p_0}^{1,0}\mathcal{D}_0$, it follows that
\begin{equation}\label{dTmoverlin}
dT_m(\overline{Z})\mathcal{K}(\cdot,p_0)=-\overline{\theta_m(Z)}\mathcal{K}(\cdot,p_0)\quad(Z\in(\mathfrak{g}_0)_-),
\end{equation}
where $\overline{X+iY}=X-iY$ for $X,Y\in\mathfrak{g}_0$.
In general, for a unitary representation $(\pi,\mathcal{H}_0)$ of an arbitrary Lie group $G_0$, the moment map $J:\mathcal{P}(\mathcal{H}_0^\infty)\ni [v]
\mapsto J_{[v]}\in\mathfrak{g}_0^*$ is defined by
\begin{equation*}
J_{[v]}(X)=\frac{1}{i}\frac{(d\pi(X)v,v)_{\mathcal{H}_0}}{(v,v)_{\mathcal{H}_0}}\quad(X\in\mathfrak{g}_0).
\end{equation*}
Let $J:\mathcal{P}(\mathcal{H}^\infty)\ni [v]
\mapsto J_{[v]}\in\mathfrak{g}_0^*$ be the moment map of $(T_{\mult},\mathcal{H})$, and we put 
\begin{equation}\label{xiJKcdotpi}
\xi=J_{[\mathcal{K}(\cdot,p_0)]}\in\mathfrak{g}_0^*. 
\end{equation} 
Then by (\ref{dTmoverlin}), we have 
\begin{equation*}
\theta_m(Z)=i\xi(Z)\quad(Z\in (\mathfrak{g}_0)_-).
\end{equation*}

\subsection{The case of bounded homogeneous domains}
A bounded domain is said to be homogeneous if the holomorphic automorphism group acts on the domain transitively. Let $\mathcal{D}\subset \mathbb{C}^N$ be a domain which is biholomorphic to a bounded homogeneous domain. It is well known that the holomorphic automorphism group $\mathrm{Aut}_{hol}(\mathcal{D})$ of $\mathcal{D}$ has a canonical structure of a Lie group from the viewpoint of the group action. 
\begin{definition}
For an arbitrary Lie group $G_0$, we call a maximal connected real split solvable Lie subgroup of $G_0$ an {\it Iwasawa} subgroup.
\end{definition}
Let $\mathrm{Aut}_{hol}(\mathcal{D})^o$ be the identity component of $\mathrm{Aut}_{hol}(\mathcal{D})$. It is known \cite[Theorem 3.2]{kaneyuki} that $\mathrm{Aut}_{hol}(\mathcal{D})^o$ is isomorphic to the identity component of a linear real algebraic Lie group. Let $G$ be the identity component of a real algebraic subgroup of $\mathrm{Aut}_{hol}(\mathcal{D})^o$ which acts on $\mathcal{D}$ transitively. For any linear real algebraic group $G_0$, the identity component $G_0^o$ can be topologically decomposed into the direct product of a maximal compact subgroup of $G_0$ and an Iwasawa subgroup of $G_0$ (see \cite[Chapter 4, Theorem 4.7]{encyclopedia}). We fix a reference point $p\in\mathcal{D}$. It is known that the isotropy subgroup of $\mathrm{Aut}_{hol}(\mathcal{D})^o$ at $p$ is a maximal compact subgroup of $\mathrm{Aut}_{hol}(\mathcal{D})^o$.
The group $G$ contains an Iwasawa subgroup $B$ of $\mathrm{Aut}_{hol}(\mathcal{D})^o$ which acts on $\mathcal{D}$ simply transitively, and hence we can identify $\mathcal{D}$ with $B$. 
Note that the isotropy subgroup $K$ of $G$ at $p$ is connected because $\mathcal{D}$ is simply connected and $G$ is connected. 

In general a bounded homogeneous domain is a contractible Stein manifold. Thus every $G$-equivariant holomorphic line bundle over $\mathcal{D}$ is isomorphic as a $G$-equivariant holomorphic line bundle to $E_{\mult}=\mathcal{D}\times \mathbb{C}$ with some holomorphic multiplier $\mult: G\times \mathcal{D}\rightarrow \mathbb{C}^\times$. For $p\in\mathcal{D}$, let $\mathfrak{g}_-\subset \mathfrak{g}_\mathbb{C}$ be the complex subalgebra defined by \eqref{mathfrakg0l}.
\begin{theorem}[{\cite[Theorem 3.6]{tirao}}]\label{fundamentaltwo}
Let $\theta:\mathfrak{g}_-\rightarrow\mathfrak{gl}(\mathcal{V})$ be a complex representation of $\mathfrak{g}_-$ whose restriction to $\mathfrak{k}$ lifts to a representation of $K$. Then there exists a holomorphic multiplier $\mult: G\times\mathcal{D}\rightarrow GL(\mathcal{V})$ such that $\theta(Z)=\theta_{\mult}(Z)$ for all $Z\in\mathfrak{g}_-$.
\end{theorem}

We get the following lemma by the decomposition $G=BK$. 
\begin{lemma}\label{extendm}
Let $\mult:G\times\mathcal{D}\rightarrow GL(\mathcal{V})$ be a \textup{(}not necessarily holomorphic\textup{)} multiplier. If a $\mathcal{V}$-valued function $f$ on $\mathcal{D}$ satisfies 
\begin{equation}\label{multconp1q}
m(k,p)f(p)=f(p)\quad(\con\in K)
\end{equation}
and 
\begin{equation}\label{fbzmultbzf}
f(bz)=\mult(b,z)f(z)\quad(b\in B, z\in\mathcal{D}),
\end{equation}
then we have
\begin{equation}\label{fgzmultgzf}
f(gz)=\mult(g,z)f(z)\quad(g\in G, z\in\mathcal{D}).
\end{equation}
\end{lemma}
\begin{proof}
We consider the $G$-equivariant vector bundle $\mathcal{D}\times\mathcal{V}$, and regard $f$ as a section of the vector bundle. Then \eqref{fbzmultbzf} means that the section $f$ is $B$-invariant under the action of $B$, and \eqref{fgzmultgzf} means that the section $f$ is $G$-invariant under the action of $G$. Therefore, since $B$ acts on $\mathcal{D}$ transitively, \eqref{fbzmultbzf} and \eqref{fgzmultgzf} are equivalent. 
\end{proof}

The following proposition is just an application of the previous lemma.
\begin{theorem}\label{LetmGtimes}
Let $\mult:G\times \mathcal{D}\rightarrow GL(\mathcal{V})$ be a holomorphic multiplier, and let $\mathcal{H}\subset \mathcal{O}(\mathcal{D},\mathcal{V})$ be a reproducing kernel Hilbert space. We fix a Hermitian inner product on $\mathcal{V}$ such that $m(k,p)\in U(\mathcal{V})$ for all $k\in K$. Suppose that the reproducing kernel $\mathcal{K}$ of $\mathcal{H}$ satisfies $\mathcal{K}(p,p)\in \mathrm{Hom}_K(\mathcal{V},\mathcal{V})$, and the representation $T_m$ satisfies $T_{\mult}(b)\mathcal{H}\subset\mathcal{H}\,(b\in B)$ and $\|T_{\mult}(b)f\|=\|f\|\,(b\in B,f\in\mathcal{H})$. Then we have $T_{\mult}(g)\mathcal{H}\subset\mathcal{H}\,(g\in G)$ and $\|T_{\mult}(g)f\|=\|f\|\,(g\in G,f\in\mathcal{H})$.
\end{theorem}
\begin{proof}[Proof]
By Lemma \ref{lem:1}, for all $g\in B$, we have
\begin{equation}\label{mathcalKbz}
\mathcal{K}(gz,gz)=\mult(g,z)K(z,z)\mult(g,z)^*\quad( z\in\mathcal{D}).
\end{equation}
Let $\mathcal{K}^d$ be the $\mathrm{End}(\mathcal{V})$-valued function on $\mathcal{D}$ given by $\mathcal{K}^d(z)=\mathcal{K}(z,z)$, and let $\tilde{m}:G\times \mathcal{D}\rightarrow GL(\mathrm{End}(\mathcal{V}))$ be a multiplier defined by 
\begin{equation*}
\tilde{m}(g,z)A=m(g,z)\circ A\circ m(g,z)^*\quad(A\in \mathrm{End}(\mathcal{V})).
\end{equation*}
Applying Lemma \ref{extendm} to $\tilde{m}$ and $\mathcal{K}^d$, we see that \eqref{mathcalKbz} holds for all $g\in G$. By the analytic continuation, the equation
\begin{equation*}
\mathcal{K}(gz,gw)=\mult(g,z)\mathcal{K}(z,w)\mult(g,w)^*\quad( g\in G, z,w\in\mathcal{D})
\end{equation*}
holds. This proves the result by Lemma \ref{lem:1}.
\end{proof}

\begin{remark}
When $\mathcal{V}=\mathbb{C}$, the condition $\mathcal{K}(p,p)\in \mathrm{Hom}_K(\mathcal{V},\mathcal{V})$ in the previous proposition holds automatically.
\end{remark}

\section{Normal $j$-algebras and bounded homogeneous domains}\label{Normaljalg}
In this section, we review the theory of normal $j$-algebras in \cite{datri,miatello,pyatetskii,Rossi,CIME} and explain the relationship between normal $j$-algebras and bounded homogeneous domains. 

For $X\in\mathfrak{aut}_{hol}(\mathcal{D})$, let $X^\#$ denote the vector field on $\mathcal{D}$ given by 
\begin{equation*}
X^\#_z=\dt e^{tX}z\quad(z\in\mathcal{D}).
\end{equation*}
We fix a $B$-invariant K\"{a}hler metric $\langle\langle\cdot,\cdot\rangle\rangle$ on $\mathcal{D}$ such that $\langle\langle j_0X,j_0 Y\rangle\rangle=\langle\langle X,Y\rangle\rangle$ for all vector fields $X,Y$ over $\mathcal{D}$, where $j_0$ denotes the complex structure on $\mathcal{D}$ induced from the one of $\mathbb{C}^N$. For example, $\langle\langle\cdot,\cdot\rangle\rangle$ may be the Bergman metric on $\mathcal{D}$, or if $\mathcal{D}$ is contained in a complex domain $\hat{\mathcal{D}}$ of larger dimension as $B$-submanifold, then we can take $\langle\langle\cdot,\cdot\rangle\rangle$ as the restriction of the Bergman metric of $\hat{\mathcal{D}}$ to $\mathcal{D}$. Let $j$ be the complex structure on $\mathfrak{b}$ given by 
\begin{equation*}
(jX)^\#_p=j_0X^\#_p\quad(X\in\mathfrak{b}),
\end{equation*}
and let $\langle\cdot,\cdot\rangle$ be the inner product on $\mathfrak{b}$ given by
\begin{equation*}
\langle X,Y\rangle=\langle\langle X^\#,Y^\#\rangle\rangle(p)\quad(X,Y\in\mathfrak{b}).
\end{equation*}
By Gindikin, Piatetski-Shapiro, and Vinberg \cite[Part III, Lemma 1]{CIME}, there exists a linear form $\omega\in\mathfrak{b}^*$ such that
\begin{equation*}
\langle X,Y\rangle=\omega([jX,Y])\quad(X,Y\in\mathfrak{b}),
\end{equation*}
and $(\mathfrak{b},j,\omega)$ is a normal $j$-algebra. Namely, $\mathfrak{b}$ is a real split solvable Lie algebra with the equality
\begin{equation}\label{XYjjXyjXjY}
[X,Y]+j[jX,Y]+j[X,jY]=[jX,jY]\quad(X,Y\in\mathfrak{b})
\end{equation} 
and the bilinear form $\langle X,Y\rangle=\omega([jX,Y])\,(X,Y\in\mathfrak{b})$ is a $j$-invariant inner product, that is, 
\begin{equation*}
\langle X,X\rangle>0\quad(X\neq 0\in\mathfrak{b}),
\end{equation*}
\begin{equation*}
\langle jX,jY\rangle=\langle X,Y\rangle \quad(X,Y\in\mathfrak{b}).
\end{equation*}

It is known that $\mathfrak{a}=[\mathfrak{b},\mathfrak{b}]^\perp$ is a Cartan subalgebra of $\mathfrak{b}$. For $\alpha\in\mathfrak{a}^*$, let $\mathfrak{b}_\alpha$ be the root space associated to $\alpha$ given by \begin{equation*}
\mathfrak{b}_\alpha=\{X\in\mathfrak{b};[A,X]=\alpha(A)X\,(A\in\mathfrak{a})\}. 
\end{equation*}
\begin{theorem}[Piatetski-Shapiro, {\cite[Chapter 2, Section 3 and 5]{pyatetskii}}]\label{Forasuitab}
For a suitable basis $A_\one,\cdots, A_r$ of $\mathfrak{a}$, the following assertions hold: if we put $E_k=-jA_k$, then we have $[A_k,E_l]=\delta_{k,l}E_l\,(1\leq k,l\leq r)$, if we denote the dual basis of $A_\one,\cdots, A_r$ by $\alpha_\one,\cdots, \alpha_r\in\mathfrak{a}^*$, then we have
\begin{equation*}
\mathfrak{b}=\mathfrak{b}(0)\oplus\mathfrak{b}(1/2)\oplus\mathfrak{b}(1),
\end{equation*}
where
\begin{equation*}\begin{split}
&\mathfrak{b}(0)=\mathfrak{a}\oplus
\sideset{}{^\oplus}\sum_{1\leq k<l\leq r}\mathfrak{b}_{(\alpha_l-\alpha_k)/2},\quad
\mathfrak{b}(1/2)=\sideset{}{^\oplus}\sum_{1\leq k\leq r}\mathfrak{b}_{\alpha_k/2},
\\&
\mathfrak{b}(1)=\sideset{}{^\oplus}\sum_{1\leq k\leq r}\mathfrak{b}_{\alpha_k}\oplus
\sideset{}{^\oplus}\sum_{1\leq k<l\leq r}\mathfrak{b}_{(\alpha_l+\alpha_k)/2},
\end{split}\end{equation*}
and the equalities $\mathfrak{b}_{\alpha_k}=\mathbb{R}E_k$, $j\mathfrak{b}_{(\alpha_l-\alpha_k)/2}=\mathfrak{b}_{(\alpha_l+\alpha_k)/2}$, and $j\mathfrak{b}_{\alpha_k/2}=\mathfrak{b}_{\alpha_k/2}$ hold. We have the relation \begin{equation*}
[\mathfrak{b}(\gamma),\mathfrak{b}(\gamma')]\subset\mathfrak{b}(\gamma+\gamma')\quad(\gamma,\gamma'=0,1/2,1),
\end{equation*}
where we put $\mathfrak{b}(\gamma)=0$ for $\gamma>1$. 
\end{theorem}
Following \cite[Chapter 2, Section 5]{pyatetskii}, we introduce the Siegel domain $\mathcal{D}(\Omega,Q)$ on which the group $B$ acts simply transitively as affine automorphisms as follows. Put 
\begin{equation*}
E=E_\one+\cdots +E_r.
\end{equation*}
Let $B(0)$ be the connected Lie subgroup of $B$ with Lie algebra $\mathfrak{b}(0)$, and let $\Omega=\mathrm{Ad}(B(0))E\subset \mathfrak{b}(1)$. Let $Q:(\mathfrak{b}(1/2),j)\times (\mathfrak{b}(1/2),j)\rightarrow \mathfrak{b}(1)_\mathbb{C}$ be the sesquilinear map defined by \begin{equation*}
Q(V,V')=\frac{1}{4}([jV,V']+i[V,V'])\quad(V,V'\in\mathfrak{b}(1/2)).
\end{equation*}
Then $\Omega\subset\mathfrak{b}(1)$ is an open convex cone containing no straight lines, and $B(0)$ acts on $\Omega$ simply transitively. One has $Q(V,V)\in \overline{\Omega}\setminus\{0\}$ for all $V\in\mathfrak{b}(1/2)\setminus\{0\}$. Let
\begin{equation*}
\mathcal{D}(\Omega,Q)=\{(U,V)\in \mathfrak{b}(1)_\mathbb{C}\oplus\mathfrak{b}(1/2):\Im U-Q(V,V)\in\Omega\}.
\end{equation*}
The subgroup $B(0)$ acts on $\mathcal{D}(\Omega,Q)$ by
\begin{equation*}
t_0(U,V)=(\mathrm{Ad}(t_0)U,\mathrm{Ad}(t_0)V)\quad(t_0\in B(0), (U,V)\in\mathcal{D}(\Omega,Q)),
\end{equation*}
and for $U_0\in\mathfrak{b}(1)$ and $V_0\in\mathfrak{b}(1/2)$, the element $\exp(U_0+V_0)$ of $B$ acts on $\mathcal{D}(\Omega,Q)$ by
\begin{equation}\label{expX0W0ZWZ}\begin{split}
\exp(U_0+V_0)(U,V)=(U+U_0+2iQ(V,V_0)+iQ(V_0,V_0),V+V_0)
\\((U,V)\in\mathcal{D}(\Omega,Q)).
\end{split}\end{equation}
Define $\mathcal{C}:\mathcal{D}(\Omega,Q)\rightarrow\mathcal{D}$ by $\mathcal{C}(b(iE,0))=bp\,(b\in B)$. Then the map $\mathcal{C}$ is biholomorphic and is a generalization of the Cayley transform. 
\begin{remark}\label{Bydatrithe}
\begin{enumerate}
\item[(i)] The exponential map $\exp:\mathfrak{b}\rightarrow B$ is bijective (\cite{fujiwara}, Theorem 5.2.16), and we have $B=B(0)\ltimes \exp(\mathfrak{b}(1/2)\oplus\mathfrak{b}(1))$ (see Lemma \ref{LetG0beaco}).
\item[(ii)] By J.E. D'Atri \cite{datri}, the decomposition 
\begin{equation*}
[\mathfrak{b},\mathfrak{b}]=
\sideset{}{^\oplus}\sum_{1\leq k<l\leq r}\mathfrak{b}_{(\alpha_l-\alpha_k)/2}\oplus
\sideset{}{^\oplus}\sum_{1\leq k\leq r}\mathfrak{b}_{\alpha_k/2}
\oplus\sideset{}{^\oplus}\sum_{1\leq k\leq r}\mathfrak{b}_{\alpha_k}\oplus
\sideset{}{^\oplus}\sum_{1\leq k<l\leq r}\mathfrak{b}_{(\alpha_l+\alpha_k)/2}
\end{equation*}
is orthogonal with respect to $\langle\cdot,\cdot \rangle$.
\item[(iii)] The number $r=\dim \mathfrak{a}$ is called the rank of $\mathfrak{b}$.
\item[(iv)] An open convex cone $\Omega_0$ in a finite-dimensional vector space $\mathcal{V}$ is called regular if $\Omega_0$ contains no straight lines, and is called homogeneous if the group
\begin{equation*}
G(\Omega_0)=\{A\in GL(\mathcal{V});A\Omega_0=\Omega_0\}
\end{equation*}
acts on $\Omega_0$ transitively. Thus the open convex cone $\Omega$ in $\mathfrak{b}(1)$ is regular and homogeneous.
\end{enumerate}
\end{remark}

\begin{example}
Let $\pone\geq \ptwo\geq 1$. The domain 
\begin{equation*}
\mathcal{D}_I(\pone,\ptwo)=\{z\in M(\pone,\ptwo;\mathbb{C});\|z\|_{op}<1 \} \end{equation*}
is a bounded symmetric domain of type I, where $\|z\|_{op}$ denotes the operator norm of $z$. Put 
\begin{equation*}\begin{split}
H_\ptwo(\mathbb{C})=\{U\in M_\ptwo(\mathbb{C});U=\overline{{}^tU}\},
\\\mathcal{P}_\ptwo=\{U\in H_\ptwo(\mathbb{C});U\gg 0\}.
\end{split}\end{equation*}
We have the following isomorphisms:
\begin{equation*}\begin{split}
&\mathfrak{b}(1)\simeq H_\ptwo(\mathbb{C}),\quad\Omega\simeq\mathcal{P}_\ptwo,\quad\mathfrak{b}(1/2)\simeq M(\pone-\ptwo,\ptwo;\mathbb{C}),
\end{split}\end{equation*}
and the following domain is biholomorphic to $\mathcal{D}_I(\pone,\ptwo)$:
\begin{equation*}\begin{split}
\mathcal{D}(\mathcal{P}_\ptwo,\mathcal{Q})&\simeq\left\{\left(\begin{array}{c}U\\V\end{array}\right)\in M(q,r;\mathbb{C});\Im U-\mathcal{Q}(V,V)\gg 0\right\},
\end{split}\end{equation*}
where $\mathcal{Q}(V,V')=\tfrac{1}{2}\overline{{}^tV'}V$.
\end{example}

\begin{lemma}\label{LetG0beaco}
Let $G_0$ be a connected and simply connected real split solvable Lie group, let $\exp:\mathfrak{g}_0\rightarrow G_0$ be the exponential map, and let $\mathfrak{h},\mathfrak{h}'\subset\mathfrak{g}_0$ be subalgebras such that $\mathfrak{g}_0=\mathfrak{h}\ltimes \mathfrak{h}'$. Then the subsets $\exp(\mathfrak{h})$ and $\exp(\mathfrak{h}')$ of $G_0$ are connected Lie subgroups of $G_0$.
\end{lemma}
\begin{proof}
Let $H$ and $H'$ be connected and simply connected Lie groups with Lie algebras $\mathfrak{h}$ and $\mathfrak{h}'$, respectively. By {\cite[Theorem 1.125]{Knapp}}, there exists an action $\tau$ of $H$ on $H'$ by automorphisms such that the Lie algebra of the semidirect product $H\times_\tau H'$ is isomorphic to $ \mathfrak{h}\ltimes\mathfrak{h}'$. Let $\tilde{H}$ and $\tilde{H'}$ be the connected Lie subgroups of $G_0$ with Lie algebras $\mathfrak{h}$ and $\mathfrak{h}'$, respectively. Since Lie groups $G_0$ and $H\times_\tau H'$ are isomorphic, the connected Lie subgroups $\tilde{H}$ and $\tilde{H'}$ are simply connected. By {\cite[Theorem 5.2.16]{fujiwara}}, we have $\exp(\mathfrak{h})=\tilde{H}$ and $\exp(\mathfrak{h}')=\tilde{H'}$ since $\mathfrak{h}$ and $\mathfrak{h}'$ are also exponential.
\end{proof}

\section{Algebraic properties of $\mathfrak{g}$}\label{Algebraicp}
\subsection{Holomorphic complete vector fields on Siegel domains}\label{Structuret}\label{Gradingstr}
Let $\mathcal{U}$ be a finite-dimensional vector space over $\mathbb{R}$, let $\Omega_0\subset\mathcal{U}$ be an open regular convex cone, let $\mathcal{V}$ be a finite-dimensional vector space over $\mathbb{C}$, and let $Q_0:\mathcal{V}\times\mathcal{V}\rightarrow \mathcal{U}_\mathbb{C}$ be a $\Omega_0$-positive Hermitian map, that is,
\begin{equation*}
Q_0(v,v)\in\overline{\Omega_0}\setminus\{0\}\quad(v\in\mathcal{V}\setminus\{0\}).
\end{equation*}
The following domain $\mathcal{D}(\Omega_0, Q_0)\subset \mathcal{U}_\mathbb{C}\oplus\mathcal{V}$ is called a Siegel domain:
\begin{equation*}
\mathcal{D}(\Omega_0, Q_0)=\{(u,v)\in\mathcal{U}_\mathbb{C}\oplus\mathcal{V};\Im u-Q_0(v,v)\in\Omega_0\}.
\end{equation*}
Let $\mathfrak{X}=\mathfrak{X}(\mathcal{D}(\Omega_0, Q_0))$ be the space of complete holomorphic vector fields on $\mathcal{D}(\Omega_0, Q_0)$. The map
\begin{equation*}
\mathfrak{aut}_{hol}(\mathcal{D}(\Omega_0, Q_0))\ni X\mapsto X^\#\in\mathfrak{X}
\end{equation*}
is bijective, and we have 
\begin{equation*}
 [X,Y]^\#=[Y^\#,X^\#]\quad(X,Y\in\mathfrak{aut}_{hol}(\mathcal{D}(\Omega_0, Q_0))).
\end{equation*}
For $u_0\in \mathcal{U}$, let $\partial_{u_0}$ be the holomorphic vector field on $\mathcal{D}(\Omega_0, Q_0)$ given by 
\begin{equation*}
\partial_{u_0}(u,v)=(u_0,0)\quad((u,v)\in\mathcal{D}(\Omega_0, Q_0)).
\end{equation*}
Here for every $(u,v)\in\mathcal{D}(\Omega_0,Q)$, we identify the tangent space $T_{(u,v)}\mathcal{D}(\Omega_0, Q_0)$ with $\mathcal{U}_\mathbb{C}\oplus\mathcal{V}$, and we consider a vector filed $X\in\mathfrak{X}$ as a $(\mathcal{U}_\mathbb{C}\oplus\mathcal{V})$-valued function. We denote by $D_X$ the corresponding differential operator
\begin{equation*}
D_X f(u,v)=\dt f((u,v)+tX(u,v)),
\end{equation*}
where $f$ is a vector-valued smooth function on $\mathcal{D}(\Omega_0,Q_0)$. Then we have 
\begin{equation*}
[X,Y]=D_XY-D_YX\quad(X,Y\in\mathfrak{X}).
\end{equation*}
For $v_0\in \mathcal{V}$, let $\tilde{\partial}_{v_0}$ be the holomorphic vector field on $\mathcal{D}(\Omega_0, Q_0)$ given by 
\begin{equation*}
\tilde{\partial}_{v_0}(u,v)= (2iQ_0(v,v_0),v_0)\quad((u,v)\in\mathcal{D}(\Omega_0, Q_0)).
\end{equation*}
For complex endomorphisms $\mathcal{A}\in\mathfrak{gl}(\mathcal{U}_\mathbb{C})$ and $\mathcal{B}\in \mathfrak{gl}(\mathcal{V})$, let $\mathcal{X}(\mathcal{A},\mathcal{B})$ be the holomorphic vector field on $\mathcal{D}(\Omega_0, Q_0)$ given by \begin{equation*}
\mathcal{X}(\mathcal{A},\mathcal{B})(u,v)=(\mathcal{A}u,\mathcal{B}v)\quad((u,v)\in\mathcal{D}(\Omega_0, Q_0)). 
\end{equation*} 
We say $\mathcal{B}\in\mathfrak{gl}(\mathcal{V})$ is associated with $\mathcal{A}\in\mathfrak{gl}(\mathcal{U}_\mathbb{C})$ if the equality
\begin{equation*}
\mathcal{A}Q_0(v,v')=Q_0(\mathcal{B}v,v')+Q_0(v,\mathcal{B}v')\quad(v,v'\in \mathcal{V})
\end{equation*}
holds. Let $\partial$ be the infinitesimal generator of the one-parameter subgroup $\mathcal{D}(\Omega_0,Q_0)\ni (u,v)\mapsto (e^{t}u,e^{t/2}v)\in\mathcal{D}(\Omega_0,Q_0)\,(t\in\mathbb{R})$. Then we have $\partial(u,v)=(u,1/2v)\,((u,v)\in\mathcal{D}(\Omega_0,Q_0))$, that is,
\begin{equation*}
\partial=\mathcal{X}(\mathrm{id}_{\mathcal{U}_\mathbb{C}},\tfrac{1}{2}\mathrm{id}_\mathcal{V}).
\end{equation*}
For $\gamma\in\mathbb{R}$, we put 
\begin{equation*}
{\mathfrak{X}}(\gamma)=\{X\in\mathfrak{X};[\partial,X]=\gamma X\}.\end{equation*}
Let $\mathfrak{g}(\Omega_0)$ denote the Lie algebra of the Lie group $G(\Omega_0)$. 
\begin{theorem}[Kaup, Matsushima, Ochiai, {\cite[Theorem 4 and 5]{kaup}}]\label{kaupth}
The Lie algebra $\mathfrak{X}$ has the following gradation:
\begin{equation*}
\mathfrak{X}=\mathfrak{X}(-1)\oplus\mathfrak{X}(-1/2)\oplus\mathfrak{X}(0)\oplus\mathfrak{X}(1/2)\oplus\mathfrak{X}(1),
\end{equation*}
and the non-positive part $\sum_{\gamma\leq 0}\mathfrak{X}(\gamma)$ is the Lie algebra corresponding to the group of affine automorphisms of $\mathcal{D}(\Omega_0,Q_0)$. One has
\begin{equation*}
\mathfrak{X}(-1)=\{\partial_{u_0};u_0\in \mathcal{U}\},
\end{equation*}
\begin{equation*}
\mathfrak{X}(-1/2)=\{\tilde{\partial}_{v_0};v_0\in \mathcal{V}\},
\end{equation*}
and \begin{equation*}\begin{split}
\mathfrak{X}(0)=\{\mathcal{X}(\mathcal{A},\mathcal{B});\mathcal{A}\in\mathfrak{g}(\Omega_0),\mathcal{B}\in\mathfrak{gl}(\mathcal{V}),\mathcal{B}\textup{ is associated with }\mathcal{A}\}.
\end{split}\end{equation*} 
\end{theorem}
We denote by $\mathcal{D}(\Omega_0)$ the tube domain $\{u\in\mathcal{U}_\mathbb{C};\Im u\in\Omega_0\}$, which is a special case of the Siegel domain with $\mathcal{V}=0$ and $Q_0=0$, and for $\mathcal{A}\in\mathfrak{gl}(\mathcal{U}_\mathbb{C})$, let $\mathcal{X}(\mathcal{A})$ be the holomorphic vector field on $\mathcal{D}(\Omega_0)$ given by $\mathcal{X}(\mathcal{A})(u)=\mathcal{A}u\,(u\in\mathcal{D}(\Omega_0))$. Then we see that
\begin{equation*}
\mathfrak{X}(\mathcal{D}(\Omega_0))(0)=\{\mathcal{X}(A);\mathcal{A}\in\mathfrak{g}(\Omega_0)\}. 
\end{equation*}
We have the following formulas (see {\cite[Chapter V, \S 1]{satake}}):
\begin{equation}\label{XABUAU}
[\mathcal{X}(\mathcal{A},\mathcal{B}),\partial_{u_0}]=-\partial_{\mathcal{A}{u_0}},
\end{equation}
\begin{equation}\label{XABVBV}
[\mathcal{X}(\mathcal{A},\mathcal{B}),\tilde{\partial}_v]=-\tilde{\partial}_{\mathcal{B}v},
\end{equation}
\begin{equation}\label{XABABXAABB}
[\mathcal{X}(\mathcal{A},\mathcal{B}),\mathcal{X}(\mathcal{A}',\mathcal{B}')]=-\mathcal{X}([\mathcal{A},\mathcal{A}'],[\mathcal{B},\mathcal{B}']).
\end{equation}

Next we see explicit descriptions of $\mathfrak{X}(1/2)$ and $\mathfrak{X}(1)$. 
\begin{proposition}[Satake, {\cite[Chapter V, Proposition 2.1]{satake}}]\label{Everyeleme}
Every element of $\mathfrak{X}(1/2)$ is uniquely written as
\begin{equation}\label{mathcalYPh}
\mathcal{Y}_{\Phi,c}(u,v)=(2iQ_0(v,\Phi(\overline{u})),\Phi(u)+c(v,v))\quad((u,v)\in\mathcal{D}(\Omega_0, Q_0))
\end{equation}
with a $\mathbb{C}$-linear map $\Phi:\mathcal{U}_\mathbb{C}\rightarrow \mathcal{V}$ and a symmetric $\mathbb{C}$-bilinear map $c:\mathcal{V}\times\mathcal{V}\rightarrow \mathcal{V}$ which satisfy the following conditions:
\begin{flalign}\label{PhiV0Umaps}
&\text{for each }v_0\in\mathcal{V},\text{ the linear map}&\tag{Y1}
\end{flalign}
\begin{equation*}
\Phi_{v_0}:\mathcal{U}\ni u\mapsto \Im Q_0(\Phi(u),v_0)\in\mathcal{U}
\end{equation*}
belongs to $\mathfrak{g}(\Omega_0)$,
\begin{equation}\label{HVcVV2iHPhiHVVVquadVVinmathcalV}
Q_0(c(v',v'),v)=2iQ_0(v',\Phi(Q_0(v,v')))\quad(v,v'\in\mathcal{V}).\tag{Y2}
\end{equation}
Conversely, for any pair $(\Phi,c)$ satisfying $(\ref{PhiV0Umaps})$ and $(\ref{HVcVV2iHPhiHVVVquadVVinmathcalV})$, the vector field $\mathcal{Y}_{\Phi,c}$ given by $(\ref{mathcalYPh})$ belongs to $\mathfrak{X}(1/2)$.
\end{proposition}
Let $e\in\Omega_0$. As we shall see at the end of this subsection, every vector field $\mathcal{Y}_{\Phi,c}$ is uniquely determined by the vector $\Phi(e)\in\mathcal{V}$, so that $\mathcal{Y}_{\Phi,c}$ will be also written as $\mathcal{Y}_\Phi$.
\begin{proposition}[Satake, {\cite[Chapter V, Proposition 2.2]{satake}}]
Every element of $\mathfrak{X}(1)$ is uniquely written as 
\begin{equation}\label{mathcalZab}
\mathcal{Z}_{a,b}(u,v)=(a(u,u),b(u,v))\quad((u,v)\in\mathcal{D}(\Omega_0, Q_0))
\end{equation}
with a symmetric $\mathbb{R}$-bilinear map $a:\mathcal{U}\times \mathcal{U}\rightarrow \mathcal{U}$ \textup{(}which we extend to a $\mathbb{C}$-bilinear map $a:\mathcal{U}_\mathbb{C}\times\mathcal{U}_\mathbb{C}\rightarrow\mathcal{U}_\mathbb{C}$\textup{)} and a $\mathbb{C}$-bilinear map $b:\mathcal{U}_\mathbb{C}\times \mathcal{V}\rightarrow \mathcal{V}$ which satisfy the following conditions: 
\begin{flalign}\label{mathcalAU0}
&\text{for each }u_0\in\mathcal{U},\text{ the linear map}&\tag{Z1}
\end{flalign}
\begin{equation*}
\mathcal{A}_{u_0}:\mathcal{U}\ni u\mapsto a(u_0,u)\in\mathcal{U}
\end{equation*}
belongs to $\mathfrak{g}(\Omega_0)$, 
\begin{flalign}\label{mathcalBU0}
&\textit{for any }u_0\in\mathcal{U}, \textit{ the linear map }&\tag{Z2}
\end{flalign}
\begin{equation*}\mathcal{B}_{u_0}:\mathcal{V}\ni v\mapsto \tfrac{1}{2}b(u_0,v)\in\mathcal{V}
\end{equation*}
is associated with $\mathcal{A}_{u_0}$, and $\Im\mathrm{\,tr\,}\mathcal{B}_{u_0}=0 $,
\begin{flalign}\label{UmapstoImH}
&\text{for any }v,v'\in\mathcal{V}, \text{ the linear map }&\tag{Z3}
\end{flalign}
\begin{align*}
\mathcal{U}\ni u\mapsto \Im Q_0(b(u,v),v')\in\mathcal{U}
\end{align*}
belongs to $\mathfrak{g}(\Omega_0)$,
\begin{equation*}\label{HVbHVVVHbH}
Q_0(b(Q_0(v'',v'),v''),v)=Q_0(v'',b(Q_0(v,v''),v'))\quad(v,v',v''\in\mathcal{V}). \tag{Z4}
\end{equation*}
Conversely, for any pair $(a,b)$ satisfying $(\ref{mathcalAU0})$, $(\ref{mathcalBU0})$, $(\ref{UmapstoImH})$, and $(\ref{HVbHVVVHbH})$, the vector field $\mathcal{Z}_{a,b}$ given by $(\ref{mathcalZab})$ belongs to $\mathfrak{X}(1)$.
\end{proposition}

\begin{example}
Let
\begin{equation*}\begin{split}
\mathcal{D}(\mathcal{P}_\ptwo,\mathcal{Q})&\simeq\left\{\left(\begin{array}{c}U\\V\end{array}\right)\in M(q,r;\mathbb{C});\Im U-\mathcal{Q}(V,V)\gg 0\right\},
\end{split}\end{equation*}
where $\mathcal{Q}(V,V')=\tfrac{1}{2}\overline{{}^tV'}V$. We put
for $\Phi\in M(\pone-\ptwo,\ptwo;\mathbb{C})$
\begin{equation*}
\mathcal{Y}_{\Phi}(U,V)=(2i\mathcal{Q}(V,\Phi \overline{{}^tU}),\Phi U+V\overline{{}^t\Phi}Vi).
\end{equation*}
Then we have
\begin{equation*}
\mathfrak{X}(1/2)=\{\mathcal{Y}_{\Phi};\Phi\in M(\pone-\ptwo,\ptwo;\mathbb{C})\}.
\end{equation*}
We put for $a\in H_r(\mathbb{C})$
\begin{equation*}
 \mathcal{Z}_a(U,V)=(UaU,VaU).
\end{equation*}
Then we have 
\begin{equation*}
\mathfrak{X}(1)=\{\mathcal{Z}_a;a\in H_r(\mathbb{C})\}.
\end{equation*}
\end{example}

\begin{lemma}[Satake, {\cite[Chapter V, \S2]{satake}}]
The following hold:
\begin{equation}\label{partialUYP}
[\partial_{u},\mathcal{Y}_\Phi]=\tilde{\partial}_{\Phi(u)},
\end{equation}
\begin{equation}\label{tildepartialVYPhiXAB}
[\tilde{\partial}_v,\mathcal{Y}_\Phi]=\mathcal{X}(\mathcal{A},\mathcal{B}),
\end{equation}
where $\mathcal{A}$ and $\mathcal{B}$ are given by $\mathcal{A}=4\Phi_v$ and $\mathcal{B}:\mathcal{V}\ni v'\mapsto 2i\Phi(Q_0(v',v))+2c(v,v')\in\mathcal{V}$. Moreover we have
\begin{equation}\label{partialUZa}
[\partial_u,\mathcal{Z}_{a,b}]=2\mathcal{X}(\mathcal{A}_u,\mathcal{B}_u).
\end{equation}
\end{lemma}
We fix a reference point $(ie,0)\in\mathcal{D}(\Omega_0, Q_0)$. Next we see a description of the subalgebra \begin{equation*}
 \mathfrak{X}_{(ie,0)}=\{X\in\mathfrak{X};X(ie,0)=0\}. 
\end{equation*}
Let $\partial'$ be the element of $\mathfrak{X}(0)$ given by
\begin{equation*}
\partial'(u,v)=(0,iv)\quad((u,v)\in\mathcal{D}(\Omega_0, Q_0)),
\end{equation*}
and let $\psi_{e}:\mathfrak{X}(1/2)\rightarrow \mathfrak{X}(-1/2)$, and $\varphi_{e}:\mathfrak{X}(1)\rightarrow\mathfrak{X}(-1)$ be linear maps given by
\begin{equation*}
\psi_{e}=\mathrm{ad}(\partial')\mathrm{ad}(\partial_{e})|_{\mathfrak{X}(1/2)}, 
\end{equation*}
and
\begin{equation*}
\varphi_{e}=\frac{1}{2}\mathrm{ad}(\partial_{e})^2|_{\mathfrak{X}(1)},
\end{equation*}
respectively. Put 
\begin{equation*}
\mathfrak{m}=\{X+\psi_{e}(X);X\in\mathfrak{X}(1/2)\},
\end{equation*}
\begin{equation*}
\mathfrak{m}'=\{X+\varphi_{e}(X); X\in\mathfrak{X}(1)\}.
\end{equation*}
\begin{theorem}[Kaup, Matsushima, Ochiai, {\cite[Theorem 6]{kaup}}]\label{Xie0mathca}
\begin{equation*}
\mathfrak{X}_{(ie,0)}=(\mathfrak{X}_{(ie,0)}\cap\mathfrak{X}(0))+\mathfrak{m}'+\mathfrak{m}.
\end{equation*} 
\end{theorem}
We note that $\varphi_{e}$ and $\psi_{e}$ are injective (see {\cite[p. 211$-$212]{satake}}),
and by {\cite[p. 215]{satake}}, we have
\begin{equation}\label{psiYPhitil}
\psi_{e}(\mathcal{Y}_\Phi)=\tilde{\partial}_{-i\Phi(e)}.
\end{equation}

\subsection{Relationship between $\mathfrak{X}(1/2)$ and $\mathfrak{X}(1)$}\label{Algebraicp12}
First we prove the following lemma on the relationship between $\mathfrak{X}(1/2)$ and $\mathfrak{X}(1)$.
\begin{lemma}\label{Letpartial}
Let $\partial_{u_0}\in\mathfrak{X}(-1)$. If $[\partial_{u_0},\mathfrak{X}(1)]=\{0\}$, then we have $[\partial_{u_0},\mathfrak{X}(1/2)]=\{0\}$.
\end{lemma}
\begin{proof}
Let $\mathcal{Y}_\Phi\in \mathfrak{X}(1/2)$. Then $\mathcal{Y}_{i\Phi}\in\mathfrak{X}(1/2)$ by Proposition \ref{Everyeleme}. We put $\mathcal{Z}_{a,b}=[\mathcal{Y}_\Phi,\mathcal{Y}_{i\Phi}]$. Then the equality 
\begin{equation*} 
a(u,u)=4Q_0(\Phi(u),\Phi(u))\quad(u\in \mathcal{U})
\end{equation*}
holds (see {\cite[Chapter V, Lemma 2.5]{satake}}). It follows from (\ref{partialUZa}) that $0=[\partial_{u_0},\mathcal{Z}_{a,b}]=2\mathcal{X}(\mathcal{A}_{u_0},\mathcal{B}_{u_0})$. Hence
\begin{equation*}
0=\mathcal{A}_{u_0}(u_0)=a(u_0,u_0)=4Q_0(\Phi(u_0),\Phi(u_0)).
\end{equation*}
By the $\Omega_0$-positivity of $Q_0$, we get $\Phi(u_0)=0$. Using (\ref{partialUYP}), we get
\begin{equation*}
[\partial_{u_0},\mathcal{Y}_\Phi]=\tilde{\partial}_{\Phi(u_0)}=0.
\end{equation*}
\end{proof}

\begin{lemma}\label{ForXABinma}
Let $\mathcal{X}(\mathcal{A},\mathcal{B})\in\mathfrak{X}(0)$, and let $\mathcal{Y}_{\Phi,c}\in\mathfrak{X}(1/2)$. We define a $\mathbb{C}$-linear map $\Phi':\mathcal{U}_\mathbb{C}\rightarrow \mathcal{V}$ and a $\mathbb{C}$-bilinear map $c':\mathcal{V}\times\mathcal{V}\rightarrow\mathcal{V}$ by $\mathcal{Y}_{\Phi',c'}=[\mathcal{Y}_\Phi,\mathcal{X}(\mathcal{A},\mathcal{B})]$. Then the followings hold:
\begin{enumerate}
\item[$(\mathrm{i})$]
$\Phi'(u)=-\Phi(\mathcal{A}u)+\mathcal{B}\Phi(u)\quad(u\in \mathcal{U}_\mathbb{C})$,
\item[$(\mathrm{ii})$]
$c'(v,v)=\mathcal{B}c(v,v)-2c(\mathcal{B}v,v)\quad(v\in\mathcal{V})$.
\end{enumerate}
\end{lemma}
\begin{proof}
For $(u,v)\in\mathcal{D}(\Omega_0, Q_0)$, we have
\begin{equation*}\begin{split}
[&\mathcal{Y}_{\Phi,c},\mathcal{X}(\mathcal{A},\mathcal{B})](u,v)=(2i\mathcal{A}Q_0(v,\Phi(\overline{u})),\mathcal{B}(\Phi(u)+c(v,v)))\\&\quad-(2iQ_0(v,\Phi(\overline{\mathcal{A}u}))+2iQ_0(\mathcal{B}v,\Phi(\overline{u})),\Phi(\mathcal{A}u)+2c(\mathcal{B}v,v))\in\mathcal{U}_\mathbb{C}\oplus\mathcal{V}.
\end{split}\end{equation*}
We see from this expression that the image of $[\mathcal{Y}_{\Phi,c},\mathcal{X}(\mathcal{A},\mathcal{B})](u,v)$ under the 2nd projection $\mathcal{U}_\mathbb{C}\oplus\mathcal{V}\ni(u,v)\mapsto v\in\mathcal{V}$ is equal to
\begin{equation*}
\mathcal{B}\Phi(u)-\Phi(\mathcal{A}u)+\mathcal{B}c(v,v)-2c(\mathcal{B}v,v),
\end{equation*}
which is same as $\Phi'(u)+c'(v,v)$.
\end{proof}

\begin{lemma}\label{ForZabazzp}
Let $\mathcal{Z}_{a,b}\in\mathfrak{X}(1)$, and let $\mathcal{X}(\mathcal{A},\mathcal{B})\in\mathfrak{X}(0)$. We define $\mathbb{C}$-bilinear maps $a':\mathcal{U}_\mathbb{C}\times\mathcal{U}_\mathbb{C}\rightarrow\mathcal{U}_\mathbb{C}$ and $b':\mathcal{U}_\mathbb{C}\times\mathcal{V}\rightarrow \mathcal{V}$ by $\mathcal{Z}_{a',b'}=[\mathcal{Z}_{a,b},\mathcal{X}(\mathcal{A},\mathcal{B})]$. Then the followings hold:
\begin{enumerate}
\item[$(\mathrm{i})$]
$a'(u,u)=\mathcal{A}a(u,u)-2a(\mathcal{A}u,u)\quad(u\in\mathcal{U}_\mathbb{C})$,
\item[$(\mathrm{ii})$]
$b'(u,v)=\mathcal{B}b(u,v)-b(\mathcal{A}u,v)-b(u,\mathcal{B}v)\quad(u\in\mathcal{U}_\mathbb{C},v\in\mathcal{V})$.
\end{enumerate}
\end{lemma}
\begin{proof}
For $(u,v)\in\mathcal{D}(\Omega_0, Q_0)$, we have
\begin{equation*}\begin{split}
[\mathcal{Z}&_{a,b},\mathcal{X}(\mathcal{A},\mathcal{B})](u,v)\\&=(\mathcal{A}a(u,u),\mathcal{B}b(u,v))-(2a(\mathcal{A}u,u),b(\mathcal{A}u,v)+b(u,\mathcal{B}v))\\&=(\mathcal{A}a(u,u)-2a(\mathcal{A}u,u),\mathcal{B}b(u,v)-b(\mathcal{A}u,v)-b(u,\mathcal{B}v))\in\mathcal{U}_\mathbb{C}\oplus\mathcal{V},
\end{split}\end{equation*}
which is same as $(a'(u,u),b'(u,v))$.
\end{proof}

\begin{proposition}\label{Whenr1fora}
Assume that $\dim\,\mathcal{U}=1$ and that $\mathfrak{f}$ is a subalgebra of $\mathfrak{X}$ which contains $\mathfrak{X}(-1/2)$ and $\partial$. Then for any $\mathcal{Y}_{\Phi}\in{\mathfrak{f}}$, we have $\mathcal{Y}_{i\Phi}\in{\mathfrak{f}}$.
\end{proposition}
\begin{proof}
When $\Phi(e)=0$, we have $\psi_e(\mathcal{Y}_\Phi)=\tilde{\partial}_{-i\Phi(e)}=0$\,(\ref{psiYPhitil}). Since $\psi_e$ is injective, we have $\mathcal{Y}_\Phi=0$. Thus $\Phi=0$ and the result follows. In what follows, we assume that $\Phi(e)\neq 0$. We define $\mathbb{C}$-linear maps $\mathcal{A}\in\mathfrak{gl}(\mathcal{U}_\mathbb{C}),\mathcal{B}\in\mathfrak{gl}(\mathcal{V})$, and $\Phi':\mathcal{U}_\mathbb{C}\rightarrow\mathcal{V}$ by \begin{equation*}
\mathcal{X}(\mathcal{A},\mathcal{B})=[ \tilde{\partial}_{\Phi(e)},\mathcal{Y}_{\Phi}],\quad \mathcal{Y}_{\Phi'}=[\mathcal{Y}_{\Phi},\mathcal{X}(\mathcal{A},\mathcal{B})].
\end{equation*}
By assumption, we have $\mathcal{X}(\mathcal{A},\mathcal{B})\in\mathfrak{f}$ and $\mathcal{Y}_{\Phi'}\in{\mathfrak{f}}$. Our goal is to prove that $\Phi'(e)=Ci\Phi(e)$ with some constant $C\in\mathbb{R}\setminus\{0\}$. Indeed, if the equation $\Phi'(e)=Ci\Phi(e)$ holds, then we have
\begin{equation*}
\psi_e(\mathcal{Y}_{\Phi'})=\tilde{\partial}_{-i\Phi'(e)}=\tilde{\partial}_{C\Phi(e)}=C\tilde{\partial}_{\Phi(e)}=C\psi_e(\mathcal{Y}_{i\Phi})=\psi_e(C\mathcal{Y}_{i\Phi}),
\end{equation*} 
and since $\psi_e$ is injective, $C\mathcal{Y}_{i\Phi}=\mathcal{Y}_{\Phi'}\in{\mathfrak{f}}$. Put $v_0=\Phi(e)$. By (\ref{tildepartialVYPhiXAB}), we have
\begin{equation*}
\mathcal{A}e=4\Im Q_0(v_0,\Phi(e))=4\Im Q_0(v_0,v_0)=0.
\end{equation*}
Since $\mathcal{U}=\mathbb{R}e$, we can define a Hermitian form $q_0$ on $\mathcal{V}$ by 
\begin{equation*}
 Q_0(v,v')=q_0(v,v')e\quad(v,v'\in\mathcal{V}). 
\end{equation*} 
Using (\ref{HVcVV2iHPhiHVVVquadVVinmathcalV}), for any $v \in \mathcal{V}$, we have
\begin{equation*}\begin{split}
Q_0&(c(v_0,v_0),v)=2iQ_0(v_0,\Phi(Q_0(v,v_0)))=2iQ_0(v_0,q_0(v,v_0)\Phi(e))
\\&=2iQ_0(v_0,q_0(v,v_0)v_0)=2iq_0(v_0,v)Q_0(v_0,v_0)=Q_0(2iq_0(v_0,v_0)v_0,v).
\end{split}\end{equation*} 
Hence 
$c(v_0,v_0)=2iq_0(v_0,v_0)v_0$. Using (\ref{tildepartialVYPhiXAB}), we get
\begin{equation*}\begin{split}
\mathcal{B}v_0&=2i\Phi(Q_0(v_0,v_0))+2c(v_0,v_0)=2iq_0(v_0,v_0)\Phi(e)+4iq_0(v_0,v_0)v_0\\&=6iq_0(v_0,v_0)v_0.
\end{split}\end{equation*}
Thanks to Lemma \ref{ForXABinma} (i), we obtain
\begin{equation*}\begin{split}
\Phi'(e)=-\Phi(\mathcal{A}e)+\mathcal{B}\Phi(e)=\mathcal{B}v_0=6iq_0(v_0,v_0)v_0=6q_0(v_0,v_0)i\Phi(e)
\end{split}\end{equation*} 
with $6q_0(v_0,v_0)\neq 0$.
\end{proof}

\subsection{Vector fields on homogeneous Siegel domains}
Consider the action of the group $B$ on the domain $\mathcal{D}(\Omega,Q)$ as in Section \ref{Normaljalg}. Let $\mathfrak{X}=\mathfrak{X}(\mathcal{D}(\Omega,Q))$ be the space of complete holomorphic vector fields on $\mathcal{D}(\Omega,Q)$. For a subspace $\mathcal{W}\subset \mathfrak{aut}_{hol}\mathcal{D}(\Omega,Q)$, let $\mathcal{W}^\#=\{X^\#;X\in \mathcal{W}\}\subset \mathfrak{X}$. Now we have
\begin{equation*}
T^\#=\mathcal{X}(\mathrm{ad}(T)|_{\mathfrak{b}(1)_\mathbb{C}},\mathrm{ad}(T)|_{\mathfrak{b}(1/2)})\quad(T\in\mathfrak{b}(0)). 
\end{equation*} 
Thus $\partial=(jE)^\#\in\mathfrak{b}^\#$. By (\ref{expX0W0ZWZ}), we also have
\begin{equation*}
U^\#=\partial_U\quad(U\in\mathfrak{b}(1)),
\end{equation*}
\begin{equation*}
V^\#=\tilde{\partial}_{V}\quad(V\in\mathfrak{b}(1/2)).
\end{equation*}
From these expressions, we see that
\begin{equation*}
\mathfrak{b}(1)^\#=\mathfrak{X}(-1),\quad \mathfrak{b}(1/2)^\#=\mathfrak{X}(-1/2),\quad\mathfrak{b}(0)^\#\subset\mathfrak{X}(0).
\end{equation*}
Note that there is a natural action of $G$ on $\mathcal{D}(\Omega,Q)$ which is given as the transfer of the action of $G$ on $\mathcal{D}$ by means of the biholomorphic map $\mathcal{C}$. We also have the $B$-invariant metric $\langle\langle\cdot,\cdot\rangle\rangle$ on $\mathcal{D}(\Omega,Q)$ which is the transfer of the metric $\langle\langle\cdot,\cdot\rangle\rangle$ on $\mathcal{D}$. We denote by $\nabla$ the Levi-Civita connection on $(\mathcal{D}(\Omega,Q),\langle\langle\cdot,\cdot\rangle\rangle)$. We define a map $\tilde{\nabla}:\mathfrak{b}\times\mathfrak{b}\rightarrow \mathfrak{b}$ by \begin{equation}\label{dtexpttild}
\dt\exp(t\tilde{\nabla}_{X}Y)(iE,0)=(\nabla_{X^\#}Y^\#)_{(iE,0)}.
\end{equation} 
Then we have
\begin{equation}\label{2langletild}
-2\langle \tilde{\nabla}_{X}{Y},Z\rangle=\langle[X,Y],Z\rangle-\langle[Z,X],Y\rangle-\langle X,[Z,Y]\rangle\quad(X,Y,Z\in\mathfrak{b}).
\end{equation}
We see from the above equation that $\tilde{\nabla}_{X}Y=\tilde{\nabla}_{Y}X$ for all $X,Y\in\mathfrak{b}(1)$. 

\begin{lemma}\label{For1krtild}
Let $1\leq k\leq r$, and let $X\in\mathfrak{b}(1)$. Then \begin{equation*}
\tilde{\nabla}_{E_{k}}X=j\mathrm{ad} (A_k)X.
\end{equation*} 
\end{lemma}
\begin{proof}
For $1\leq l< m\leq r$, put $n_{ml}=\dim\mathfrak{b}_{(\alpha_{m}+\alpha_{l})/2}\geq 0$. We take an orthogonal basis $(E_{ml}^{\kappa})_{\kappa=1}^{n_{ml}}$ of $\mathfrak{b}_{(\alpha_m+\alpha_{l})/2}$ such that \begin{equation*}
 [jE_{ml}^{\kappa},E_{ml}^{\kappa}]=E_{m}.
\end{equation*}
For $1\leq l\leq r$, put $C_{l}=\langle E_{l},E_{l}\rangle$. Then for $1\leq l<m\leq r$ and $1\leq \kappa\leq n_{ml}$, we have
\begin{equation*}
\langle E_{ml}^{\kappa},E_{ml}^{\kappa}\rangle=\omega(E_{m})=\omega([jE_{m},E_{m}])=C_{m}.
\end{equation*}
For $X=\sum_{1\leq m'\leq r}x_{m'}E_{m'}+\sum_{1\leq k'<l'\leq r}\sum_{1\leq \lambda\leq n_{l'k'}}x_{l'k'}^{\lambda}E_{l'k'}^{\lambda}$ and $1\leq l\leq r$, we have
\begin{equation*}\begin{split}
2&\langle \tilde{\nabla}_{E_k}X,jE_l\rangle=\langle[jE_l,E_k],X\rangle+\langle E_k,[jE_l,X]\rangle
\\&=\langle \delta_{kl}E_k,X\rangle+\textstyle\left\langle E_k,\frac{1}{2}\sum_{k'=1}^{l-1}\sum_{\lambda=1}^{n_{lk'}} x_{lk'}^{\lambda}E_{lk'}^{\lambda}+x_lE_l+\frac{1}{2}\sum_{l'=l+1}^r\sum_{\lambda=1}^{n_{l'l}} x_{l'l}^{\lambda}E_{l'l}^{\lambda}\right\rangle
\\&=2\delta_{kl}C_kx_l.
\end{split}\end{equation*} 
On the other hand, we have 
\begin{equation*}\begin{split}
\langle&\mathrm{ad}(A_k)X,E_l\rangle \\&=\textstyle\left\langle\frac{1}{2}\sum_{k'=1}^{k-1}\sum_{\lambda=1}^{n_{kk'}} x_{kk'}^{\lambda}E_{kk'}^{\lambda}+x_kE_k+\frac{1}{2}\sum_{l'=k+1}^r\sum_{\lambda=1}^{n_{l'k}} x_{l'k}^{\lambda}E_{l'k}^{\lambda},E_l\right\rangle\\&=\delta_{kl}C_kx_k. 
\end{split}
\end{equation*}
For $1\leq m<l\leq r$ and $1\leq \kappa\leq n_{lm}$, we have
\begin{equation*}\begin{split}
2\langle\tilde{\nabla}_{E_k}X,jE_{lm}^{\kappa}\rangle&=\langle[jE_{lm}^{\kappa},E_k],X\rangle+\langle E_k,[jE_{lm}^{\kappa},X]\rangle
\\&=\langle \delta_{mk}E_{lm}^{\kappa},X\rangle+\langle E_k,[jE_{lm}^{\kappa},x_{lm}^{\kappa}E_{lm}^{\kappa}]\rangle
\\&=\langle \delta_{mk}E_{lm}^{\kappa},X\rangle+\langle E_k,x_{lm}^{\kappa}E_l\rangle
\\&=\delta_{mk}C_lx_{lm}^{\kappa}+\delta_{kl}C_{l}x_{lm}^{\kappa}.
\end{split}\end{equation*}
On the other hand, we have
\begin{equation*}\begin{split}
\langle &\mathrm{ad}(A_k)X,E_{lm}^{\kappa}\rangle\\&=
\textstyle\left\langle\frac{1}{2}\sum_{k'=1}^{k-1}\sum_{\lambda=1}^{n_{kk'}}x_{kk'}^{\lambda}E_{kk'}^{\lambda}+x_kE_k+\frac{1}{2}\sum_{l'=k+1}^r\sum_{\lambda=1}^{n_{l'k}}x_{l'k}^{\lambda}E_{l'k}^{\lambda},E_{lm}^{\kappa}\right\rangle
\\&=\frac{1}{2}(\delta_{mk}C_lx_{lm}^{\kappa}+\delta_{kl}C_lx_{lm}^{\kappa}).
\end{split}\end{equation*}
Therefore we get 
\begin{equation*}
\langle \mathrm{ad}(A_k)X,Y\rangle=2\langle \tilde{\nabla}_{E_k}X,jY\rangle\quad(Y\in\mathfrak{b}(1)). 
\end{equation*}
Moreover, we see from (\ref{2langletild}) that
\begin{equation}
\langle \tilde{\nabla}_{E_k}X,Y\rangle=0\quad(Y\in\mathfrak{b}(1)\oplus\mathfrak{b}(1/2)).
\end{equation} 
The proof is complete.
\end{proof}

\begin{lemma}\label{ForF}
Let $1\leq k\leq r$, and let $\mathcal{A}\in\mathfrak{g}(\Omega)$. Then 
\begin{equation*}
\mathcal{A}E_k\in\sideset{}{^\oplus}\sum_{1\leq m\leq r}\mathfrak{b}_{(\alpha_k+\alpha_m)/2}.
\end{equation*}
\end{lemma}
\begin{proof}
The connected Lie subgroup of $B$ with the Lie algebra $\mathfrak{b}(0)\oplus\mathfrak{b}(1)$ is an Iwasawa subgroup of $\mathrm{Aut}_{hol}(\mathcal{D}(\Omega))$, and we have 
\begin{equation*}
\mathfrak{g}(\Omega)=\mathfrak{g}_{E}(\Omega)\oplus \{\mathrm{ad}(X);X\in\mathfrak{b}(0)\},
\end{equation*}
where $\mathfrak{g}_E(\Omega)$ is the Lie algebra of the Lie group 
\begin{equation*}
G_E(\Omega)=\{A\in G(\Omega);AE=E\}. 
\end{equation*}
The result for $\mathcal{A}=\mathrm{ad}(X)$ with $X\in\mathfrak{b}(0)$ follows from (\ref{XYjjXyjXjY}). Let $\mathcal{A}\in\mathfrak{g}_E(\Omega)$. By (\ref{XABUAU}), we have $[\mathcal{X}(\mathcal{A}),E_k^\#]=-(\mathcal{A}E_k)^\#$. Let $\langle\langle\cdot,\cdot\rangle\rangle'$ be the Bergman metric on $\mathcal{D}(\Omega)$, and let $\nabla'$ denote the connection on $(\mathcal{D}(\Omega),\langle\langle\cdot,\cdot\rangle\rangle')$. Then we also have the map $\tilde{\nabla'}:(\mathfrak{b}(0)\oplus\mathfrak{b}(1))\times(\mathfrak{b}(0)\oplus\mathfrak{b}(1))\rightarrow\mathfrak{b}(0)\oplus\mathfrak{b}(1)$, which is defined by (\ref{dtexpttild}). Let $1\leq l\leq r$ and $l\neq k$. Since $\mathcal{X}(\mathcal{A})$ generates isometries of $\mathcal{D}(\Omega)$, we have
\begin{equation}\label{mathcalXma}\begin{split}
[\mathcal{X}(\mathcal{A}),\nabla'_{E_l^\#}E_k^\#]&=\nabla'_{[\mathcal{X}(\mathcal{A}),E_l^\#]}E_k^\#+\nabla'_{E_l^\#}[\mathcal{X}(\mathcal{A}),E_k^\#]
\\&=-\nabla'_{(\mathcal{A}E_l)^\#}{E_k^\#}-\nabla'_{E_l^\#}(\mathcal{A}E_k)^\#. 
\end{split}\end{equation}
By Lemma \ref{For1krtild}, we have $(\nabla'_{E_l^\#}E_k^\#)_{iE}=0$. By looking at the value of (\ref{mathcalXma}) at $iE\in\mathcal{D}(\Omega)$, we have
\begin{equation}\label{0tildenabl}
0=\tilde{\nabla'}_{\mathcal{A}E_l}{E_k}+\tilde{\nabla'}_{E_l}\mathcal{A}E_k=\tilde{\nabla'}_{E_k}{\mathcal{A}E_l}+\tilde{\nabla'}_{E_l}\mathcal{A}E_k.
\end{equation}
We remark that the equation (\ref{0tildenabl}) can be seen from \cite{dorfmeister 1979} and \cite{dorfmeister 1982}.
By Lemma \ref{For1krtild}, 
we have 
\begin{equation*}
[A_k,\mathcal{A}E_l]=-[A_l,\mathcal{A}E_k]\in\mathfrak{b}. 
\end{equation*}
Thus $[A_k,\mathcal{A}E_l]$ belongs to both $\sum_{1\leq m\leq r}^\oplus\mathfrak{b}_{(\alpha_k+\alpha_m)/2}$ and $\sum_{1\leq m\leq r}^\oplus\mathfrak{b}_{(\alpha_l+\alpha_m)/2}$, and hence we obtain
\begin{equation}\label{AkmathcalA}
[A_l,\mathcal{A}E_k]\in\mathfrak{b}_{(\alpha_l+\alpha_k)/2}\quad(l\neq k).
\end{equation}
If $\textstyle\mathcal{A}E_k\notin\sum_{1\leq m\leq r}^\oplus\mathfrak{b}_{(\alpha_k+\alpha_m)/2}$, then $\mathcal{A}E_k$ can be written as $\mathcal{A}E_k=X+Y$ with $X\in\mathfrak{b}_{(\alpha_{l'}+\alpha_{k'})/2}\setminus\{0\}$ and $Y\in\sum_{\substack{1\leq k''\leq l''\leq r\\(k'',l'')\neq (k',l')}}^\oplus\mathfrak{b}_{(\alpha_{l''}+\alpha_{k''})/2}$ for some $k'\neq k$ and $l'\neq k$ satisfying $1\leq k'\leq l'\leq r$. Then
$[A_{l'},X+Y]=CX+[A_{l'},Y]$ with $C=1/2$ or $1$. Thus $[A_{l'},\mathcal{A}E_k]\notin \mathfrak{b}_{(\alpha_l+\alpha_k)/2}$, which contradicts (\ref{AkmathcalA}). Hence it follows that $\mathcal{A}E_k\in\sum_{1\leq m\leq r}^\oplus\mathfrak{b}_{(\alpha_k+\alpha_m)/2}$. 
\end{proof}
\begin{lemma}\label{ForV}
Let $V\in\mathfrak{b}(1/2)$, and let $1\leq k\leq r$. If $[\mathfrak{b}(1/2),V]\subset\sum_{1\leq m\leq r}^\oplus\mathfrak{b}_{(\alpha_k+\alpha_{m})/2}$, then $V\in\mathfrak{b}_{\alpha_k/2}$.
\end{lemma}
\begin{proof}
Let $V=\sum_{1\leq m\leq r}V_{m}$ with $V_{m}\in\mathfrak{b}_{\alpha_{m}/2}$, and suppose that $V\notin \mathfrak{b}_{\alpha_k/2}$. Then there exists $1\leq m_0\leq r$ such that $m_0\neq k$ and $V_{m_0}\neq 0$. We have
\begin{equation*}
[jV_{m_0},V_{m_0}]\in\mathfrak{b}_{\alpha_{m_0}}\setminus\{0\}.
\end{equation*}
Thus $[jV_{m_0}, V]\notin \sum_{1\leq m\leq r}^\oplus\mathfrak{b}_{(\alpha_k+\alpha_{m})/2}$.
\end{proof}

\begin{lemma}\label{foranyY}
Let $\mathcal{Y}_\Phi\in\mathfrak{X}(1/2)$. Then the followings hold:
\begin{enumerate}
 \item [$(\mathrm{i})$]
$\Phi(E_k)\in\mathfrak{b}_{\alpha_k/2}\quad(1\leq k\leq r)$,
\item[$(\mathrm{ii})$]
$[A_k^\#,[A_l^\#,\mathcal{Y}_\Phi]]=0\quad(1\leq k\leq r, 1\leq l\leq r, k\neq l)$.
\end{enumerate}
\end{lemma}
\begin{proof}
Let $1\leq k\leq r$, $1\leq l\leq r$, and $k\neq l$. We define $\mathbb{C}$-linear maps $\Phi^{l},\Phi^{lk}:\mathfrak{b}(1)_\mathbb{C}\rightarrow\mathfrak{b}(1/2)$ by \begin{equation*}
\mathcal{Y}_{\Phi^{l}}=[\mathcal{Y}_\Phi,A_l^\#],\quad\mathcal{Y}_{\Phi^{lk}}=[\mathcal{Y}_{\Phi^{l}},A_k^\#]. 
\end{equation*}
By Lemma \ref{ForXABinma} (i), we have
\begin{equation*}\begin{split}
&\Phi^{lk}(E)
\\&=-\Phi^{l}([A_k,E])+[A_k ,\Phi^{l}(E)]\\&=-(-\Phi([A_l,[A_k,E]])+[A_l,\Phi([A_k,E])])+[A_k,-\Phi([A_l,E])+[A_l,\Phi(E)]]
\\&=-[A_l,\Phi([A_k,E])]-[A_k,\Phi([A_l,E])]=-[A_l,\Phi(E_k)]-[A_k,\Phi(E_l)].
\end{split}\end{equation*}
Thus to prove (ii), it is enough to show that
\begin{equation}\label{AlPhiEk0}
[A_l,\Phi(E_k)]=0.
\end{equation}
For any $V\in\mathfrak{b}(1/2)$, we have \begin{equation*}
[V,\Phi(E_k)]=-4\Im Q(\Phi(E_k),V)=-4\Phi_V(E_k).
\end{equation*}
Since $\Phi_V\in\mathfrak{g}(\Omega)$, Lemma \ref{ForF} implies that
\begin{equation*}
[V,\Phi(E_k)]\in\sideset{}{^\oplus}\sum_{1\leq m\leq r}\mathfrak{b}_{(\alpha_k+\alpha_m)/2}\quad(V\in\mathfrak{b}(1/2)).
\end{equation*}
Thus Lemma \ref{ForV} shows that $\Phi(E_k)\in\mathfrak{b}_{\alpha_k/2}$. Hence (\ref{AlPhiEk0}) holds, and the proof is complete.
\end{proof}

Put 
\begin{equation*}
\mathfrak{f}=\mathfrak{g}^\#,\quad \mathfrak{f}(\gamma)=\{X\in\mathfrak{f}:[\partial,X]=\gamma X\}\quad(\gamma\in\mathbb{R}). 
\end{equation*}
\begin{lemma}\label{mathfrakzm}
The center $\mathfrak{z}(\mathfrak{f})$ of $\mathfrak{f}$ is trivial.
\end{lemma}
\begin{proof}
Let $X\in\mathfrak{z}(\mathfrak{f})$. Then we have $X\in\mathfrak{f}(0)$ since $[\partial, X]=0$. Put $X=\mathcal{X}(\mathcal{A},\mathcal{B})\in\mathfrak{f}(0)$. By (\ref{XABUAU}), for any $U\in\mathfrak{b}(1)$, we have
\begin{equation*}
0=[\mathcal{X}(\mathcal{A},\mathcal{B}),\partial_U]=-\partial_{\mathcal{A}U}.
\end{equation*}
Thus $\mathcal{A}=0$, and also $\mathcal{B}=0$ by (\ref{XABVBV}). Now we see that 
\begin{equation*}
X=\mathcal{X}(\mathcal{A},\mathcal{B})=0.
\end{equation*} 
\end{proof}

We shall extend the result of Proposition \ref{Whenr1fora} to the case of bounded homogeneous domains of arbitrary ranks by induction (see Proposition \ref{forasubalg}). From now on, we assume that $r\geq 2$. We define subalgebras $\check{\mathfrak{b}}\subset \mathfrak{b}$ and $\check{\mathfrak{f}}\subset \mathfrak{f}$ by 
\begin{equation*}\begin{split}
\check{\mathfrak{b}}=
\sideset{}{^\oplus}\sum_{2\leq k\leq r}\langle A_k\rangle&\oplus
\sideset{}{^\oplus}\sum_{2\leq k<l\leq r}\mathfrak{b}_{(\alpha_l-\alpha_k)/2}\oplus
\sideset{}{^\oplus}\sum_{2\leq k\leq r}\mathfrak{b}_{\alpha_k/2}\\&\oplus
\sideset{}{^\oplus}\sum_{2\leq k\leq r}\mathfrak{b}_{\alpha_k}\oplus
\sideset{}{^\oplus}\sum_{2\leq k<l\leq r}\mathfrak{b}_{(\alpha_l+\alpha_k)/2}
\end{split}\end{equation*}
and 
\begin{equation*}
\check{\mathfrak{f}}=\{X\in\mathfrak{f};[X,A_\one^\#]=[X,E_\one^\#]=0\}.
\end{equation*}
Put \begin{equation*}
\check{\mathfrak{f}}(\gamma)=\check{\mathfrak{f}}\cap\mathfrak{f}(\gamma)\quad(\gamma\in\mathbb{R}),
\end{equation*} 
\begin{equation*}
\check{\mathfrak{b}}(\gamma)=\check{\mathfrak{b}}\cap\mathfrak{b}(\gamma)\quad(\gamma=0,1/2,1).
\end{equation*} 
Then $\check{\mathfrak{b}}$ is a normal $j$-algebra of rank $r-1$.
Define
\begin{equation*}
\check{\Omega}=\exp (\check{\mathfrak{b}}(0))(E_2+\cdots +E_r),
\end{equation*}
\begin{equation*}
\mathcal{D}(\check{\Omega},\check{Q})=\{(U,V)\in\check{\mathfrak{b}}(1)_\mathbb{C}\oplus\check{\mathfrak{b}}(1/2);\Im U-Q(V,V)\in\check{\Omega}\}.
\end{equation*}
According to Lemma \ref{Omega1E1Om} below, the following inclusion holds:
\begin{equation*}
iE_\one+\mathcal{D}(\check{\Omega},\check{Q})\subset\mathcal{D}(\Omega,Q). 
\end{equation*}

\begin{lemma}\label{Omega1E1Om}
$\check{\Omega}+E_\one=\Omega\cap(\check{\mathfrak{b}}(1)+E_\one)$.
\end{lemma}
\begin{proof}
It follows from \cite[Proposition 2.5]{ishi 2000}.
\end{proof}

\begin{lemma}\label{ForY12andZ}
\begin{enumerate}
\item[$\textup{(i)}$] Let $\mathcal{Y}_\Phi\in\check{\mathfrak{X}}(1/2)$. Then 
\begin{equation*}
\mathcal{Y}_\Phi(iE_\one+U,V)\in\check{\mathfrak{b}}(1)_\mathbb{C}\oplus\check{\mathfrak{b}}(1/2)\quad(U\in\check{\mathfrak{b}}(1)_\mathbb{C}, V\in\check{\mathfrak{b}}(1/2)).
\end{equation*}
\item[$\textup{(ii)}$] 
An element $\mathcal{Y}_\Phi$ of ${\mathfrak{X}}(1/2)$ belongs to $\check{\mathfrak{X}}(1/2)$ if and only if $[\mathcal{Y}_\Phi,A_\one^\#]=0$.
\end{enumerate}
\end{lemma}
\begin{proof}
(i) Let $\mathcal{Y}_\Phi\in\check{\mathfrak{X}}(1/2)$. By Lemma \ref{ForXABinma} (i), we have \begin{equation*}
\Phi([A_\one,U])=[A_\one,\Phi(U)]\quad(U\in\mathfrak{b}(1)_\mathbb{C}). 
\end{equation*}
Thus
\begin{equation*}
\Phi([A_\one,E])=\Phi([A_\one,[A_\one,E]])=[A_\one,\Phi([A_\one,E])],
\end{equation*}
which implies 
\begin{equation}\label{PhiE1PhiA1}
 \Phi(E_\one)=\Phi([A_\one,E])=0. 
\end{equation} 
Let $U\in\check{\mathfrak{b}}(1)_\mathbb{C}$, and let $V\in\check{\mathfrak{b}}(1/2)$. Lemma \ref{ForXABinma} (i) shows that $[A_\one,\Phi(U)]=0$, and hence
\begin{equation}\label{PhiZinmath}
\Phi(U)\in\check{\mathfrak{b}}(1/2).
\end{equation}
By (\ref{PhiE1PhiA1}) and (\ref{PhiZinmath}), we see that
\begin{equation}\label{2itildeQPh}
2iQ(V,\Phi(-iE_\one+\overline{U}))=2iQ(V,\Phi(\overline{U}))\in\check{\mathfrak{b}}(1)_\mathbb{C}.
\end{equation}
On the other hand, from Lemma \ref{ForXABinma} (ii), it follows that
$[A_\one,c(V,V)]=0$. Thus $c(V,V)\in\check{\mathfrak{b}}(1/2)$. 
We see from from $(\ref{PhiZinmath})$ and (\ref{2itildeQPh}) that
\begin{equation*}
 \Phi(iE_\one+U)+c(V,V)\in\check{\mathfrak{b}}(1/2),
\end{equation*}
which proves (i).

(ii) Let $\mathcal{Y}_\Phi\in\mathfrak{X}(1/2)$, and suppose that $[\mathcal{Y}_\Phi,A_\one^\#]=0$. Then we see from (\ref{partialUYP}) and (\ref{PhiE1PhiA1}) that
\begin{equation*}
[E_\one^\#,\mathcal{Y}_\Phi]=(\Phi(E_\one))^\#=0.
\end{equation*}
This proves (ii).
\end{proof}

To simplify some of the notation, we abbreviate $\psi_E:\mathfrak{X}(1/2)\rightarrow \mathfrak{X}(-1/2)$ and $\varphi_E:\mathfrak{X}(1)\rightarrow \mathfrak{X}(-1)$ as $\psi$ and $\varphi$, respectively. 
\begin{lemma}\label{thefollowing}
Consider $\mathcal{D}(\check{\Omega},\check{Q})$ as a complex submanifold of $\mathcal{D}(\Omega,Q)$. Then for $\mathcal{Y}_\Phi\in\check{\mathfrak{X}}(1/2)$, we have $\mathcal{Y}_\Phi|_{iE_\one+\check{\mathfrak{b}}(1)_\mathbb{C}\oplus \check{\mathfrak{b}}(1/2)}\in \mathfrak{X}(\mathcal{D}(\check{\Omega},\check{Q}))(1/2)$, and the map $\check{\mathfrak{X}}(1/2)\ni \mathcal{Y}_\Phi\mapsto \mathcal{Y}_\Phi|_{iE_\one+\check{\mathfrak{b}}(1)_\mathbb{C}\oplus \check{\mathfrak{b}}(1/2)}\in\mathfrak{X}(\mathcal{D}(\check{\Omega},\check{Q}))(1/2)$ is injective.
\end{lemma}
\begin{proof}
First by Lemma \ref{Omega1E1Om}, the following equality holds:
\begin{equation*}
iE_\one+\mathcal{D}(\check{\Omega},\check{Q})=\mathcal{D}(\Omega,Q)\cap (iE_\one+\check{\mathfrak{b}}(1)_\mathbb{C}\oplus\check{\mathfrak{b}}(1/2)). 
\end{equation*} Let $\mathcal{Y}_\Phi\in\check{\mathfrak{X}}(1/2)$, and let $X$ be the element of $\mathfrak{aut}_{hol}(\mathcal{D}(\Omega,Q))$ such that $X^\#=\mathcal{Y}_\Phi$. Let $y:\mathbb{R}\rightarrow \mathrm{Aut}_{hol}(\mathcal{D}(\Omega,Q))$ denote the one-parameter subgroup of $\mathrm{Aut}_{hol}(\mathcal{D}(\Omega,Q))$ given by $y(t)=\exp(tX)\,(t\in\mathbb{R})$. Then $y$ preserves $\mathcal{D}(\Omega,Q)\cap(iE_\one+\check{\mathfrak{b}}(1)_\mathbb{C}\oplus\check{\mathfrak{b}}(1/2))$ by Lemma \ref{ForY12andZ}. Thus
$\mathcal{Y}_\Phi|_{iE_\one+\check{\mathfrak{b}}(1)_\mathbb{C}\oplus \check{\mathfrak{b}}(1/2)}\in\mathfrak{X}(\mathcal{D}(\check{\Omega},\check{Q}))$. Let $\partial$ be the vector field on $\mathcal{D}(\Omega,Q)$ defined in Section \ref{Gradingstr}. Since 
\begin{equation*}
[(A_2+\cdots +A_r)^ \#, \mathcal{Y}_\Phi]=[(A_\one+\cdots +A_r)^\#,\mathcal{Y}_\Phi]=[\partial, \mathcal{Y}_\Phi]=\tfrac{1}{2}\mathcal{Y}_\Phi,
\end{equation*}
we have $\mathcal{Y}_\Phi|_{iE_\one+\check{\mathfrak{b}}(1)_\mathbb{C}\oplus \check{\mathfrak{b}}(1/2)}\in\mathfrak{X}(\mathcal{D}(\check{\Omega},\check{Q}))(1/2)$. It remains to show that the map $\check{\mathfrak{X}}(1/2)\ni \mathcal{Y}_\Phi\mapsto \mathcal{Y}_\Phi|_{iE_\one+\check{\mathfrak{b}}(1)_\mathbb{C}\oplus \check{\mathfrak{b}}(1/2)}\in\mathfrak{X}(\mathcal{D}(\check{\Omega},\check{Q}))(1/2)$ is injective. Suppose that $\mathcal{Y}_\Phi|_{iE_\one+\check{\mathfrak{b}}(1)_\mathbb{C}\oplus \check{\mathfrak{b}}(1/2)}=0$. The image of $\mathcal{Y}_\Phi(iE,0)$ under the projection $\mathfrak{b}(1)_\mathbb{C}\oplus\mathfrak{b}(1/2)\ni(U,V)\mapsto V\in\mathfrak{b}(1/2)$ is $i\Phi(E)$, which is equal to $0$. Hence we see from (\ref{psiYPhitil}) that $\psi(\mathcal{Y}_\Phi)=\tilde{\partial}_{-i\Phi(E)}=0$. Since $\psi$ is injective, we obtain $\mathcal{Y}_\Phi=0$.
\end{proof}
\begin{lemma}\label{ForZabinma}
Let $\mathcal{Z}_{a,b}\in\check{\mathfrak{X}}(1)$. Then 
\begin{equation*}
\mathcal{Z}_{a,b}(iE_\one+U,V)\in\check{\mathfrak{b}}(1)_\mathbb{C}\oplus\check{\mathfrak{b}}(1/2)\quad(U\in\check{\mathfrak{b}}(1)_\mathbb{C}, V\in\check{\mathfrak{b}}(1/2)). 
\end{equation*}
\end{lemma}
\begin{proof}
By Lemma \ref{ForZabazzp}, for $U\in\mathfrak{b}(1)_\mathbb{C}$ and $V\in\mathfrak{b}(1/2)$, we have
\begin{equation}\label{[A_\one,a(U,U)]=2a([A_\one,U],U)}
[A_\one,a(U,U)]=2a([A_\one,U],U)
\end{equation}
and
\begin{equation}\label{[A_\one,b(U,V)]=b(U,[A_\one,V])+b([A_\one,U],V)}
[A_\one,b(U,V)]=b([A_\one,U],V)+b(U,[A_\one,V]).
\end{equation}
Put $U=iE_\one$. Then (\ref{[A_\one,a(U,U)]=2a([A_\one,U],U)}) becomes
\begin{equation*}
[A_\one,a(iE_\one,iE_\one)]=2a(iE_\one,iE_\one).
\end{equation*}
Hence we have 
\begin{equation}\label{aiE1iE10}
a(iE_\one,iE_\one)=0.
\end{equation}
Let $U\in\check{\mathfrak{b}}(1)_\mathbb{C}$, and let $V\in\check{\mathfrak{b}}(1/2)$. Then (\ref{[A_\one,a(U,U)]=2a([A_\one,U],U)}) gives
\begin{equation}\label{A1aZZ0}
[A_\one,a(U,U)]=0.
\end{equation}
From (\ref{[A_\one,a(U,U)]=2a([A_\one,U],U)}), (\ref{aiE1iE10}), and (\ref{A1aZZ0}), it follows that
\begin{equation*}\begin{split}
2[A_\one,a(iE_\one,U)]&=[A_\one,a(iE_\one+U,iE_\one+U)]\\&=2a([A_\one,iE_\one+U],iE_\one+U)\\&=2a(iE_\one,U).
\end{split}\end{equation*}
Thus $a(iE_\one,U)\in({\mathfrak{b}_{\alpha_\one}})_\mathbb{C}$. At the same time, for $k\neq 1$, we have $a(E_\one,E_k)=\mathcal{A}_{E_\one}E_k\in\sum_{1\leq m\leq r}^\oplus\mathfrak{b}_{(\alpha_m+\alpha_k)/2}$ by Lemma \ref{ForF}, and hence $\mathcal{A}_{E_\one}E_k=0$. Thus we get $\mathcal{A}_{E_\one}E=0$, which implies $\mathcal{A}_{E_\one}\in\mathfrak{g}_E(\Omega)$. Hence the map $\mathcal{A}_{E_\one}:\mathfrak{b}(1)\rightarrow \mathfrak{b}(1)$ is an orthogonal transformation with respect to the inner product $\langle\cdot, \cdot\rangle'$ on $\mathfrak{b}(0)\oplus\mathfrak{b}(1)$ defined by
\begin{equation*}
\langle X,Y\rangle'=\langle\langle X^\#,Y^\#\rangle\rangle'_{iE}\quad(X,Y\in\mathfrak{b}(0)\oplus\mathfrak{b}(1)).
\end{equation*}
Thus the equality
\begin{equation}\label{aE1ZAE1Z0}
a(E_\one,U)=\mathcal{A}_{E_\one}U=0
\end{equation}
follows from the equation $\mathcal{A}_{E_\one}^2U=0$. By (\ref{aiE1iE10}), (\ref{A1aZZ0}), and (\ref{aE1ZAE1Z0}), we get
\begin{equation}\label{aiE1checkU}
a(iE_\one+U,iE_\one+U)=a(U,U)+2a(iE_\one, U)=a(U,U)\in\check{\mathfrak{b}}(1)_\mathbb{C}.
\end{equation}
On the other hand, by (\ref{[A_\one,b(U,V)]=b(U,[A_\one,V])+b([A_\one,U],V)}), we have
\begin{equation}\label{A1biE1ZWbi}
[A_\one,b(iE_\one+U,V)]=b(iE_\one,V)
\end{equation}
and
\begin{equation}\label{A1bZW0}
[A_\one,b(U,V)]=0.
\end{equation}
Now (\ref{A1biE1ZWbi}) implies $b(iE_\one,V)=0$, and (\ref{A1bZW0}) implies $b(U,V)\in\check{\mathfrak{b}}(1/2)$. Thus 
\begin{equation}\label{biE1checkU}
b(iE_\one+U,V)=b(U,V)\in\check{\mathfrak{b}}(1/2).
\end{equation}
We see from (\ref{aiE1checkU}) and (\ref{biE1checkU}) that (i) follows.
\end{proof}

\begin{lemma}\label{forx=-1}
Consider $\mathcal{D}(\check{\Omega},\check{Q})$ as a complex submanifold of $\mathcal{D}(\Omega,Q)$. Then for any $\gamma\in\{-1,-1/2,0,1/2,1\}$ and $X\in\check{\mathfrak{X}}(\gamma)$, we have 
\begin{equation*}
X|_{iE_\one+\check{\mathfrak{b}}(1)_\mathbb{C}\oplus \check{\mathfrak{b}}(1/2)}\in\mathfrak{X}(\mathcal{D}(\check{\Omega},\check{Q}))(\gamma). 
\end{equation*}
\end{lemma}
\begin{proof}
Let $U_0\in\mathfrak{b}(1)$, and suppose $[\partial_{U_0},A_\one^\#]=0$. Then $[A_\one,U_0]=0$ by (\ref{XABUAU}), and we have $U_0\in\check{\mathfrak{b}}(1)$. Since
\begin{equation*}
\partial_{U_0}(iE_\one+U,V)=(U_0,0)\quad(U\in\check{\mathfrak{b}}(1)_\mathbb{C}, V\in\check{\mathfrak{b}}(1/2)),
\end{equation*}
the result for $\gamma=-1$ follows. Let $V_0\in\mathfrak{b}(1/2)$, and suppose $[\tilde{\partial}_{V_0},A_\one^\#]=0$. Then $[A_\one,V_0]=0$ by (\ref{XABVBV}), and $V_0\in\check{\mathfrak{b}}(1/2)$. Since
\begin{equation*}
\tilde{\partial}_{V_0}(iE_\one+U,V)=(2iQ(V,V_0),V_0)\quad(U\in\check{\mathfrak{b}}(1)_\mathbb{C}, V\in\check{\mathfrak{b}}(1/2)),
\end{equation*}
the result for $\gamma=-1/2$ follows. 
Let $\mathcal{X}(\mathcal{A},\mathcal{B})\in\mathfrak{X}(0)$, and suppose $[\mathcal{X}(\mathcal{A},\mathcal{B}),A_\one^\#]=0$. Then we see from (\ref{XABABXAABB}) that 
\begin{equation*}
\mathcal{A}\mathrm{ad}(A_\one)=\mathrm{ad}(A_\one)\mathcal{A}\,\text{ and }\,\mathcal{B}\mathrm{ad}(A_\one)=\mathrm{ad}(A_\one)\mathcal{B}.
\end{equation*}
Hence we have
\begin{equation*}
[A_\one,\mathcal{A}U]=[A_\one, \mathcal{B}V]=0\quad(U\in\check{\mathfrak{b}}(1)_\mathbb{C}, V\in\check{\mathfrak{b}}(1/2)). 
\end{equation*}
Thus $\mathcal{A}U\in\check{\mathfrak{b}}(1)_\mathbb{C}, \mathcal{B}V\in\check{\mathfrak{b}}(1/2)$ for all $U\in\check{\mathfrak{b}}(1)_\mathbb{C}$ and $V\in\check{\mathfrak{b}}(1/2)$. Now suppose $[\mathcal{X}(\mathcal{A},\mathcal{B}),E_\one^\#]=0$. Then $\mathcal{A}E_\one=0$ by (\ref{XABUAU}). We have
\begin{equation*}
\mathcal{X}(\mathcal{A},\mathcal{B})(iE_\one+U,V)=(\mathcal{A}(iE_\one+U),\mathcal{B}V)=(\mathcal{A}U,\mathcal{B}V)\quad(U\in\check{\mathfrak{b}}(1)_\mathbb{C}, V\in\check{\mathfrak{b}}(1/2)),
\end{equation*}
which shows the assertion for $\gamma=0$. We have shown the result for $\gamma=1/2$ in Lemma \ref{thefollowing}. From Lemma \ref{ForZabinma} and the same arguments as in Lemma \ref{thefollowing}, the result for $\gamma=1$ follows. This completes the proof.
\end{proof}

\begin{remark}\label{Theequalit}
The equality 
\begin{equation*}
\check{\mathfrak{f}}=\check{\mathfrak{f}}(-1)\oplus\check{\mathfrak{f}}(-1/2)\oplus\check{\mathfrak{f}}(0)\oplus\check{\mathfrak{f}}(1/2)\oplus\check{\mathfrak{f}}(1)
\end{equation*}
shows that for any $X\in\check{\mathfrak{X}}$, we have $X|_{iE_\one+\check{\mathfrak{b}}(1)_\mathbb{C}\oplus \check{\mathfrak{b}}(1/2)}\in\mathfrak{X}(\mathcal{D}(\check{\Omega},\check{Q}))$, and the map
\begin{equation*}
\check{\mathfrak{X}}\ni X \mapsto X|_{iE_\one+\check{\mathfrak{b}}(1)_\mathbb{C}\oplus \check{\mathfrak{b}}(1/2)}\in \mathfrak{X}(\mathcal{D}(\check{\Omega},\check{Q}))
\end{equation*}
defines a Lie algebra homomorphism.
\end{remark}
\begin{proposition}\label{forasubalg}
For any $\mathcal{Y}_\Phi\in\mathfrak{f}$, one has $ \mathcal{Y}_{j\Phi}\in\mathfrak{f}$.
\end{proposition}
\begin{proof}
We show the assertion by induction on rank $r$ of the normal $j$-algebra. For the case $r=1$, we have shown the assertion in Proposition \ref{Whenr1fora}. Let $r\geq 2$. We define a $\mathbb{C}$-linear map $\Phi':\mathfrak{b}(1)_\mathbb{C}\rightarrow\mathfrak{b}(1/2)$ by 
\begin{equation*}
\mathcal{Y}_{\Phi'}=[\mathcal{Y}_\Phi,A_\one^\#].
\end{equation*}
By Lemma \ref{foranyY} (ii), we have
\begin{equation*}
[\mathcal{Y}_{\Phi'},A_k^\#]=[[\mathcal{Y}_\Phi,A_\one^\#],A_k^\#]=0\quad(2\leq k\leq r).
\end{equation*}
Thus
\begin{equation*}\begin{split}
[\mathcal{Y}_{\Phi+2\Phi'},A_\one^\#]&=\mathcal{Y}_{\Phi'}+[\mathcal{Y}_{2\Phi'},A_\one^\#]=\mathcal{Y}_{\Phi'}+[\mathcal{Y}_{2\Phi'},(A_\one+\cdots+A_r)^\#]\\&=\mathcal{Y}_{\Phi'}+[\mathcal{Y}_{2\Phi'},\partial]=\mathcal{Y}_{\Phi'}-\mathcal{Y}_{\Phi'}=0.
\end{split}\end{equation*}
From Lemma \ref{ForY12andZ}, it follows that $\mathcal{Y}_{\Phi+2\Phi'}\in\check{\mathfrak{f}}$. We denote by $R$ the Lie algebra homomorphism in Remark \ref{Theequalit}. 
Since $(\check{\mathfrak{b}})^\#\subset \check{\mathfrak{f}}$, one has $R((\check{\mathfrak{b}})^\#)\subset R(\check{\mathfrak{f}})\subset\mathfrak{X}(\mathcal{D}(\check{\Omega},\check{Q}))$. By the inductive hypothesis and the equality $R(\mathcal{Y}_{\Phi+2\Phi'})=\mathcal{Y}_{({\Phi+2\Phi'}|_{\check{\mathfrak{b}}(1)_\mathbb{C}+\check{\mathfrak{b}}(1/2)})}$, we get \begin{equation*}
\mathcal{Y}_{(j({\Phi+2\Phi'})|_{\check{\mathfrak{b}}(1)_\mathbb{C}+\check{\mathfrak{b}}(1/2)})}\in R(\check{\mathfrak{f}}). 
\end{equation*}
Hence we see from Lemma \ref{thefollowing} that
\begin{equation*}
\mathcal{Y}_{j(\Phi+2\Phi')}\in\check{\mathfrak{f}}.
\end{equation*} 
Clearly
\begin{equation*}
\mathcal{Y}_{j\Phi}=\mathcal{Y}_{j(\Phi+2\Phi')}-2\mathcal{Y}_{j\Phi'}.
\end{equation*}
Thus in order to show that $Y_{j\Phi}\in \mathfrak{f}$, it suffices to show that
\begin{equation}\label{YiPsi1inmathfrakg12}\mathcal{Y}_{j\Phi'}\in\mathfrak{f}.
\end{equation}
Next we prove \eqref{YiPsi1inmathfrakg12}. Let $c':\mathfrak{b}(1/2)\times\mathfrak{b}(1/2)\rightarrow \mathfrak{b}(1/2)$ be the $\mathbb{C}$-bilinear map such that $\mathcal{Y}_{\Phi'}=\mathcal{Y}_{\Phi',c'}$. We define $\mathbb{C}$-linear maps $\Phi'':\mathfrak{b}(1)_\mathbb{C}\rightarrow \mathfrak{b}(1/2)$, $\mathcal{A}\in\mathfrak{gl}(\mathfrak{b}(1)_\mathbb{C})$, and $\mathcal{B}\in\mathfrak{gl}(\mathfrak{b}(1/2))$ by 
\begin{equation*}
\mathcal{X}(\mathcal{A},\mathcal{B})=[\tilde{\partial}_{\Phi'(E)},\mathcal{Y}_{\Phi'}],\,\text{ and }\,\mathcal{Y}_{\Phi''}=[\mathcal{Y}_{\Phi'},\mathcal{X}(\mathcal{A},\mathcal{B})].
\end{equation*} 
The inclusion relation $\mathfrak{b}^\#\subset \mathfrak{f}$ shows that $\mathcal{Y}_{\Phi''}\in\mathfrak{f}$. Put $V_0=\Phi'(E)$. If $V_0=0$, then $\mathcal{Y}_{\Phi'}=0$, and it follows that $\Phi'=0$. Thus $0=\mathcal{Y}_{j\Phi'}\in\mathfrak{f}$. In what follows, we assume that $V_0\neq 0$. Lemma \ref{ForXABinma} (i) shows that $\Phi''(E)=-\Phi'(\mathcal{A}E)+\mathcal{B}\Phi'(E)$. And by (\ref{tildepartialVYPhiXAB}), we have
\begin{equation}\label{AE4Psi1V_0E4}
\mathcal{A}E=4\Phi'_{V_0}(E)=4\Im Q(V_0,V_0)=0
\end{equation}
and
\begin{equation}\label{BPsi1E2iPs}
\mathcal{B}\Phi'(E)=2j\Phi'(Q(V_0,V_0))+2c'(V_0,V_0).
\end{equation}
From Lemma \ref{ForXABinma} (i) and Lemma \ref{foranyY} (ii), it follows that
\begin{equation*}
V_0=\Phi'(E)=-\Phi([A_\one,E])+[A_\one,\Phi(E)]=-\tfrac{1}{2}\Phi(E_\one)\in\mathfrak{b}_{\alpha_\one/2}.
\end{equation*}
Thus we see from (\ref{HVcVV2iHPhiHVVVquadVVinmathcalV}) that 
\begin{equation*}\begin{split}
Q(c'(V_0,V_0),V)=2iQ(V_0,\Phi'(Q(V,V_0)))\in\Bigl(\sideset{}{^\oplus}\sum_{1\leq m\leq r}\mathfrak{b}_{(\alpha_m+\alpha_\one)/2}\Bigr)_\mathbb{C}\\(V\in\mathfrak{b}(1/2)).
\end{split}\end{equation*}
By Lemma \ref{ForV}, we have $c'(V_0,V_0)\in\mathfrak{b}_{\alpha_\one/2}$. We define a Hermitian form $q$ on $\mathfrak{b}_{\alpha_\one/2}$ by
\begin{equation*}
Q(W,W')=q(W,W')E_\one\quad(W,W'\in\mathfrak{b}_{\alpha_\one/2}). 
\end{equation*}
By (\ref{HVcVV2iHPhiHVVVquadVVinmathcalV}), for any $V\in\mathfrak{b}_{\alpha_\one/2}$, we have
\begin{equation*}\begin{split}
Q(c'(V_0,V_0),V)&=2iQ(V_0,\Phi'(Q(V,V_0)))=2iQ(V_0,q(V,V_0)\Phi'(E))
\\&=2iQ(V_0,q(V,V_0)V_0)=2iq(V_0,V)Q(V_0,V_0)
\\&=Q(2jq(V_0,V_0)V_0,V).
\end{split}\end{equation*}
Thus we get \begin{equation}\label{c1V0V02jqV}
 c'(V_0,V_0)=2jq(V_0,V_0)V_0.
\end{equation} 
From (\ref{AE4Psi1V_0E4}), (\ref{BPsi1E2iPs}), and (\ref{c1V0V02jqV}), it follows that
\begin{equation*}\begin{split}
\Phi''(E)&=\mathcal{B}\Phi'(E)=2j\Phi'(Q(V_0,V_0))+4jq(V_0,V_0)V_0=6jq(V_0,V_0)V_0
\\&=6jq(V_0,V_0)\Phi'(E)
\end{split}\end{equation*}
with $q(V_0,V_0)\neq 0$. Thus \eqref{YiPsi1inmathfrakg12} holds, and the proof is complete.
\end{proof}

\begin{proposition}\label{Forasubalg2}
Let $\partial_{U_0}\in\mathfrak{X}(-1)$. If $[\partial_{U_0},\mathfrak{f}(1)]=0$, then one has $[\partial_{U_0},\mathfrak{f}(1/2)]=0$.
\end{proposition}
\begin{proof}
We replace $\mathfrak{X}(\gamma)$ by $\mathfrak{f}(\gamma)$ for $\gamma=-1,1/2,1$ in the proof of Lemma \ref{Letpartial}. To complete the proof, it is enough to show that $\mathcal{Y}_{j\Phi}\in\mathfrak{f}(1/2)$. Hence the result follows from Proposition \ref{forasubalg}.
\end{proof}

\section{Unitary equivalences among the unitarizable representations}\label{Unitaryequ}
\subsection{Unitary equivalences among representations of $B$}\label{Constructi}
We take $(iE,0)\in\mathcal{D}(\Omega,Q)$ as a reference point of $\mathcal{D}(\Omega,Q)$. Let $M: G\times \mathcal{D}(\Omega,Q)\rightarrow\mathbb{C}^\times$ be a holomorphic multiplier. Set $\mathfrak{b}_-=\mathfrak{g}_-\cap \mathfrak{b}_\mathbb{C}$.
\begin{proposition}[Rossi and Vergne, {\cite[Proposition 4.21]{Rossi}}]
Let $\tau:\mathfrak{b}\rightarrow \mathfrak{b}_-$ be the $\mathbb{R}$-linear map defined by 
\begin{equation*}
\tau(U+V+T)=(V+ijV)/2+T+ijT\quad(U\in\mathfrak{b}(1),V\in\mathfrak{b}(1/2),T\in\mathfrak{b}(0)).
\end{equation*}
Then $\tau$ is a Lie algebra homomorphism, and if $\tau$ is extended to a $\mathbb{C}$-linear map $\tau:\mathfrak{b}_\mathbb{C}\rightarrow\mathfrak{b}_\mathbb{C}$, then we have $\tau|_{\mathfrak{b}_-}=\mathrm{id}_{\mathfrak{b}_-}$. 
\end{proposition}

Put $\theta=\theta_M\in(\mathfrak{g}_-)^*$. 
\begin{theorem}[Ishi, {\cite[Theorem 12]{ishi 2011}}]
Let $\chi^\theta$ be the function on $B$ defined by
\begin{equation}\label{chithetaex}
\chi^\theta(\exp X)=e^{\theta\circ \tau(X)}\quad(X\in\mathfrak{b}).
\end{equation}
Then the function
\begin{equation*}
B\times\mathcal{D}(\Omega,Q)\ni(b,(U,V))\mapsto\chi^\theta(b)\in\mathbb{C}^\times
\end{equation*}
is a holomorphic multiplier and is $B$-equivalent to $M$.
\end{theorem}

By Lemma \ref{equivalentline}, there exists a holomorphic function $f: \mathcal{D}(\Omega,Q)\rightarrow \mathbb{C}^\times$ such that the equality 
\begin{equation*}
\chi^{\theta}(b)=f(b(U,V))M(b,(U,V))f(U,V)^{-1}\quad(b\in B, (U,V)\in\mathcal{D}(\Omega,Q))
\end{equation*}
holds. We define a holomorphic multiplier $M_\theta:G\times\mathcal{D}(\Omega,Q)\rightarrow \mathbb{C}^\times$ by 
\begin{equation*}
M_\theta(g,(U,V))=f(g(U,V))M(g,(U,V))f(U,V)^{-1}\quad(g\in G, (U,V)\in\mathcal{D}(\Omega,Q)).
\end{equation*}
Then holomorphic multipliers $M$ and $M_\theta$ are $G$-equivalent by Lemma \ref{equivalentline}. Now we see that
\begin{equation*}
\dt M (e^{tX},(iE,0))=\dt M_{\theta}(e^{tX},(iE,0))\quad(X\in\mathfrak{k}).
\end{equation*}
Since the maps $K\ni k\mapsto M(k,(iE,0))\in\mathbb{C}^\times$ and $K\ni k\mapsto M_{\theta}(k,(iE,0))\in\mathbb{C}^\times$ define one-dimensional representations of $K$, we have
\begin{equation*}
M(k,(iE,0))=M_{\theta}(k,(iE,0))\quad(k\in \GiE). 
\end{equation*} 
Clearly, we also have 
\begin{equation*}
M_\theta(b,(U,V))=\chi^\theta(b)\quad(b\in B).
\end{equation*}
Now we assume that the representation $T_M$ is unitarizable. Let $\xi\in\mathfrak{g}^*$ be the linear form given by \eqref{xiJKcdotpi}. We denote the unitarization of the representation $\chi^{i\xi}$ by $(\chi^{i\xi},\mathcal{H}_\xi)$, and we denote the reproducing kernel of $\mathcal{H}_\xi$ by $\mathcal{K}^\xi$.

We consider the unitary representations $(T_{\chi^{i\xi'}},\mathcal{H}_{\xi'})$ which are obtained by holomorphic multipliers $M':G\times\mathcal{D}(\Omega,Q)\rightarrow\mathbb{C}^\times$. We review the construction of the intertwining operators among the representations $(T_{\chi^{i\xi'}},\mathcal{H}_{\xi'})$ of $B$ in \cite{ishi 1999, ishi 2001}. The group $B(0)$ acts on $\mathfrak{b}(1)^*$ by 
\begin{equation*}
\langle U,t_0\ell\rangle=\langle t_0^{-1}U,\ell\rangle\quad(U\in\mathfrak{b}(1), t_0\in B(0), \ell\in\mathfrak{b}(1)^*). 
\end{equation*} 

\begin{theorem}[Ishi, \cite{ishi 1999}]
There exists a unique $\mathrm{Ad}(B(0))$-orbit $\mathcal{O}_\xi^*\subset \mathfrak{b}(1)^*$ and a unique measure $d\nu_\xi$ on $\mathcal{O}_\xi^*$ such that 
\begin{equation*}\label{dnuxutFchi}
d\nu_{\xi}(t_0\ell)=|\chi^{\xi}(t_0)|^2d\nu_{\xi}(\ell)\quad(\ell\in\mathcal{O}_\xi^*, t_0\in B(0)),
\end{equation*}
\begin{equation*}
\int_{\mathcal{O}_\xi^*}e^{-\langle U,\ell\rangle}\,d\nu_\xi(\ell)<\infty\text{ for all } U\in\Omega.
\end{equation*}
If $\chi^{i\xi}$ and $\chi^{i\xi'}$ define equivalent unitarizations, then $\mathcal{O}_\xi^*=\mathcal{O}_{\xi'}^*$.
\end{theorem}
In \cite{ishi 1999}, $\mathcal{O}_\xi^*$ and $d\nu_\xi$ are written as $\mathcal{O}_\varepsilon^*$ and $d\mathcal{R}_{\Re s^*}^*$, respectively. The dual cone $\Omega^*\subset\mathfrak{b}(1)^*$ of $\Omega$ is defined by
\begin{equation*}
\Omega^*=\{\ell\in\mathfrak{b}(1)^*;\langle U,\ell\rangle> 0\text{ for all }U\in\overline{\Omega}\backslash\{0\}\}.
\end{equation*}
For $\ell\in\overline{\Omega^*}$, let $Q_\ell$ be the Hermitian form on $\mathfrak{b}(1/2)$ given by 
\begin{equation*}
Q_\ell(V,V')=\langle 2Q(V,V'),\ell\rangle\quad(V,V'\in\mathfrak{b}(1/2)).
\end{equation*} 
Then $Q_\ell$ is positive definite. Let 
\begin{equation*}
N_\ell=\{V\in\mathfrak{b}(1/2);Q_\ell(V,V)=0\}, 
\end{equation*}
and let $\mathcal{F}_\ell$ be the space of holomorphic functions $F$ on $\mathfrak{b}(1/2)$ such that
\begin{enumerate}
 \item[(i)]$F(V+V')=F(V)$ for all $V\in \mathfrak{b}(1/2)$ and $V'\in N_\ell$,
\item[(ii)]$\|F\|_{\mathcal{F}_\ell}^2=\int_{\mathfrak{b}(1/2)/N_\ell} |F(V)|^2e^{-Q_\ell(V,V)}\,d\mu_\ell([V])<\infty$, 
\end{enumerate}
where $d\mu_\ell$ denotes the Lebesgue measure on $\mathfrak{b}(1/2)/N_\ell$ normalized in such a way that 
\begin{equation*}
\|1\|_{\mathcal{F}_\ell}=1. 
\end{equation*}
Let $\mathcal{L}_{\xi}$ be the function space consists of all equivalence classes of measurable functions $f$ on $\mathcal{O}_\xi^*\times \mathfrak{b}(1/2)$ such that 
\begin{enumerate}
\item[(i)]
$f(\ell,\cdot)\in\mathcal{F}_\ell$ for almost all $\ell\in\mathcal{O}_\xi^*$ with respect to the measure $d\nu_{\xi}$,
\item[(ii)]
$\|f\|_{\mathcal{L}_\xi}^2=\int_{\mathcal{O}^*_\xi}\|f(\ell,\cdot)\|_{\mathcal{F}_\ell}^2d\nu_{\xi}(\ell)<\infty$.
\end{enumerate}
\begin{theorem}[Ishi, {\cite[Theorem 4.10]{ishi 1999}}]\label{Fouriertra}
The map $\phi_{\xi}:\mathcal{L}_{\xi}\rightarrow\mathcal{H}_{\xi}$ defined by
\begin{equation*}
\phi_{\xi} f(U,V)=\int_{\mathcal{O}_\xi^*}e^{i\langle U,\ell\rangle}f(\ell,V)\,d\nu_{\xi}(\ell)\quad((U,V)\in\mathcal{D}(\Omega,Q))
\end{equation*}
gives a Hilbert space isomorphism.
\end{theorem}

We define a unitary representation $\check{T}_{\chi^{i\xi}}$ of $B$ on $\mathcal{L}_\xi$ by 
\begin{equation*}
\phi_\xi(\check{T}_{\chi^{i\xi}}(b)f)=T_{\chi^{i\xi}}(b)\phi_\xi(f)\quad(b\in B, f\in \mathcal{L}_\xi). 
\end{equation*}
For $(U_0,V_0)\in\mathcal{D}(\Omega,Q)$, let
\begin{equation*}
k_{(U_0,V_0)}(\ell,V)=e^{-i\langle \overline{U_0}, \ell\rangle}e^{Q_\ell(V,V_0)}\quad((\ell,V)\in\mathcal{O}_\xi^*\times\mathfrak{b}(1/2)).
\end{equation*}
Then $k_{(U_0,V_0)}\in\mathcal{L}_{\xi}$, and we have the following equalities (see {\cite[p. 450]{ishi 1999}}):
\begin{equation}\label{PhixifZWfk}
\phi_{\xi} f(U,V)=(f|k_{(U,V)})_{\mathcal{L}_{\xi}}\quad((U,V)\in\mathcal{D}(\Omega,Q), f\in\mathcal{L}_{\xi}),
\end{equation}
\begin{equation}\label{kUVkZWmath}\begin{split}
(k_{(U',V')}|k_{(U,V)})_{\mathcal{L}_{\xi}}=\mathcal{K}^{\xi}((U,V),(U',V'))
\\((U,V),(U',V')\in\mathcal{D}(\Omega,Q)).
\end{split}\end{equation}

Suppose that the unitarizations of $T_{\chi^{i\xi}}$ and $T_{\chi^{i\xi'}}$ are equivalent as unitary representations of $B$. As in {\cite[p. 541]{ishi 2001}}, we fix a function $\Upsilon\neq 0$ on $\mathcal{O}_{\xi}^*$, which is also a function on $\mathcal{O}_{\xi'}^*$, such that
\begin{equation*}
\Upsilon(t_0\ell)=\overline{\chi^{i\xi}(t_0)\chi^{-i\xi'}(t_0)}\Upsilon(\ell)\quad(t_0\in B(0),\ell\in\mathcal{O}_{\xi}^*).
\end{equation*}
\begin{proposition}[Ishi, {\cite[Proposition 4.5]{ishi 2001}}]\label{Thefollowingm}
There exists a nonzero constant $C$ such that the following map $\check{\Psi}_{\xi,\xi'}:\mathcal{L}_{\xi}\rightarrow\mathcal{L}_{\xi'}$ gives the intertwining operator between the unitary representations $(\check{T}_{\chi^{i\xi}},\mathcal{L}_{\xi})$ and $(\check{T}_{\chi^{i\xi'}},\mathcal{L}_{\xi'})$ of $B$:
\begin{equation*}
\check{\Psi}_{\xi,\xi'}f(\ell,V)=C\Upsilon(\ell)f(\ell,\,V)\quad(f\in\mathcal{L}_{\xi},(\ell,V)\in \mathcal{O}_{\xi}^*\times\mathfrak{b}(1/2)).
\end{equation*}
\end{proposition}

Let $\Delta_{\xi}$ and $\Delta_{\xi,\xi'}$ be the functions on $\Omega$ define by
\begin{equation*}
\Delta_{\xi}(t_0E)=|\chi^{i\xi}(t_0)|^2,\quad \Delta_{\xi,\xi'}(t_0E)=\overline{\chi^{i\xi}(t_0)\chi^{-i\xi'}(t_0)}\quad(t_0\in B(0)).
\end{equation*}
\begin{proposition}[Ishi, {\cite[Corollary 2.5 and Proposition 4.6]{ishi 1999}}]
Two functions $\Delta_{\xi}$ and $\Delta_{\xi,\xi'}$ extend to functions on $\Omega+i\mathfrak{b}(1)$ holomorphically, and for $(U,V),(U',V')\in\mathcal{D}(\Omega,Q)$, we have
\begin{equation*}
\mathcal{K}^{\xi}((U,V),(U',V'))=\Delta_{\xi}\left(\tfrac{U-\overline{U'}}{i}-2Q(V,V')\right).
\end{equation*}
\end{proposition}

\subsection{Unitary equivalences among representations of $G$}\label{Intertwini}
Suppose that the unitarizations $(T_{M_{i\xi}},\mathcal{H}_{\xi})$ and $(T_{M_{i\xi'}},\mathcal{H}_{\xi'})$ are equivalent as unitary representations of $B$. By Schur's lemma and the decomposition $G=BK$, the unitarizations are equivalent as unitary representations of $G$ if and only if the intertwining operator between representations $(T_{M_{i\xi}}|_B,\mathcal{H}_{\xi})$ and $(T_{M_{i\xi'}}|_B,\mathcal{H}_{\xi'})$ preserves the actions of $\GiE$. From this point of view, we get the equation (\ref{diff}) in Proposition \ref{Thefollowingc} below which determines whether the unitarizations are equivalent as unitary representations of $G$.

From now on, we assume that $(T_{M_{i\xi}}|_B,\mathcal{H}_{\xi})$ and $(T_{M_{i\xi'}}|_B,\mathcal{H}_{\xi'})$ are equivalent as unitary representations of $B$. Let $\Psi_{\xi,\xi'}=\phi_{\xi'}\circ\check{\Psi}_{\xi,\xi'}\circ\phi_{\xi}^{-1}:\mathcal{H}_{\xi}\rightarrow\mathcal{H}_{\xi'}$.
\begin{lemma}\label{prop:2}
There exists a nonzero constant $C'$ such that 
\begin{equation*}
\begin{split}
\Psi_{\xi,\xi'}(\mathcal{K}_{(iE,0)}^{\xi})(U,V)&=C'\mathcal{K}_{(iE,0)}^{\xi'}(U,V)\Delta_{\xi,\xi'}\left(\frac{U-\overline{iE}}{i}\right)\quad((U,V)\in\mathcal{D}(\Omega,Q)).
\end{split}
\end{equation*}
\end{lemma}
\begin{proof}
Put $\mathcal{K}'=\Psi_{\xi,\xi'}(\mathcal{K}_{(iE,0)}^{\xi})$. By (\ref{PhixifZWfk}) and (\ref{kUVkZWmath}), for $(U,V)\in\mathcal{D}(\Omega,Q)$, we have
\begin{equation*}
\phi_{\xi}(k_{(iE,0)})(U,V)=(k_{(iE,0)}|k_{(U,V)})_{\mathcal{L}_{\xi}}=\mathcal{K}^{\xi}((U,V),(iE,0))=\mathcal{K}_{(iE,0)}^{\xi}(U,V).
\end{equation*}
Thus 
\begin{equation*}
\mathcal{K}'=\phi_{{\xi'}}\circ\check{\Psi}_{{\xi},{\xi'}}(k_{(iE,0)}).
\end{equation*}
Hence by Theorem \ref{Fouriertra} and Proposition \ref{Thefollowingm}, we have
\begin{equation*}\begin{split}
\mathcal{K}'(U,V)&=C\int_{\mathcal{O}_{\xi'}^*}e^{i\langle U,\ell\rangle}k_{(iE,0)}(\ell,V)\Upsilon(\ell)\,d\nu_{\xi'}(\ell)
\\&=C\int_{\mathcal{O}_{\xi'}^*}e^{i\langle U-\overline{iE},\ell\rangle}\Upsilon(\ell)\,d \nu_{\xi'}(\ell).
\end{split}\end{equation*}
When $(U,V)=(iU_0,V)$ with $U_0\in \Omega$, we see that
\begin{equation*}
i\langle U-\overline{iE},\ell\rangle=i\langle iU_0+iE,\ell\rangle=-\langle U_0+E, \ell\rangle.
\end{equation*}
It follows that $U_0+E\in\Omega$ from $U_0-Q(V,V)\in\Omega$ and $Q(V,V)\in\overline{\Omega}$ (see Remark \ref{Wenotethat}). Thus there exists $t_0\in B(0)$ such that $t_0E=U_0+E$. Then 
\begin{equation*}\begin{split}
\mathcal{K}'(iU_0,V)&=C\int_{\mathcal{O}_{\xi'}^*}e^{-\langle t_0\cdot E,\ell\rangle}\Upsilon(\ell)\,d\nu_{\xi'}(\ell)
\\&=C|\chi^{i\xi'}(t_0)|^2\int_{\mathcal{O}_{\xi'}^*}e^{-\langle E, \ell\rangle}\Upsilon(t_0\ell)\,d\nu_{\xi'}(\ell)
\\&=C|\chi^{i\xi'}(t_0)|^2\overline{\chi^{i\xi}(t_0)\chi^{-i\xi'}(t_0)}\int_{\mathcal{O}_{\xi'}^*} e^{-\langle E,\ell\rangle}\Upsilon(\ell)\,d\nu_{\xi'}(\ell)
\\&=C'\Delta_{\xi'}(U_0+E)\Delta_{\xi,\xi'}(U_0+E)
\\&=C'\Delta_{\xi'}\left(\frac{iU_0+iE}{i}\right)\Delta_{\xi,\xi'}\left(\frac{iU_0+iE}{i}\right),
\end{split}\end{equation*} 
where we put $C'=C\int_{\mathcal{O}_{\xi'}^*} e^{-\langle E,\ell\rangle}\Upsilon(\ell)\,d\nu_{\xi'}(\ell)$. By the analytic continuation, we have
\begin{equation*}\begin{split}
\mathcal{K}'(U,V)=C'\mathcal{K}_{(iE,0)}^{{\xi'}}(U,V)\Delta_{\xi,\xi'}\left(\frac{U-\overline{iE}}{i}\right)\quad((U,V)\in\mathcal{D}(\Omega,Q)).
\end{split}\end{equation*}
\end{proof}
\begin{remark}\label{Wenotethat}
Let $\Omega_0\subset \mathbb{R}^{N_0}$ be an open convex cone, let $v\in \Omega_0$, and let $v'\in\overline{\Omega_0}$. Then it follows that $v+v'\in \Omega_0$. Indeed, let $B_\epsilon(v)$ be an open ball of radius $\epsilon >0$ centered at $v$ satisfying $B_\epsilon(v)\subset \Omega_0$. Then we have $v_0+v'\in\overline{\Omega_0}$ for all $v_0\in B_\epsilon(v)$. Hence we obtain $v+v'\in \mathrm{Int}(\overline{\Omega_0})$. It is known that for a convex set $S_0\subset \mathbb{R}^{N_0}$, the equality $\mathrm{Int}(\overline{S_0})=\mathrm{Int}(S_0)$ holds. Thus we have $v+v'\in \Omega_0$ since $\Omega_0$ is a convex set.
\end{remark}

If the unitarizations $(T_{M_{i\xi}},\mathcal{H}_{\xi})$ and $(T_{M_{i\xi'}},\mathcal{H}_{\xi'})$ are equivalent as unitary representations of $G$, then the equality
\begin{equation}\label{eq:1}
T_{M_{i\xi'}}(\con)\Psi_{\xi,\xi'}(\mathcal{K}_{(iE,0)}^{\xi})=\Psi_{\xi,\xi'}(T_{M_{i\xi}}(\con)\mathcal{K}_{(iE,0)}^{\xi})\quad(\con\in \GiE)
\end{equation}
holds. The converse is also true as we shall see in the next proposition. In what follows, we put $(U(g),V(g))=g(U,V)$ for $g\in G$ and $(U,V)\in\mathcal{D}(\Omega,Q)$.
\begin{proposition}\label{Thefollowingc}
The following are equivalent:
\begin{enumerate}
 \item[$(\mathrm{i})$]
the unitarizations of $T_{M}$ and $T_{M'}$ are equivalent as unitary representations of $G$,
\item[$(\mathrm{ii})$]
$(\ref{eq:1})$ holds,
\item[$(\mathrm{iii})$] the following equality holds:
\begin{equation}\label{diff}\begin{split}
\Delta_{{\xi},{\xi'}}\left(\dfrac{U(\con^{-1})-\overline{iE}}{i}\right)=MM'^{-1}(\con,(iE,0))
\Delta_{{\xi},{\xi'}}\left(\frac{U-\overline{iE}}{i}\right)
\\(\con\in \GiE, (U,V)\in\mathcal{D}(\Omega, Q)).
\end{split}\end{equation}
\end{enumerate}\end{proposition}
\begin{proof}
First we show that (i) and (ii) are equivalent. Thanks to the remark preceding Proposition \ref{Thefollowingc}, it is enough to show that (ii) implies (i). We suppose that (\ref{eq:1}) holds. Let $b\in B$, and let $\con\in \GiE$. Then we can write $kb=b'k'$ with $b'\in B$ and $k'\in \GiE$, and we have
\begin{equation*}\begin{split}
\Psi&_{\xi,\xi'}(T_{M_{i\xi}}(\con)T_{M_{i\xi}}(b)\mathcal{K}_{(iE,0)}^{\xi})
=\Psi_{\xi,\xi'}(T_{M_{i\xi}}(b')T_{M_{i\xi}}(k')\mathcal{K}_{(iE,0)}^{\xi})
\\&=T_{M_{i\xi'}}(b')\Psi_{\xi,\xi'}(T_{M_{i\xi}}(k')\mathcal{K}_{(iE,0)}^{\xi})
=T_{M_{i\xi'}}(b')T_{M_{i\xi'}}(k')\Psi_{\xi,\xi'}(\mathcal{K}_{(iE,0)}^{\xi})\\&
=T_{M_{i\xi'}}(\con)T_{M_{i\xi'}}(b)\Psi_{\xi,\xi'}(\mathcal{K}_{(iE,0)}^{\xi})=T_{M_{i\xi'}}(\con)\Psi_{\xi,\xi'}(T_{M_{i\xi}}(b)\mathcal{K}_{(iE,0)}^{\xi}).
\end{split}\end{equation*}
Now $\Psi_{\xi,\xi'}$ is continuous, and the subspace of $\mathcal{H}_{\xi}$ generated by $T_{M_{i\xi}}(b)\mathcal{K}_{(iE,0)}^{\xi}\,(b\in B)$ is dense in $\mathcal{H}_{\xi}$. Thus
\begin{equation*}
\Psi_{\xi,\xi'}(T_{M_{i\xi}}(\con)f)=T_{M_{i\xi'}}(\con)\Psi_{\xi,\xi'}(f)\quad(\con\in \GiE, f\in\mathcal{H}_{\xi}),
\end{equation*}
which implies $(T_{M_{i\xi}},\mathcal{H}_{\xi})$ and $(T_{M_{i\xi'}},\mathcal{H}_{\xi'})$ are equivalent as unitary representations of $G$. Thus (i) follows. Next we show that (ii) and (iii) are equivalent. 
By Lemma \ref{prop:2} and the transformation law of the reproducing kernel, for $\con\in \GiE$, we have
\begin{equation*}\begin{split}
{C'}^{-1}&T_{M_{i\xi'}}(\con)\Psi_{{\xi},{\xi'}}(\mathcal{K}_{(iE,0)}^{\xi})(U,V)
\\&
=M_{i\xi'}(\con^{-1},(U,V))^{-1}\mathcal{K}_{(iE,0)}^{{\xi'}}(U(\con^{-1}),V(\con^{-1}))\Delta_{{\xi},{\xi'}}\left(\frac{U(\con^{-1})-\overline{iE}}{i}\right)\\
&=M'(\con,(iE,0))\mathcal{K}_{(iE,0)}^{{\xi'}}(U,V)\Delta_{{\xi},{\xi'}}\left(\frac{U(\con^{-1})-\overline{iE}}{i}\right).
\end{split}
\end{equation*}
For $\con\in \GiE$, we also have
\begin{equation*}\begin{split}
{C'}^{-1}&\Psi_{\xi,\xi'}(
T_{M_{i\xi}}(\con)\mathcal{K}_{(iE,0)}^{\xi}
)(U,V)={C'}^{-1}\Psi_{\xi,\xi'}(M(\con,(iE,0))\mathcal{K}_{(iE,0)}^{\xi})(U,V)
\\&=M(k,(iE,0))\mathcal{K}_{(iE,0)}^{{\xi'}}(U,V)\Delta_{{\xi},{\xi'}}\left(\frac{U-\overline{iE}}{i}\right).
\end{split}\end{equation*}
Thus (\ref{eq:1}) holds if and only if 
\begin{equation*}\begin{split}
\Delta_{{\xi},{\xi'}}\left(\frac{U(\con^{-1})-\overline{iE}}{i}\right)=MM'^{-1}(\con,(iE,0))\Delta_{{\xi},{\xi'}}\left(\frac{U-\overline{iE}}{i}\right)
\\(\con\in \GiE, (U,V)\in\mathcal{D}(\Omega,Q)).
\end{split}\end{equation*}
\end{proof}

\subsection{Isotropy representation}\label{Isotropyre}
In this subsection, we shall consider the isotropy representation $\rho: \GiE\rightarrow GL(\mathfrak{g}/{\giE})$. We identify $\mathfrak{b}$ with $\mathfrak{g}/\giE$ by the map $\mathfrak{b}\ni X\mapsto X+\giE\in\mathfrak{g}/\giE$. We denote by $\omega'\in\mathfrak{b}^*$ the Koszul form on $\mathfrak{b}$ which is defined by 
\begin{equation*}
\omega'(X)=\mathrm{tr}_\mathfrak{b}(\mathrm{ad}(jX)-j\mathrm{ad}(X))\quad(X\in\mathfrak{b}).
\end{equation*}
Then $(\mathfrak{b},j,\omega')$ is a normal $j$-algebra. Put
\begin{equation*}
\langle X,Y\rangle'=\omega'([jX,Y])\quad(X,Y\in\mathfrak{b}),
\end{equation*}
and let $\langle\cdot,\cdot \rangle''$ be the Hermitian form on $\mathfrak{b}$ given by 
\begin{equation*}
\langle X,Y\rangle''=\langle X,Y\rangle'+i\langle X,jY \rangle'\quad(X,Y\in\mathfrak{b}). 
\end{equation*}
Then we can regard $\mathfrak{b}$ as a complex Hilbert space, and $\rho$ is a unitary representation of $K$. Define
\begin{equation*}
 \mathfrak{b}_{triv}=\{X\in\mathfrak{b};[X,\giE]\subset\giE\}.
\end{equation*}
\begin{lemma}\label{IfXinmathf}
For any $X\in\mathfrak{b}_{triv}$, we have $[X,\giE]=\{0\}$.
\end{lemma}
\begin{proof}
By Lemma \ref{mathfrakzm}, we can regard $\mathfrak{g}$ as a subalgebra of $\mathfrak{gl}(\mathfrak{g})$. We denote the connected Lie subgroup of $GL(\mathfrak{g})$ with Lie algebra $\mathfrak{b}\subset \mathfrak{gl}(\mathfrak{g})$ by $B_0$. We put $N'=\dim \mathfrak{g}$. Choose an Iwasawa subgroup $B'$ of $GL(\mathfrak{g})$ which contains $B_0$. By \cite[Chapter 4, Theorem 4.9]{encyclopedia}, the group $B'$ is realized as the subgroup $L(N')\subset GL(N',\mathbb{R})$ of lower triangular matrices with positive diagonal entries in some basis of $\mathfrak{g}$. Hence $\mathfrak{b}$ is realized as a subalgebra of $\mathfrak{l}(N')$. For $X\in\mathfrak{l}(N')$, the linear map $\mathrm{ad}(X):\mathfrak{gl}(N',\mathbb{R})\rightarrow \mathfrak{gl}(N',\mathbb{R})$ has only real eigenvalues. Let $X\in\mathfrak{b}_{triv}$. Then the linear map $\mathrm{ad}(X):\mathfrak{g}\rightarrow\mathfrak{g}$ also has only real eigenvalues. On the other hand, we have $\mathrm{ad}(X)[\giE,\giE]\subset[\giE,\giE]$, and $\mathrm{ad}(X)|_{[\giE,\giE]}$ has only pure imaginary eigenvalues and is diagonalizable. Hence $\mathrm{ad}(X)|_{[\giE,\giE]}=0$. Put \begin{equation*}
\mathfrak{t}=\mathfrak{z}(\giE),\quad \mathbb{T}=\exp \mathfrak{t}\subset G,\quad N''=\dim \mathfrak{t}. 
\end{equation*} 
Then the Lie group $\mathbb{T}$ is isomorphic to 
\begin{equation*}
(S^1)^{N''}=\{(\zeta_\one,\cdots, \zeta_{N''})\in\mathbb{C}^{N''};|\zeta_l|=1\text{ for all }l=1,\cdots,N''\}. 
\end{equation*}
Let $\iso:(S^1)^{N''}\rightarrow\mathbb{T}$ be an isomorphism, and let $t\in\mathbb{R}$. The map
\begin{equation*}
\mathrm{Inn}(e^{tX}):\mathbb{T}\ni g\mapsto e^{tX}ge^{-tX}\in \mathbb{T}
\end{equation*}
defines an automorphism of $\mathbb{T}$. Thus there exists a map $\mathbb{R}\ni t\mapsto (\mult_1(t),\cdots,m_{N''}(t))\in\mathbb{Z}^{N''}$ such that 
\begin{equation*}
\mathrm{Inn}(e^{tX})\circ\iso(\zeta_\one,\cdots,\zeta_{N''})=\iso(\zeta_\one^{\mult_1(t)},\cdots,\zeta_{N''}^{m_{N''}(t)})
\end{equation*}
for all $(\zeta_\one,\cdots,\zeta_{N''})\in (S^1)^{N''}$. The Lie algebra of $(S^1)^{N''}$ is isomorphic to 
\begin{equation*}
(i\mathbb{R})^{N''}=\{(i\gamma_\one,\cdots,i\gamma_{N''}):\gamma_l\in \mathbb{R}\text{ for }l=1,\cdots,N''\},
\end{equation*}
and we have
\begin{equation*}\begin{split}
\mathrm{Ad}(e^{tX})\circ(\iso_*)_\unit(i\gamma_\one,\cdots, i\gamma_{N''})=(\iso_*)_\unit(im_\one(t)\gamma_\one,\cdots,im_{N''}(t)\gamma_{N''})
\end{split}\end{equation*}
for all $(i\gamma_\one,\cdots,i\gamma_{N''})\in (i\mathbb{R})^{N''}$, where $(F_*)_\unit:(i\mathbb{R})^{N''}\rightarrow\mathfrak{t}$ is the differential of $F$ at $\unit\in (S^1)^{N''}$. Since $\mathbb{Z}^{N''}$ is discrete, we have $\mathrm{ad}(X)|_{\mathfrak{t}}=0$. Thus it follows that $[X,\giE]=\{0\}$.
\end{proof}
Let $\gamma\in\{-1,-1/2,0,1/2,1\}$, and let $\mathfrak{g}(\gamma)\subset \mathfrak{g}$ be the subspace given by $\mathfrak{g}(\gamma)^\#=\mathfrak{f}(\gamma)$. Then the following equalities hold:
\begin{equation*}
\mathfrak{g}(\gamma)=\{X\in\mathfrak{g}; \mathrm{ad}(jE)X=-\gamma X\},
\end{equation*}
\begin{equation*}
\mathfrak{g}=\mathfrak{g}(-1)\oplus\mathfrak{g}(-1/2)\oplus\mathfrak{g}(0)\oplus\mathfrak{g}(1/2)\oplus\mathfrak{g}(1).
\end{equation*}
Note that
\begin{equation*}
\mathfrak{b}(1)=\mathfrak{g}(-1),\quad \mathfrak{b}(1/2)=\mathfrak{g}(-1/2),\quad \mathfrak{b}(0)\subset \mathfrak{g}(0).
\end{equation*}
Let $\mathfrak{n}, \mathfrak{n}'\subset \mathfrak{g}$ be the subalgebras given by $\mathfrak{n}^\#=\mathfrak{m}$ and $\mathfrak{n}'^\#=\mathfrak{m}'$. 
Then we have
\begin{equation*}
(\mathfrak{n}\cap\mathfrak{g})^\#=\mathfrak{m}\cap\mathfrak{f}=\{X+\psi(X);X\in\mathfrak{f}(1/2)\},
\end{equation*}
\begin{equation*}
(\mathfrak{n}'\cap\mathfrak{g})^\#=\mathfrak{m}'\cap\mathfrak{f}=\{X+\varphi(X);X\in\mathfrak{f}(1)\},
\end{equation*}
\begin{equation*}
\giE=(\giE\cap\mathfrak{g}(0))\oplus\mathfrak(\mathfrak{n}\cap\mathfrak{g})\oplus(\mathfrak{n}'\cap\mathfrak{g}).
\end{equation*}
From now on, for $X\in\mathfrak{g}$, let $X_\gamma$ denote the projection of $X$ on $\mathfrak{g}(\gamma)$.
\begin{proposition}\label{b0inv}
The subalgebra $\mathfrak{b}_{triv}\subset\mathfrak{b}$ is $\mathrm{ad}(jE)$-invariant. \end{proposition}
\begin{proof}
Let $X\in\mathfrak{b}_{triv}$. By Lemma \ref{IfXinmathf}, for $Y_0\in\giE\cap\mathfrak{g}(0)$, we have
\begin{equation*}\begin{split}
[Y_0,X]=[Y_0,X_{-1}]+[Y_0,X_{-1/2}]+[Y_0,X_{0}]=0
\end{split}\end{equation*}
with $[Y_0,X_{-1}]\in\mathfrak{g}(-1)$, $[Y_0,X_{-1/2}]\in\mathfrak{g}(-1/2)$, and $[Y_0,X_{0}]\in\mathfrak{g}(0)$. Clearly $[Y_0,X_{-1/2}]=0, [Y_0,X_{-1}]=0$. Thus
\begin{equation*}\begin{split}
\mathrm{ad}(Y_0)\mathrm{ad}(jE)(X)&=\mathrm{ad}(Y_0)\left(\frac{1}{2}X_{-1/2}+X_{-1}\right)=0.
\end{split}\end{equation*}
For $Y'=Y'_{-1}+Y'_{1}\in \mathfrak{n}'\cap\mathfrak{g}$, we have
\begin{equation*}
[Y',X]=[Y'_{-1},X_0]+[Y'_\one,X_{-1}]+[Y'_\one,X_{-1/2}]+[Y'_\one,X_0]=0
\end{equation*}
with $[Y'_{-1},X_0]\in\mathfrak{g}(-1), [Y'_\one,X_{-1}]\in\mathfrak{g}(0), [Y'_\one,X_{-1/2}]\in\mathfrak{g}(1/2)$, and $[Y'_\one,X_0]\in\mathfrak{g}(1)$. 
Clearly
\begin{equation*}
[Y'_\one,X_{-1}]=0, [Y'_\one,X_{-1/2}]=0,
\end{equation*}
and we see from Proposition \ref{Forasubalg2} that 
\begin{equation}\label{X12X10quad}
[Y_{1/2},X_{-1}]=0\quad(Y_{1/2}\in\mathfrak{g}(1/2)).
\end{equation}
We have
\begin{equation*}\begin{split}
\mathrm{ad}(Y')\mathrm{ad}(jE)(X)&=\mathrm{ad}(Y')\left(\frac{1}{2}X_{-1/2}+X_{-1}\right)\\&
=\left[Y'_{-1}+Y'_\one,\frac{1}{2}X_{-1/2}+X_{-1}\right]\\&
=\frac{1}{2}[Y'_\one,X_{-1/2}]+[Y'_\one,X_{-1}]=0.
\end{split}\end{equation*}
For $Y=Y_{-1/2}+Y_{1/2}\in\mathfrak{n}\cap\mathfrak{g}$, we have
\begin{equation*}\begin{split}
[&Y,X]\\&=[Y_{-1/2},X_{-1/2}]+[Y_{-1/2},X_0]+[Y_{1/2},X_{-1}]+[Y_{1/2},X_{-1/2}]+[Y_{1/2},X_0]=0
\end{split}\end{equation*}
with $[Y_{-1/2},X_{-1/2}]\in\mathfrak{g}(-1), [Y_{-1/2},X_0]+[Y_{1/2},X_{-1}]\in\mathfrak{g}(-1/2),[Y_{1/2},X_{-1/2}]\in\mathfrak{g}(0)$, and $[Y_{1/2},X_0]\in\mathfrak{g}(1/2)$. By (\ref{X12X10quad}), we have
\begin{equation*}\begin{split}
\mathrm{ad}(Y)\mathrm{ad}(jE)(X)&=\mathrm{ad}(Y)\left(\frac{1}{2}X_{-1/2}+X_{-1}\right)\\&
=\left[Y_{-1/2}+Y_{1/2},\frac{1}{2}X_{-1/2}+X_{-1}\right]\\&
=\frac{1}{2}[Y_{-1/2},X_{-1/2}]+\frac{1}{2}[Y_{1/2},X_{-1/2}]+[Y_{1/2},X_{-1}]=0.
\end{split}\end{equation*}
Thus for any $X\in\mathfrak{b}_{triv}$ and $W\in\giE$, we have
\begin{equation*}
\mathrm{ad}(W)\mathrm{ad}(jE)X=0.
\end{equation*}
This completes the proof.
\end{proof}

\begin{proposition}\label{addinv}
The subspace $\mathfrak{b}_{triv}^\perp\subset\mathfrak{b}$ is $\mathrm{ad}(jE)$-invariant.
\end{proposition}
\begin{proof}
By Remark \ref{Bydatrithe} (i), for $X\in\mathfrak{b}_{triv}^\perp$ and $Y\in\mathfrak{b}_{triv}$, we have
\begin{equation*}\begin{split}
\langle\mathrm{ad}(jE)X,Y\rangle=\langle X,\mathrm{ad}(jE)Y\rangle,
\end{split}\end{equation*}
which is equal to $0$ by Proposition \ref{b0inv}. This implies that the assertion holds.
\end{proof}

\subsection{Actions of the isotropy subgroup on holomorphic vector bundles}
\begin{lemma}\label{compa}
Let $G_0$ be a connected compact Lie group, and let $(\pi,\mathcal{V})$ be a finite-dimensional unitary representation of $G_0$. For $X\in\mathfrak{g}_0$ and $\zeta\in\mathbb{C}$, put 
\begin{equation*}
\mathcal{V}(X,\zeta)=\{v\in\mathcal{V}; d\pi(X)v=\zeta v\}.
\end{equation*}
If $\pi$ is irreducible and nontrivial, then 
\begin{equation*}
\mathcal{V}=\sum_{X\in\mathfrak{g}_0,\zeta\in\mathbb{C}\backslash\{0\}}\mathcal{V}(X,\zeta).
\end{equation*}
\end{lemma}
\begin{proof}
Let $\mathcal{V}_1=\sum_{X\in\mathfrak{g}_0,\zeta\in\mathbb{C}\backslash\{0\}}\mathcal{V}(X,\zeta)$, and let $\mathbb{T}$ be a maximal torus of $G_0$. If $\mathcal{V}_1=0$, then the character $\chi_\pi(g)=\mathrm{tr}\,\pi(g)\,(g\in G_0)$ satisfies $\chi_\pi|_\mathbb{T}=\dim \mathcal{V}$ identically. Two finite-dimensional representations of $G_0$ are equivalent if and only if their character are equal (see \cite[Part 2, Corollary 5.3.4]{wolf}), and any two maximal tori of $G_0$ are conjugate (see \cite[Corollary 4.35]{Knapp}), so that $\pi$ is trivial. This contradicts the assumption. Thus $\mathcal{V}_1\neq 0$. Let $v\in \mathcal{V}(X,\zeta)$. For $g\in G_0$, we have
\begin{equation*}\begin{split}
d\pi(\mathrm{Ad}(g)X)\pi(g)v&=\dt\pi(ge^{tX}g^{-1}g)v=\dt\pi(ge^{tX})v\\&=\pi(g)d\pi(X)v=\zeta\pi(g)v,
\end{split}\end{equation*}
so that $\pi(g)v\in \mathcal{V}(\mathrm{Ad}(g)X,\zeta)$.
Hence $\mathcal{V}_1$ is a $G_0$-invariant subspace of $\mathcal{V}$. Since $\pi$ is a finite-dimensional irreducible representation and $\mathcal{V}_1\neq 0$, we have $\mathcal{V}=\mathcal{V}_1$.
\end{proof}

Let $\theta:\mathfrak{g}_-\rightarrow\mathbb{C}$ be a complex representation of $\mathfrak{g}_-$, and let $\chi^\theta:B\rightarrow \mathbb{C}^\times$ be the representation of $B$ given by \eqref{chithetaex}.
\begin{theorem}\label{extension}
Suppose that $\theta(\mathfrak{k})=0$. Extend the representation $d\chi^\theta:
\mathfrak{b}\rightarrow\mathbb{C}$ of $\mathfrak{b}$ to a linear map $d\chi^\theta:
\mathfrak{g}\rightarrow\mathbb{C}$ by the zero-extension along with the decomposition $\mathfrak{g}=\mathfrak{b}\oplus\giE$. Then $d\chi^\theta:\mathfrak{g}\rightarrow \mathbb{C}$ defines a representation of $\mathfrak{g}$.
\end{theorem}
\begin{proof}
First we show that $\mathfrak{b}_{triv}^\perp=\sum_{W\in\giE,\zeta\in\mathbb{C}\backslash\{0\}}\mathfrak{b}(W,\zeta)$. Since every irreducible subrepresentation of $(\rho,\mathfrak{b}_{triv}^\perp)$ is nontrivial, we see from Lemma \ref{compa} that $\mathfrak{b}_{triv}^\perp\subset\sum_{W\in\giE,\zeta\in\mathbb{C}\backslash\{0\}}\mathfrak{b}(W,\zeta)$. Conversely for $W\in\giE$, $\zeta\in\mathbb{C}\backslash\{0\}$, $X\in\mathfrak{b}(W,\zeta)$, and $X'\in\mathfrak{b}_{triv}$, we have $\langle X,X'\rangle=0$. This shows that $\mathfrak{b}_{triv}^\perp\supset\sum_{W\in\giE,\zeta\in\mathbb{C}\backslash\{0\}}\mathfrak{b}(W,\zeta)$. Thus 
\begin{equation*}
\mathfrak{b}_{triv}^\perp=\sum_{W\in\giE, \zeta\in\mathbb{C}\backslash\{0\}}\mathfrak{b}(W,\zeta)=\sum_{W\in\giE, \gamma\in\mathbb{R}\backslash\{0\}}\mathfrak{b}(W,i\gamma). 
\end{equation*}
Let $\gamma\in\mathbb{R}\backslash\{0\}$, and let $W\in\mathfrak{k}$. Second we show that $\theta (X+ijX)=0$ for all $X\in \mathfrak{b}(W,i\gamma)$. Let $X\in\mathfrak{b}(W,i\gamma)$. Then
\begin{equation*}\begin{split}
0&=[\theta(W),\theta(X+ijX)]=\theta([W,X+ijX])=\theta([W,X]+i[W,jX])
\\&=\theta(d\rho(W)X+id\rho(W)(jX))=\theta(\gamma j X-\gamma i X)=-\gamma i\theta(X+ijX).
\end{split}\end{equation*}
This proves that $\theta(X+ijX)=0$, and hence 
\begin{equation*}
\theta(X+ijX)=0\quad(X\in\mathfrak{b}_{triv}^\perp).
\end{equation*}
Let $X\in\mathfrak{b}_{triv}^\perp$. Then we have $X_{-1/2}, X_0\in\mathfrak{b}_{triv}^\perp$ by Proposition \ref{addinv}. Thus
\begin{equation*}
d\chi^\theta(X)=\theta (\tau(X))=\theta((X_{-1/2}+ijX_{-1/2})/2+X_0+ijX_0)=0.
\end{equation*}
We see from the above equality that $d\chi^\theta([\mathfrak{b},\giE])=0$. Now let $X,X'\in\mathfrak{g}$, and write $X=Y+W,X'=Y'+W'$ with $Y,Y'\in\mathfrak{b}$ and $W,W'\in\mathfrak{k}$. Then we have
\begin{equation*}
d\chi^\theta([X,X'])=d\chi^\theta([Y,Y'])=[d\chi^\theta(Y),d\chi^\theta(Y')]=[d\chi^\theta(X),d\chi^\theta(X')].
\end{equation*}
This completes the proof.
\end{proof}

Let $M,M':G\times\mathcal{D}(\Omega,Q)\rightarrow \mathbb{C}^\times$ be holomorphic multipliers. Put $\theta=\theta_{M}, \theta'=\theta_{M'}$. 
\begin{theorem}\label{Supposethat}
Suppose that $M(\con,(iE,0))=M'(\con,(iE,0))$ for all $\con\in K$. Then $M_{\theta}(\con,(U,V))=M_{\theta'}(\con,(U,V))$ for all $\con\in \GiE$ and $(U,V)\in\mathcal{D}(\Omega,Q)$.
\end{theorem}
\begin{proof}
By Theorem \ref{extension}, the representation $d\chi^{\theta-\theta'}:\mathfrak{b}\rightarrow \mathbb{C}$ of $\mathfrak{b}$ extends to a representation of $\mathfrak{g}$. Let us use the same symbol $d\chi^{\theta-\theta'}$ to denote the extension of the representation. Let $\widetilde{G}$ be the universal covering group of $G$. We denote the covering homomorphism by $\tilde{p}:\widetilde{G}\rightarrow G$. Then we have $\widetilde{G}=\tilde{p}^{-1}(B)^{o}\tilde{p}^{-1}(\GiE)$. Since the map $\tilde{p}|_{\tilde{p}^{-1}(B)^o}:\tilde{p}^{-1}(B)^o\rightarrow B$ is bijective, we have $\tilde{p}^{-1}(B)^o\cap \tilde{p}^{-1}(K)=\{e\}$. Let $\chi':\tilde{G}\rightarrow \mathbb{C}^\times$ be the lifting of $d\chi^{\theta-\theta'}:\mathfrak{g}\rightarrow\mathbb{C}$ to a representation of $\widetilde{G}$. For $g\in G$, let $g'$ and $g''$ be elements of $\tilde{G}$ such that $\tilde{p}(g')=\tilde{p}(g'')=g$. Then 
\begin{equation*}
g'g''^{-1}\in \tilde{p}^{-1}({e})\subset \tilde{p}^{-1}(K). 
\end{equation*}
Thus it follows that $\chi'(g'g''^{-1})=1$. Hence $\chi'$ descends to $G$, i.e. there exists a representation $\chi:G\rightarrow \mathbb{C}^\times$ such that $\chi'=\chi\circ\tilde{p}$. Now $\chi$ defines a holomorphic multiplier $\chi:G\times\mathcal{D}(\Omega,Q)\rightarrow \mathbb{C}^\times$, and we have \begin{equation*}
M_{\theta}M_{\theta'}^{-1}\chi^{-1}(k,(iE,0))=1\quad(k\in K),
\end{equation*} 
and
\begin{equation*}
M_{\theta}M_{\theta'}^{-1}\chi^{-1}(b,(U,V))=1\quad(b\in B,(U,V)\in\mathcal{D}(\Omega,Q)).
\end{equation*}
By the same arguments in Lemma \ref{extendm}, we have
\begin{equation*}
M_{\theta}M_{\theta'}^{-1}\chi^{-1}(g,(U,V))=1\quad(g\in G,(U,V)\in\mathcal{D}(\Omega,Q)). 
\end{equation*}
This proves the result.
\end{proof}

\begin{corollary}\label{LetEandEbe}
Let $L$ and $L'$ be $G$-equivariant holomorphic line bundles over a bounded homogeneous domain $\mathcal{D}$. Suppose that the actions of $K$ on the fibers $L_p$ and $L'_p$ coincide. Then $L$ and $L'$ are isomorphic as $K$-equivariant holomorphic line bundles.
\end{corollary}

\subsection{Unitary equivalences and the actions of the isotropy subgroup}
We see one formula on the function $\Delta_{\xi,\xi'}$. 
Let $\theta$ and $\theta'$ be one-dimensional complex representations of $\mathfrak{g}_-$. Then we have
\begin{equation}\label{chii}\begin{split}
\Delta_{\xi,\xi'}\left(\tfrac{U(b)-\overline{U'(b)}}{i}-2Q(V(b),V'(b))\right)=\overline{\chi^{\theta-\theta'}(b)}\Delta_{\xi,\xi'}\left(\tfrac{U-\overline{U'}}{i}-2Q(V,V')\right)\\(b\in B,(U,V),(U',V')\in\mathcal{D}(\Omega,Q)).
\end{split}\end{equation} 
Let $M,M':G\times\mathcal{D}(\Omega,Q)\rightarrow \mathbb{C}^\times$ be holomorphic multipliers. 
\begin{proposition}\label{pretheorem}
Suppose that the representations $T_{M}$ and $T_{M'}$ have unitarizations $(T_M, \mathcal{H})$ and $(T_{M'},\mathcal{H}')$ and that they are equivalent as unitary representations of $B$. Then $(T_{M}, \mathcal{H})$ and $(T_{M'}, \mathcal{H}')$ are equivalent as unitary representations of $G$ if and only if 
\begin{equation*}
M(\con,(iE,0))=M'(\con,(iE,0))\quad(\con\in K).
\end{equation*}
\end{proposition}
\begin{proof}
First we show the `only if' part. Putting $(U,V)=(iE,0)$ in (\ref{diff}), we obtain $M(\con,(iE,0))=M'(\con,(iE,0))$ for all $\con\in \GiE$. Second we show the `if' part. Suppose that $M(\con,(iE,0))=M'(\con,(iE,0))$ for all $\con\in \GiE$. Let $\xi$ and $\xi'$ be the linear forms on $\mathfrak{g}$ given by \eqref{xiJKcdotpi}. Then (\ref{chii}) gives
\begin{equation*}\begin{split}
&\Delta_{\xi,\xi'}\left(\tfrac{U(b)-\overline{U(b)}}{i}-2Q(V(b),V(b))\right)
=\overline{\chi^{i\xi-i\xi'}(b)}\Delta_{\xi,\xi'}\left(\tfrac{U-\overline{U}}{i}-2Q(V,V)\right)
\\&=\overline{M_{i\xi}M_{-i\xi'}(b,(U,V))}\Delta_{\xi,\xi'}\left(\tfrac{U-\overline{U}}{i}-2Q(V,V)\right)\quad(b\in B, (U,V)\in\mathcal{D}).
\end{split}\end{equation*}
Then Lemma \ref{extendm} and Theorem \ref{Supposethat} show that
\begin{equation*}\begin{split}
\Delta_{\xi,\xi'}&\left(\tfrac{U(\con)-\overline{U(\con)}}{i}-2Q(V(\con),V(\con))\right)
\\&=\overline{M_{i\xi}M_{-i\xi'}(\con,(U,V))}\Delta_{\xi,\xi'}\left(\tfrac{U-\overline{U}}{i}-2Q(V,V)\right)
\\&=\Delta_{\xi,\xi'}\left(\tfrac{U-\overline{U}}{i}-2Q(V,V)\right)\quad(\con\in \GiE, (U,V)\in\mathcal{D}(\Omega,Q)).
\end{split}\end{equation*}
By the analytic continuation, we have
\begin{equation*}\begin{split}
\Delta_{\xi,\xi'}\left(\tfrac{U(\con)-\overline{U'(\con)}}{i}-2Q(V(\con),V'(\con))\right)=\Delta_{\xi,\xi'}\left(\tfrac{U-\overline{U'}}{i}-2Q(V,V')\right)\\\quad(\con\in \GiE,(U,V),(U',V')\in\mathcal{D}(\Omega,Q)).
\end{split}\end{equation*}
Putting $(U',V')=(iE,0)$ in the above equation, we obtain
\begin{equation*}
\Delta_{\xi,\xi'}\left(\frac{U(\con)-\overline{iE}}{i}\right)=\Delta_{\xi,\xi'}\left(\frac{U-\overline{iE}}{i}\right)\quad(\con\in \GiE,(U,V)\in\mathcal{D}(\Omega,Q)).
\end{equation*}
Thus we get the equation (\ref{diff}), and hence the unitary representations $(T_{M},\mathcal{H})$ and $(T_{M'},\mathcal{H}')$ of $G$ are equivalent by Proposition \ref{Thefollowingc}. The proof is complete.
\end{proof}

\begin{theorem}\label{main}
Let $\mult,\mult':G\times\mathcal{D}\rightarrow \mathbb{C}^\times$ be holomorphic multipliers. Suppose that there exist Hilbert spaces $\mathcal{H}$ and $\mathcal{H}'$ of holomorphic functions on $\mathcal{D}$ which give the unitarizations of $T_m$ and $T_{m'}$, respectively. Then the following two conditions are equivalent:
\begin{enumerate}
\item[$(\mathrm{i})$]
 $(T_{\mult},\mathcal{H})$ and $(T_{\mult'},\mathcal{H}')$ are equivalent as unitary representations of $G$.
\item[$(\mathrm{ii})$]
$(T_{\mult},\mathcal{H})$ and $(T_{\mult'},\mathcal{H}')$ are equivalent as unitary representations of $B$, and $\mult(\con,p)=\mult'(\con,p)$ for all $\con\in K$.
\end{enumerate}
\end{theorem}
\begin{proof}
We assume that $(T_{\mult},\mathcal{H})$ and $(T_{\mult'},\mathcal{H}')$ are equivalent as unitary representations of $B$. Let $M,M':G\times \mathcal{D}(\Omega, Q)\rightarrow \mathbb{C}^\times$ be holomorphic multipliers given by $M(g,(U,V))={\mult}(g,\mathcal{C}((U,V)))$, $M'(g,(U,V))={\mult'}(g,\mathcal{C}((U,V)))$. We see from Proposition \ref{pretheorem} that the unitarizations of $T_{M}$ and $T_{M'}$ are equivalent as unitary representations of $G$ if and only if $M(\con,(iE,0))=M'(\con,(iE,0))$ for all $\con\in \GiE$. The map 
\begin{equation*}
\mathcal{C}^*:\mathcal{O}(\mathcal{D})\ni f\mapsto f\circ \mathcal{C}\in\mathcal{O}(\mathcal{D}(\Omega,Q)) 
\end{equation*}
intertwines representations $T_{{\mult}}$ and $T_{M}$ of $G$ and also intertwines representations $T_{{\mult'}}$ and $T_{M'}$ of $G$. Thus \text{(i)} holds if and only if $(T_{M},\mathcal{C}^*(\mathcal{H}))$ and $(T_{M'},\mathcal{C}^*(\mathcal{H}'))$ are equivalent as unitary representations of $G$. 
Moreover, by Proposition \ref{pretheorem}, the representations $(T_{M},\mathcal{C}^*(\mathcal{H}))$ and $(T_{M'},\mathcal{C}^*(\mathcal{H}'))$ are equivalent as unitary representations of $G$ if and only if
\begin{equation*}\begin{split}
\mult(\con,p)=M(\con,(iE,0))=M'(\con,(iE,0))=\mult'(\con,p) \quad(\con\in K).
\end{split}\end{equation*}
Thus (i) and (ii) are equivalent.
\end{proof}

\section{Application to a certain bounded homogeneous domain}\label{Applicatio}
In this section, we see an application of Theorem \ref{main}. We consider the following domain:
\begin{equation*}
\mathcal{D}(\Omega_\one)=\left\{U=\left[\begin{array}{ccc}z_\one&0&z_4\\0&z_2&z_5\\z_4&z_5&z_3\end{array}\right]\in \mathrm{Sym}(3,\mathbb{C});\Im U\gg 0 \right\}.
\end{equation*}
Let $\mathcal{U}=\left\{U_0=\left[\begin{array}{ccc}x_\one&0&x_4\\0&x_2&x_5\\x_4&x_5&x_3\end{array}\right];x_1,\cdots ,x_5\in\mathbb{R}\right\}$, and let 
\begin{equation*}
\Omega_\one=\mathcal{U}\cap \mathcal{P}(3,\mathbb{R}),
\end{equation*}
where $\mathcal{P}(3,\mathbb{R})$ denotes the homogeneous convex cone consists of all $3$-by-$3$ real positive-definite symmetric matrices. 
The domain $\mathcal{D}(\Omega_\one)$ is a Siegel domain of tube type, i.e. $\mathcal{D}(\Omega_\one)=\mathcal{U}+i\,\Omega_\one$.
We see the description of the holomorphic automorphism group of $\mathcal{D}(\Omega_\one)$ which is determined by Geatti \cite{geatti 1987}. Let $y_\one,\cdots,y_5\in\mathbb{R}$ and let $y_\one,y_2,y_3>0$. Put \begin{equation*}
T_0=\left[\begin{array}{ccc}y_{1}&0&0\\0&y_{2}&0\\y_4&y_5&y_3\end{array}\right].
\end{equation*} 
Let $x_\one,\cdots, x_5\in\mathbb{R}$, and put
\begin{equation*}
U_0=\left[\begin{array}{ccc}x_\one&0&x_4\\0&x_2&x_5\\x_4&x_5&x_3\end{array}\right].
\end{equation*} 
Let
\begin{equation*}
gl_{T_0}:\mathcal{D}(\Omega_\one)\ni U\mapsto T_0U{}^tT_0\in\mathcal{D}(\Omega_\one),
\end{equation*}
\begin{equation*}
t_{U_0}:\mathcal{D}(\Omega_\one)\ni U\mapsto U+U_0\in\mathcal{D}(\Omega_\one),
\end{equation*}
and for $\vartheta,\gamma\in\mathbb{R}$, and $U\in\mathcal{D}(\Omega_\one)$, let $k_{\vartheta,\gamma}(U)$
\begin{equation*}\begin{split}
=\large\left[\begin{array}{ccc}
\frac{\sin \vartheta+z_\one\cos\vartheta}{\cos\vartheta-z_\one\sin\vartheta}&0&\frac{z_4}{\cos\vartheta-z_\one\sin\vartheta}\\
0&\frac{\sin\gamma+z_2\cos\gamma}{\cos\gamma-z_2\sin\gamma}&\frac{z_5}{\cos\gamma-z_2\sin\gamma}\\\frac{z_4}{\cos\vartheta-z_\one\sin\vartheta}&\frac{z_5}{\cos\gamma-z_2\sin\gamma}&z_3+\frac{\sin\vartheta z_4^2}{\cos\vartheta-z_\one\sin\vartheta}+\frac{\sin\gamma z_5^2}{\cos\gamma-z_2\sin\gamma}\end{array}\right].
\end{split}\end{equation*}
\begin{theorem}[Geatti, \cite{geatti 1987}]\label{Theidentit}
The group $G=\mathrm{Aut}_{hol}(\mathcal{D}(\Omega_\one))^o$ is generated by $gl_{T_0}$, $t_{U_0}$, and $k_{\vartheta,\gamma}$.
\end{theorem}

Put $\mathcal{T}=\left\{T_0;y_\one,\cdots,y_5\in\mathbb{R}, y_\one,y_2,y_3>0 \right\}$. Let $B=\langle gl_{T_0},t_{U_0}\rangle_{T_0\in\mathcal{T},U_0\in\mathcal{U}}$ be the subgroup of $G$ generated by $gl_{T_0}$ and $t_{U_0}$. Then $B$ acts on $\mathcal{D}(\Omega_\one)$ simply transitively and is an Iwasawa subgroup of $G$. We take $iI_3\in\mathcal{D}(\Omega_\one)$ as a reference point of $\mathcal{D}(\Omega_\one)$. By Theorem \ref{Theidentit}, we have $\langle k_{\vartheta,\gamma}\rangle_{\vartheta,\gamma\in\mathbb{R}}=K$. Let $j$ be the complex structure on $\mathfrak{b}$ defined in Section \ref{Normaljalg}. The following holomorphic vector fields on $\mathcal{D}(\Omega_\one)$ are given by the action of a one-parameter subgroup of $\langle dl_{T_0}\rangle_{T_0\in\mathcal{T}}$:
\begin{equation*}\begin{split}
 A_\one^\#&=\left[\begin{array}{ccc}{z_1} & 0 & z_4/2\\
0 & 0 & 0\\
z_4/2 & 0 & 0\end{array}\right],\quad A_2^\#=\left[\begin{array}{ccc}0 & 0 & 0\\
0 & {z_2} & z_5/2\\
0 & z_5/2 & 0\end{array}\right],\\
A_3^\#&=\left[\begin{array}{ccc}0 & 0 & z_4/2\\
0 & 0 & z_5/2\\
z_4/2 & z_5/2 & {z_3}\end{array}\right],\quad
A_{3,1}^\#=\left[\begin{array}{ccc}0 & 0 & {z_1}\\
0 & 0 & 0\\
{z_1} & 0 & 2 {z_4}\end{array}\right],\\
A_{3,2}^\#&=\left[\begin{array}{ccc}0 & 0 & 0\\
0 & 0 & {z_2}\\
0 & {z_2} & 2 {z_5}\end{array}\right]. 
\end{split}\end{equation*}
The following holomorphic vector fields on $\mathcal{D}(\Omega_\one)$ are given by the action of a one-parameter subgroup of $\langle t_{U_0}\rangle_{U_0\in\mathcal{U}}$:
\begin{equation*}\begin{split}
E_\one^\#&=\left[\begin{array}{ccc}1 & 0 & 0\\
0 & 0 & 0\\
0 & 0 & 0\end{array}\right],\quad
E_2^\#=\left[\begin{array}{ccc}0 & 0 & 0\\
0 & 1 & 0\\
0 & 0 & 0\end{array}\right],\quad
E_3^\#=\left[\begin{array}{ccc}0 & 0 & 0\\
0 & 0 & 0\\
0 & 0 & 1\end{array}\right],\\
E_{3,1}^\#&=\left[\begin{array}{ccc}0 & 0 & 1\\
0 & 0 & 0\\
1 & 0 & 0\end{array}\right],\quad
E_{3,2}^\#=\left[\begin{array}{ccc}0 & 0 & 0\\
0 & 0 & 1\\
0 & 1 & 0\end{array}\right].
\end{split}\end{equation*}
The following holomorphic vector fields on $\mathcal{D}(\Omega_\one)$ are given by the action of a one-parameter subgroup of $\langle k_{\vartheta,\gamma}\rangle_{\vartheta, \gamma\in\mathbb{R}}$:
\begin{equation*}
W_\one^\#=\left[\begin{array}{ccc}-{{{z_1}}^{2}}-1 & 0 & -{z_1} {z_4}\\
0 & 0 & 0\\
-{z_1} {z_4} & 0 & -{{{z_4}}^{2}}\end{array} \right],\quad
W_2^\#=\left[\begin{array}{ccc}0 & 0 & 0\\
0 & -{{{z_2}}^{2}}-1 & -{z_2} {z_5}\\
0 & -{z_2} {z_5} & -{{{z_5}}^{2}}\end{array}\right]. 
\end{equation*}
Consider the bijection $\mathfrak{g}\ni X\mapsto X^\#\in\mathfrak{X}(\mathcal{D}(\Omega_\one))$,
and let $A_\one, A_2, A_3, A_{3,1}, A_{3,2}$ denote the elements of $\mathfrak{g}$ whose images under the above map are $A_\one^\#, A_2^\#, A_3^\#, A_{3,1}^\#, A_{3,2}^\#$, respectively, and set $E_\one, E_2, E_3, E_{3,1},E_{3,2}, W_\one, W_2$ in the same way. Then $r=3$, $\mathfrak{a}=\langle A_\one,A_2,A_3\rangle$,
\begin{equation*}
\mathfrak{b}_{(\alpha_l-\alpha_k)/2}=\langle A_{l,k}\rangle,\quad \mathfrak{b}_{(\alpha_l+\alpha_k)/2}=\langle E_{l,k}\rangle\quad(1\leq k< l\leq 3),
\end{equation*}
and
\begin{equation*}
\mathfrak{b}_{\alpha_k}=\langle E_k \rangle\quad(1\leq k\leq 3)
\end{equation*}
in Theorem \ref{Forasuitab}. We have
\begin{equation*}
\mathfrak{b}_-=\langle E_\one+iA_\one, E_2+iA_2, E_3+iA_3, E_{3,1}+iA_{3,1}, E_{3,2}+iA_{3,2} \rangle,
\end{equation*}
and since $[\mathfrak{b}_-,\mathfrak{b}_-]=\mathfrak{b}_-\cap[\mathfrak{b},\mathfrak{b}]_\mathbb{C}$, we have
\begin{equation*}
[\mathfrak{b}_-,\mathfrak{b}_-]=\mathfrak{b}_-\cap(\sideset{}{^\oplus}\sum_{1\leq k<l\leq r}\mathfrak{b}_{(\alpha_l-\alpha_k)/2}\oplus\sideset{}{^\oplus}\sum_{1\leq k<l\leq r}\mathfrak{b}_{(\alpha_l+\alpha_k)/2})_\mathbb{C}.
\end{equation*}
The subspace $[\mathfrak{k},\mathfrak{b}_-]$ is generated by the following elements: 
\begin{equation*}\begin{split}
&[W_\one,E_\one+iA_\one]=-2A_\one+i(2E_\one+W_\one), \,[W_\one,E_{3,1}+iA_{3,1}]=iE_{3,1}-A_{3,1},\\& [W_2,E_2+iA_2]=-2A_2+i(W_2+2E_2),\, [W_2,E_{3,2}+iA_{3,2}]=iE_{3,2}-A_{3,2}.
\end{split}\end{equation*}
Clearly, $[\mathfrak{k},\mathfrak{k}]=0$. 
Thus every $\xi\in\mathfrak{g}^*$ satisfying $\xi([\mathfrak{g}_-,\mathfrak{g}_-])=0$ can be written as
\begin{equation*}\begin{split}
\xi=\xi(x, y, n, n')=xE_{3}^*+yA_{3}^*+\frac{n}{2}(2W_\one^*-E_{1}^*)+\frac{n'}{2}(2W_2^*-E_{2}^*)\\\quad (x,y,n,n' \in\mathbb{R}).
\end{split}\end{equation*}
If the representation $i\xi|_{\mathfrak{k}}:\mathfrak{k}\rightarrow\mathbb{C}$ lifts to a representation of $K$, then
$n,n'\in\mathbb{Z}$. 

Let $x, y\in\mathbb{R}$ and let $n,n'\in \mathbb{Z}$. We shall apply Theorem 13 in \cite{ishi 2011} to the representation $T_{\chi^{i\xi}}$ with $\xi=\xi(x,y,n,n')$. The theorem gives the set of all parameters $(x,y,n,n')$ such that the representation $T_{\chi^{i\xi}}$ of $B$ is unitarizable, and defines the equivalence relation on the set which corresponds to the unitary equivalence among the unitarizable representations. We see from Theorem 13(i) in \cite{ishi 2011} that the representation $T_{\chi^{i\xi}}$ has a unitarization if
\begin{equation*}
x<0,\quad n>0,\quad n'>0
\end{equation*} 
or
\begin{equation*}
x=0,\quad n\geq 0,\quad n'\geq 0.
\end{equation*}
Put
\begin{equation*}
\Theta(G)=\left\{\theta\in\mathfrak{g}_-^*;\begin{array}{c}\theta \text{ is a one-dimensional representation of }\mathfrak{g}_-\text{ such that }\\ \text{ its restriction to } \mathfrak{k}\text{ lifts to a representation of }K\text{ and}\\\text{the representation }T_{M_\theta}\text{ of }G\text{ is unitarizable}\end{array} \right\}.
\end{equation*}
By Theorem \ref{LetmGtimes}, it follows that
\begin{equation*}\begin{split}
\Theta(G)=&\{i\xi(x, y, n, n') ;x<0, y\in\mathbb{R}, n,n'\in\mathbb{Z}_{>0} \}
\\&\bigsqcup\{i\xi(0, y, n, n'); y\in\mathbb{R}, n,n'\in\mathbb{Z}_{\geq0} \}. 
\end{split}\end{equation*}
We see from Theorem \ref{fundamentalone} and Theorem \ref{fundamentaltwo} that $\Theta(G)$ parametrizes the following set:
\begin{equation*}
\left\{[L];\begin{array}{c}L\text{ is a }G\text{-equivariant holomorphic line bundle over }\mathcal{D}(\Omega_\one)\text { such that}\\\text{ the representation }l \text{ of } G\text{ is unitarizable}\end{array}\right\},
\end{equation*}
where $[L]$ denotes the equivalence class of $L$ of $G$-equivariant holomorphic line bundles over $\mathcal{D}(\Omega_\one)$. Let
\begin{equation*}
\Theta_{B,-}=\{\xi(x, y, n, n'); x<0, y\in\mathbb{R}, n,n'\in\mathbb{Z}_{>0}\}
\end{equation*}
and
\begin{equation*}
\Theta_{B,0,y,n,n'}=\{\xi(0, y, n, n')\}\quad(y\in\mathbb{R}, n, n'\in\mathbb{Z}_{\geq 0}).
\end{equation*}
We see from Theorem 13(iii) in \cite{ishi 2011} that the partition of $\Theta(G)$ corresponding to the unitary equivalence classes of representations of $B$ is described as follows:
\begin{equation*}
\Theta(G)=\Theta_{B,-}\large{\bigsqcup}\bigsqcup_{y\in\mathbb{R}, n, n'\in\mathbb{Z}_{\geq 0}}\Theta_{B,0,y,n,n'}.
\end{equation*}
Let
\begin{equation*}\begin{split}
\Theta_{G,-,n,n'}=\{\xi(x, y, n, n'); x<0, y\in\mathbb{R}\}\quad(n, n'\in\mathbb{Z}_{>0})
\end{split}\end{equation*}
and
\begin{equation*}
\Theta_{G,0,n,n'}=\Theta_{B,0,n,n'} =\{\xi(0,y,n,n')\}\quad(y\in\mathbb{R},n,n'\in\mathbb{Z}_{\geq 0}).
\end{equation*}
By Theorem \ref{main}, it follows that the partition of $\Theta(G)$ corresponding to the unitary equivalence classes of representations of $G$ is described as follows:
\begin{equation*}
\Theta(G)=\bigsqcup_{n,n'\in\mathbb{Z}_{>0}}\Theta_{G,-,n,n'}\displaystyle\bigsqcup\bigsqcup_{y\in\mathbb{R}, n, n'\in\mathbb{Z}_{\geq 0}}\Theta_{G,0,y,n,n'}.
\end{equation*}

\section*{Acknowledgements}
The author would like to thank Professor H. Ishi for a lot of helpful advice on this paper. The author shows his greatest appreciation to Professor T. Uzawa for his insightful comments, and Professor M. Pevzner for valuable discussions.

\end{document}